\documentclass[11pt]{amsart}

\usepackage[dvipsnames]{xcolor}
\usepackage{enumitem}
\usepackage{amssymb,graphicx,mathtools}
\usepackage[margin=1in]{geometry}
\usepackage{etoolbox}

\usepackage[square]{natbib}

\usepackage[citecolor=PineGreen,colorlinks=true,linkcolor=RoyalBlue,urlcolor=BrickRed]{hyperref}

\theoremstyle{plain}
\newtheorem{theorem}{Theorem}[section]
\newtheorem{proposition}[theorem]{Proposition}
\newtheorem{corollary}[theorem]{Corollary}
\newtheorem{lemma}[theorem]{Lemma}
\theoremstyle{definition}
\newtheorem{remark}[theorem]{Remark}

\newcommand{\Z}{\mathbb{Z}}
\newcommand{\E}{\mathbb{E}}
\renewcommand{\P}{\mathbb{P}}
\DeclareMathOperator{\Var}{Var}
\DeclareMathOperator{\Cov}{Cov}
\newcommand{\one}{\mathbf{1}}
\newcommand{\eps}{\varepsilon}

\makeatletter
\patchcmd{\@settitle}{\uppercasenonmath\@title}{\Large}{}{}
\patchcmd{\@setauthors}{\MakeUppercase}{\large}{}{}
\makeatother

\begin{document}

\title{Sharpness and critical scaling of parking}

\author{Ahmed Bou-Rabee}
\author{Christoforos Panagiotis}

\begin{abstract}
In the parking model, each site of the $d$-dimensional lattice independently
starts with one car with probability $p$ or one parking spot with probability $1-p$. Cars move according to independent discrete-time simple random walks and park at the first spot they find free. We prove that in the critical regime $p=1/2$, the expected number of visits to a site in $n$ rounds is of order
$n^{(4-d)/4}$ for $d\leq3$ and $\log n$ for $d\geq4$. We also prove that in the subcritical regime $p\in(0,1/2)$, the parking-time tail is bounded above and below by stretched exponentials with exponent $d/(d+2)$. As $p\uparrow1/2$, we also determine the divergence of the
expected total number of visits to a site: its order is $(1-2p)^{-3}$,
$(1-2p)^{-1}$ and $(1-2p)^{-1/3}$ in dimensions one, two and three,
respectively, and $\log(1/(1-2p))$ in dimensions four and higher. Our proof uses a representation of the parking process as the divisible sandpile of \citet{LevinePeres09} plus a martingale-type term.
These results answer questions posed by \citet{DGJLS}.
\end{abstract}

\maketitle

{\small\tableofcontents}

\section{Introduction}\label{sec:intro}
At each site of $\Z^d$ place a random number of particles and a random number of
holes, independently from site to site. Write $\eta(x)$ for the number of particles at $x$ minus the number of holes there. The particles move according to independent simple random walks in discrete time: in each round, every particle that has not settled takes one step. A particle \textbf{settles} when it reaches a hole that is still unfilled. When $j$ particles arrive in the same round at a site with $h$ unfilled holes, choose $\min\{j,h\}$ of them uniformly at random to settle there.

This model generalizes both the parking model introduced by \citet{DGJLS}, where each site has either one parking spot or one car, and the particle--hole model \citep[Section~2]{CRS14}, where each site has one hole and a random nonnegative number of additional particles. The particle--hole model is also related to internal diffusion-limited aggregation with multiple sources and activated
random walk at infinite sleep rate \citep[Sections~1.1 and~10.3]{Rolla}.

The \textbf{odometer} $U_n(x)$ counts the steps taken from $x$ during the first $n$ rounds. As explained in Section~\ref{sec:model} below, it can be defined using the stack representation of \citet{DiaconisFulton}. Attach to each site an
independent sequence of uniformly chosen neighbors, which are the destinations of its
successive steps. Let $I_{y,x}(m)$ count the entries equal to $x$ among the first $m$ entries at $y$. For each neighbor $x$ of $y$, this is a sum of $m$ independent Bernoulli variables of mean $1/(2d)$. Then $U_0=0$ and
\begin{equation}\label{eq:parallel-intro}
	U_{n+1}(x)=\Bigl(\eta(x)+\sum_yI_{y,x}(U_n(y))\Bigr)^{\!+}\,.
\end{equation}

For i.i.d.\ initial configurations in the parking and particle--hole models, the two densities are equal exactly when $\E\eta(0)=0$. This is the critical point: when $\E\eta(0)<0$, every fixed site has finite odometer almost surely; when $\E\eta(0)=0$ and $\Var\eta(0)>0$, every site has infinite odometer almost surely \citep[Corollary~10.3 and Theorem~10.4]{Rolla}.
Figure~\ref{fig:parking} shows the two regimes.

\begin{figure}
	\centering
	\includegraphics[width=0.75\textwidth]{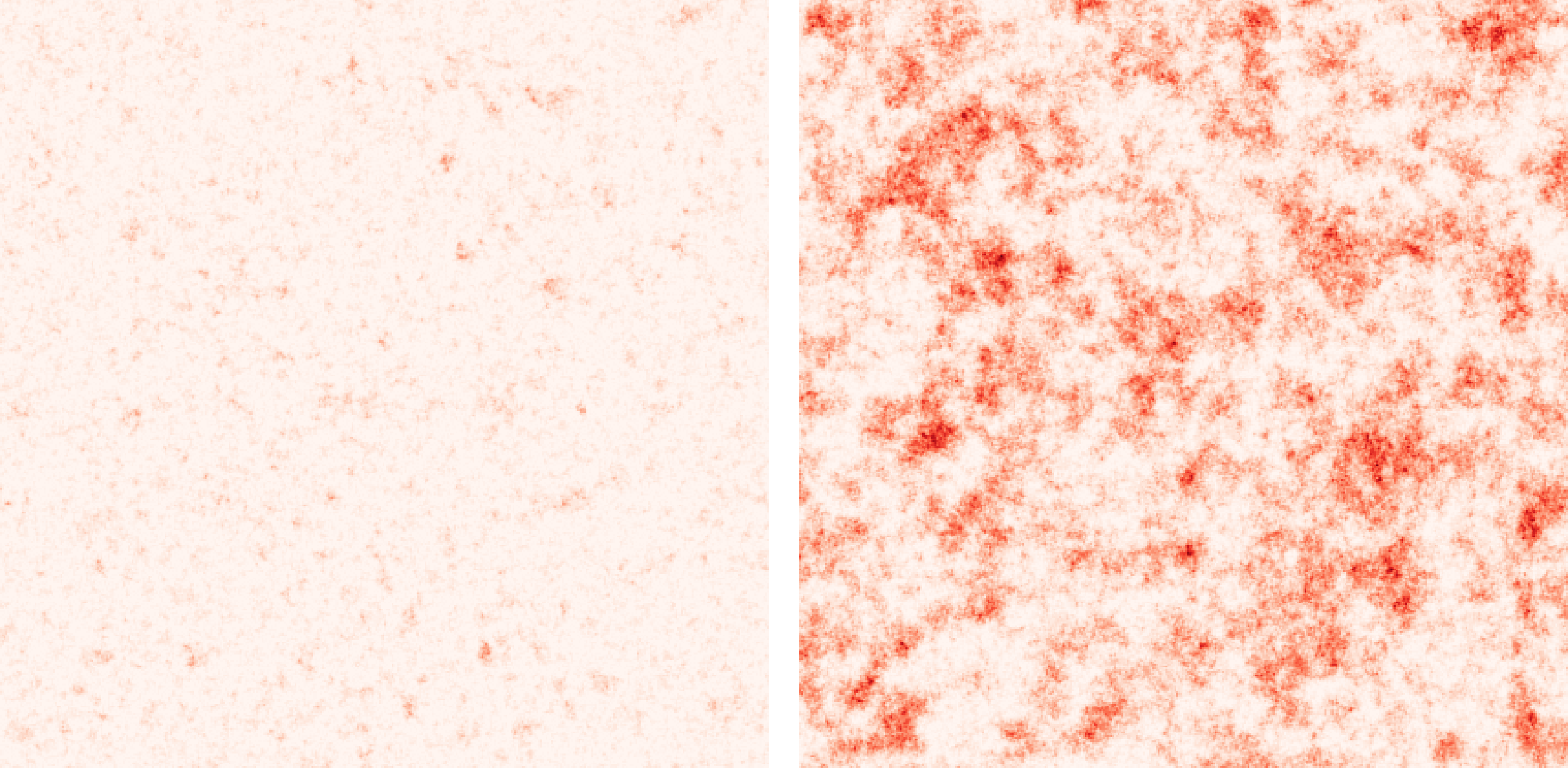}
	\caption{A snapshot of the odometer $U_{420}$ for the parking model on $\Z^2$. On the left, $\E\eta(0)=-0.16$ and on the right, $\E\eta(0)=0$. Darker shades
	indicate higher values, and the two panels share the same scale.}
	\label{fig:parking}
\end{figure}

In the special case of the parking model, \citet[Open Question~1]{DGJLS} asked
how fast $\E U_n(0)$ grows at the critical density, and conjectured asymptotics of order $n^{1/4}$ and $n^{1/2}$ for the oriented and unoriented walks on $\Z^2$. They also asked
whether the total number of visits has finite expectation below the critical density, and how that expectation diverges as the density rises to its critical value \citep[Open Question~2]{DGJLS}. We prove both conjectured asymptotics, show that the expectation is finite below the critical density, and determine its order of divergence on $\Z^d$.

Our main tool is a comparison to a deterministic analogue of the update rule in
\eqref{eq:parallel-intro}, the \textbf{divisible sandpile} of
\citet{LevinePeres09}, defined by replacing $I_{y,x}(m)$ there with its expected value $P(y,x)m$, where $P$ is the transition matrix of simple random walk. Specifically, the divisible sandpile is defined by starting with a signed mass $\eta(x)$ at each site $x \in \Z^d$. In each round, every site holding positive mass \textbf{topples}, dividing that mass equally among its $2d$ neighbors, while a site of negative mass absorbs what arrives until its mass turns positive. Writing $u_n(x)$ for the total mass emitted from $x$ during the first $n$ rounds,
\begin{equation}\label{eq:divisible-intro}
	u_0=0\,,\qquad u_{n+1}=(\eta+Pu_n)^+\,.
\end{equation}
We call $u_n$ the \textbf{divisible sandpile odometer} and $U_n$ the \textbf{particle odometer}.

The divisible sandpile odometer solves an optimal stopping problem for the walk. For every configuration $\eta$, consider a walk started at $x$ that collects reward $\eta(y)$ whenever it visits $y$, and stop no later than round $n$. Then, the largest expected total is equal to $u_n(x)$. Writing $X$ for the walk started at $x$ and $\E_x$ for the average over it alone,
\begin{equation}\label{eq:stopping-intro}
	u_n(x)=\sup_{\sigma\leq n}\E_x\sum_{j=0}^{\sigma-1}\eta(X_j)\,,
\end{equation}
the supremum being over stopping times for the natural filtration of the walk, bounded by $n$ and including $\sigma=0$; the stopping rule may depend on $\eta$. The representation is
classical \citep[Chapter~I]{PeskirShiryaev} and was observed in the divisible sandpile setting by \citet{BPRS}. It relates the divisible sandpile to a random walk in random scenery. All quantitative estimates on $u_n$ below start from
\eqref{eq:stopping-intro}.

\subsection{Subcritical sharpness}\label{ssec:results-sub}

Tag a particle at the origin, and let $R_t$ be the set of sites its walk has visited by time $t$. By independence, the event that no site of $R_t$ other than the origin has a hole at time zero has probability $\P(\eta(0)\geq0)^{|R_t|-1}$, and on that event the particle has not settled by time $t$.

Write $S_t$ for the expected number of particles which start at the origin and have not settled by time $t$. The origin itself holds a particle with probability $\P(\eta(0)>0)$, so averaging over the walk started there,
\[
	S_t\geq\P(\eta(0)>0)\,\E_0\bigl[\P(\eta(0)\geq0)^{|R_t|-1}\bigr]\,.
\]
By the Donsker--Varadhan estimate for the range
\citep[Theorem~1]{DonskerVaradhan}, the logarithm of the expectation on the
right is asymptotic to $-kt^{d/(d+2)}$ for some $k>0$. We prove an upper bound
with the same exponent at every density below the critical one.

\begin{theorem}[Stretched-exponential settling time]\label{thm:subcritical-tail}
	Let $\eta=(\eta(x))_{x\in\Z^d}$ have i.i.d.\ integer-valued coordinates,
	with
	\[
		\E|\eta(0)|<\infty\,,\qquad \E\eta(0)<0\,,\qquad \P(\eta(0)>0)>0\,,
	\]
	and suppose that $\E e^{\theta\eta(0)}<\infty$ for some $\theta>0$. Then there are
	$0<c\leq C<\infty$ such that, for every $t\geq1$,
	\begin{equation}\label{eq:sharpness}
		\frac1C\exp\bigl\{-Ct^{d/(d+2)}\bigr\}\leq S_t\leq
		C\exp\bigl\{-ct^{d/(d+2)}\bigr\}\,.
	\end{equation}
\end{theorem}

\citet{DLS} proved the two bounds in \eqref{eq:sharpness} for the parking model, in every dimension
for $p$ small and on $\Z$ for every $0<p<1/2$; Theorem~\ref{thm:subcritical-tail}
removes both restrictions. In that model Theorem~\ref{thm:subcritical-tail} also implies that the expected parking time is finite for every $0<p<1/2$. This addresses the first half of
\citet[Open Question~2]{DGJLS}, and Theorem~\ref{thm:near} below addresses the
second. In Remark~\ref{rem:stabilization-mixing} below, we discuss how Theorem~\ref{thm:subcritical-tail} also implies a mixing property of the final
configuration.

\subsection{Critical growth}\label{ssec:results-crit}

Suppose now that $\eta(0)$ is nonconstant with mean zero. If
$\E e^{\theta|\eta(0)|}<\infty$ for some $\theta>0$, then \citet{BP} show that
\begin{equation}\label{eq:bp}
	\E u_n(0)\asymp
	\begin{cases}
		n^{(4-d)/4}&d\leq3\,,\\
		\log n&d=4\,,\\
		(\log n)^{2/d}&d\geq5\,,
	\end{cases}
\end{equation}
the third case requiring also that $\eta(0)$ be bounded below. We show that $\E U_n(0)$ is comparable to $\E u_n(0)+\log n$. In dimensions at most three, \citet{BP} also prove locally uniform convergence of the rescaled odometer to a continuous random field defined by a Brownian optimal-stopping problem. Proposition~\ref{prop:spatial-scaling} transfers that scaling limit to the particle
odometer.

\begin{theorem}[Comparison of the two expected odometers]\label{thm:master}
	Let $\eta=(\eta(x))_{x\in\Z^d}$ have i.i.d.\ integer-valued coordinates.
	Suppose that $\eta(0)$ is nonconstant, $\E\eta(0)=0$, and
	$\E e^{\theta|\eta(0)|}<\infty$ for some $\theta>0$. Then there are
	$0<c\leq C<\infty$ such that, for every $n\geq2$,
	\begin{equation}\label{eq:master}
		c\bigl(\E u_n(0)+\log n\bigr)\leq\E U_n(0)
		\leq C\bigl(\E u_n(0)+\log n\bigr)\,.
	\end{equation}
\end{theorem}

Only the upper bound in \eqref{eq:master} uses the exponential moment. For the
lower bound, it is enough that the coordinates are i.i.d.\ and integer-valued,
with $\eta(0)$ nonconstant and of mean zero.

\begin{corollary}[Growth at the critical density]\label{cor:growth}
	Under the assumptions of Theorem~\ref{thm:master},
	\begin{equation}\label{eq:growth}
		\E U_n(0)\asymp
		\begin{cases}
			n^{(4-d)/4}&d\leq3\,,\\
			\log n&d\geq4\,.
		\end{cases}
	\end{equation}
	When $d\leq3$, we have $S_t\asymp (t+1)^{-d/4}$ for every $t\geq0$.
	When $d\geq4$, there are $0<c\leq C<\infty$ such that, for every $t\geq0$,
	\begin{equation}\label{eq:activity}
		\frac{c}{t+1}\leq S_t\leq\frac{C\log(t+2)}{t+1}\,.
	\end{equation}
\end{corollary}

\citet[Open Question~1]{DGJLS} asked how fast $\E U_n(0)$ grows at the critical density, and whether on $\Z^2$ it is asymptotic to a constant multiple of $n^{1/2}$. Corollary~\ref{cor:growth} gives the order in every dimension, and Theorem~\ref{thm:trichotomy} gives the asymptotic on $\Z$, $\Z^2$ and $\Z^3$.

Corollary~\ref{cor:growth} removes a logarithm from the bounds of \citet{PRS} on
$\Z$ and from the lower bound of \citet{JJLS} on $\Z^2$ and $\Z^3$, where no
matching upper bound is proved. From dimension four on, the upper bound in
\eqref{eq:activity} is new on the lattice. The lower bound with the correct
power is due to \citet{CRS} and is reproved here in discrete time. \citet{BL}
obtained the power $t^{-1}$ from dimension four on for a system in which the
holes also move; the upper bound in \eqref{eq:activity} carries an extra
logarithm.

Since $S_t\to0$, every particle settles after finitely many steps and every hole is eventually filled. However, as we show in
Proposition~\ref{prop:everyone-settles} below, the odometer explodes: $U_\infty(x)$ is infinite at every site. No single particle is responsible, because each of them takes finitely many steps. The steps taken from a site are instead contributed by infinitely many different particles, arriving from ever further away. For the parking model, \citet[Theorems~2.1
and~2.2]{DGJLS} proved that every car parks, every spot is filled, and every site has infinite odometer. For the particle--hole model, \citet[Theorem~6]{CRS14} proved that every site has infinite odometer.

\subsection{Quenched odometer comparison}\label{ssec:results-compare}
By \eqref{eq:bp}, the divisible sandpile term on the right of \eqref{eq:master} dominates the logarithm when $d\leq3$, matches it when $d=4$, and is dominated by it when $d\geq5$ and $\eta(0)$ is bounded below.
\begin{theorem}[Quenched odometer comparison]\label{thm:trichotomy}
	Let $\eta=(\eta(x))_{x\in\Z^d}$ have i.i.d.\ integer-valued coordinates.
	Suppose that $\eta(0)$ is nonconstant, $\E\eta(0)=0$, and
	$\E e^{\theta|\eta(0)|}<\infty$ for some $\theta>0$.
	All limits below are as $n\to\infty$.
	\begin{enumerate}[label=\textup{(\roman*)}]
		\item \underline{\textup{Dimensions one, two, and three.}}  The two
		odometers agree to leading order: almost surely and in $L^r$ for every
		$r\geq1$,
		\[
			\frac{U_n(0)-u_n(0)}{\E u_n(0)}\longrightarrow0\,.
		\]
		Consequently $\E U_n(0)/\E u_n(0)\to1$, and $n^{-(4-d)/4}\E U_n(0)$
		converges to a limit in $(0,\infty)$.

		\item \underline{\textup{Dimension four.}}  For every $n\geq1$, the
		ratio $\E U_n(0)/\E u_n(0)$ is at least one. For each fixed law this
		ratio is bounded above uniformly in $n$, but no such upper bound is
		uniform over the laws in this theorem: for every $B>0$ there is a law
		satisfying its hypotheses with
		\begin{equation}\label{eq:four-ratio}
			\liminf_{n\to\infty}\frac{\E U_n(0)}{\E u_n(0)}\geq B\,.
		\end{equation}

		\item \underline{\textup{Dimensions five and higher.}}  If in addition
		$\eta(0)$ is bounded below, then $\E U_n(0)\asymp\log n$ and
		$\E|U_n(0)-u_n(0)|\asymp\log n$, while
		\[
			\frac{\E U_n(0)}{\E u_n(0)}\longrightarrow\infty\,.
		\]
	\end{enumerate}
\end{theorem}

As we show in the proof, the law in part (ii) can be taken from the
particle--hole model, so \eqref{eq:four-ratio} already occurs there.

The conclusion $\E U_n(0)/\E u_n(0)\to\infty$ in part~\textup{(iii)} can fail when $\eta(0)$ is not bounded below.
\citet{BP} show that $\E u_n(0)\asymp(\log n)^{1/\gamma}$ when
$-\log\P(\eta(0)\leq-s)\asymp s^\gamma$ with $1\leq\gamma<d/2$. An exponential lower tail is the case $\gamma=1$, and it is compatible with the exponential moment assumed above. That tail gives $\E u_n(0)\asymp\log n$, so \eqref{eq:master} implies that $\E U_n(0)\asymp\log n$ as well and the ratio stays bounded.

\subsection{Nearest particle and hole}\label{ssec:results-nearest}

Color each site according to which is closer: the nearest car still
driving or the nearest parking spot still free. Figure~\ref{fig:segregation} suggests the appearance of larger monochromatic regions at later times. \citet[Open Question~3]{DGJLS} asked whether the origin therefore ends up closer to a car:
\[
	\P(\text{the origin is closer to a free spot than to a driving
	car at time } t)\longrightarrow0\,?
\]
For the general particle--hole model considered here, we answer the analogous
question affirmatively when $d\leq3$ and construct counterexamples when $d\geq5$.

\begin{figure}[!ht]
	\centering
	\includegraphics[width=0.88\textwidth]{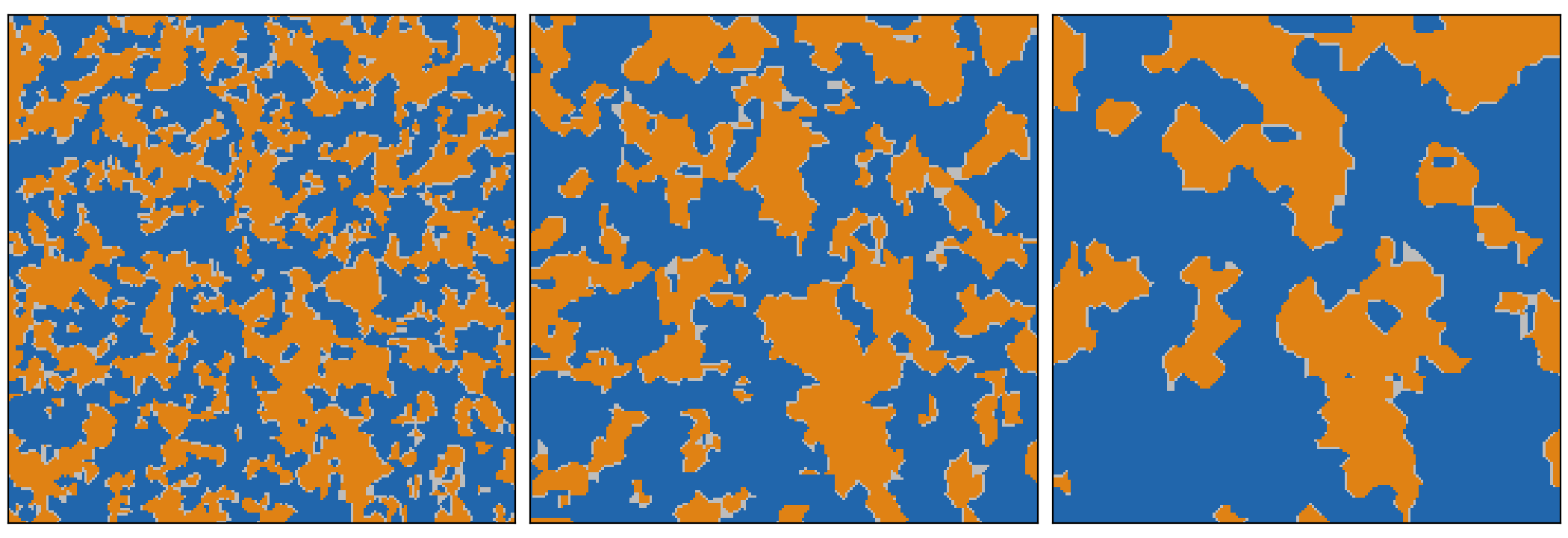}
	\caption{The parking model on $\Z^2$ at the critical density, on a
	$200\times200$ window at times $50$, $200$ and $800$. Sites are colored blue when the nearest driving car is closer, orange when the nearest free spot is closer, and gray in case of a tie.}
	\label{fig:segregation}
\end{figure}

\begin{theorem}\label{thm:nearest}
	Let $d\leq3$, and let $(\eta(x))_{x\in\Z^d}$ be independent and identically
	distributed, integer-valued and nonconstant, with $\E\eta(0)=0$ and
	$\E e^{\theta|\eta(0)|}<\infty$ for some $\theta>0$. Then
	\[
		\P(\text{the origin is closer to an unfilled hole than to an active
		particle at time } t)\longrightarrow0\,,
	\]
	as $t\to\infty$.
\end{theorem}

In fact, for each $K\geq1$, the probability that, at time $t$, the nearest
unfilled hole is at least $K$ times as far from the origin as the nearest active
particle tends to one as $t\to\infty$.

The proof uses the diffusive scaling limit at time one. Positivity of the
continuum odometer on a ball implies that the corresponding rescaled lattice
ball contains no unfilled hole. Strict positivity of its time derivative on a
smaller ball forces that ball to contain an active particle.

From dimension five on, the following theorem gives a counterexample with a
bounded, nonconstant, mean-zero initial law.

\begin{theorem}\label{thm:nearest-counterexample}
	For every $d\geq5$, there exist $p_d\in(0,1/2)$ and $c>0$ such that, if
	$(\eta(x))_{x\in\Z^d}$ are i.i.d.\ with
	\begin{equation}\label{eq:nearest-counterexample-law}
		\P(\eta(0)=1)=\P(\eta(0)=-1)=p_d,\qquad
		\P(\eta(0)=0)=1-2p_d\,,
	\end{equation}
	then
	     \begin{equation}\label{eq:nearest-counterexample}
		\liminf_{t\to\infty}
		\P(\text{the origin is closer to an unfilled hole than to an active
		particle at time }t)\geq c\,.
	\end{equation}
\end{theorem}
Simulations suggest that the counterexample in
Theorem~\ref{thm:nearest-counterexample} extends to dimension four.

\subsection{Near criticality}\label{ssec:results-near}

Theorem~\ref{thm:subcritical-tail} shows that the expected limiting odometer is finite
below the critical density, while Corollary~\ref{cor:growth} shows that it is infinite
at the critical density. The next theorem gives the rate of divergence for a
family whose mean increases to zero.

\begin{theorem}[Near criticality]\label{thm:near}
	Let $\delta_0>0$. For each $0\leq\delta\leq\delta_0$, let
	$\eta_\delta=(\eta_\delta(x))_{x\in\Z^d}$ have independent and identically
	distributed integer-valued coordinates of mean $-\delta$. Suppose that
	$\eta_0(0)$ is nonconstant. Let $U^\delta$ be the particle odometer started
	from $\eta_\delta$.
	Suppose that, for some $\theta>0$, the expectations
	$\E e^{\theta|\eta_\delta(0)|}$ are bounded uniformly over
	$0\leq\delta\leq\delta_0$. Suppose also that, for a constant $K<\infty$ and every
	$0<\delta\leq\delta_0$, the variables $\eta_\delta(0)$ and $\eta_0(0)$ admit a coupling
	with $\E|\eta_\delta(0)-\eta_0(0)|\leq K\delta$. Then, as
	$\delta\downarrow0$,
	\begin{equation}\label{eq:near}
		\E U_\infty^\delta(0)\asymp
		\begin{cases}
			\delta^{-3}&d=1\,,\\
			\delta^{-1}&d=2\,,\\
			\delta^{-1/3}&d=3\,,\\
			\log(e/\delta)&d\geq4\,.
		\end{cases}
	\end{equation}
\end{theorem}

In the parking model, each site holds either one car, corresponding to $\eta_\delta(0)=1$, or one spot, corresponding to $\eta_\delta(0)=-1$. Thus, if $p$ denotes the density of cars, then $\delta=1-2p$. In this case, \eqref{eq:near} becomes
\[
	\E U_\infty^\delta(0)\asymp
	\begin{cases}
		(1-2p)^{-3}&d=1\,,\\
		(1-2p)^{-1}&d=2\,,\\
		(1-2p)^{-1/3}&d=3\,,\\
			\log\bigl(e/(1-2p)\bigr)&d\geq4\,,
	\end{cases}
\]
which is the order of the divergence asked about by
\citet[Open Question~2]{DGJLS}. \citet{JJLS} proved the upper bound on $\Z$, and the
lower bound for $d\leq3$ up to a logarithmic factor. The upper bound
	for $d\geq2$ is new.

\subsection{Oriented walk}\label{ssec:results-oriented}

Fix $d\geq2$ and let $e_1,\ldots,e_d$ be the standard basis vectors of $\Z^d$.
Replace the simple random walk in the model by the walk which steps from $x$ to
$x+e_i$ with probability $1/d$ for each $1\leq i\leq d$, and let $\vec U_n$ be
its particle odometer. Let $\vec P(x,x-e_i)=1/d$ and define
$\vec u_0=0$ and $\vec u_{n+1}=(\eta+\vec P\vec u_n)^+$. Thus $\vec P$ is the
transpose of the particles' transition kernel, as required by
Remark~\ref{rem:transpose}. Let $\vec S_t$ be the expected number of particles
from the origin not settled by time $t$.

\begin{theorem}[Oriented walk]\label{thm:oriented-walk}
	Let $\eta=(\eta(x))_{x\in\Z^d}$ have i.i.d.\ integer-valued coordinates.
	Suppose that $\eta(0)$ is nonconstant, $\E\eta(0)=0$, and
	$\E e^{\theta|\eta(0)|}<\infty$ for some $\theta>0$. Then
	\[
		\E \vec U_n(0)\asymp
		\begin{cases}
			n^{1/4}&d=2\,,\\
			\log n&d\geq3\,.
		\end{cases}
	\]
	When $d=2$, the order $n^{1/4}$ sharpens to an asymptotic: the ratio $\E \vec U_n(0)/\E \vec u_n(0)$ tends to one, and there is $\mu\in(0,\infty)$ such that
	\[
		\E \vec U_n(0)\sim\mu n^{1/4}
		\qquad\text{and}\qquad
		\vec S_t\sim\frac{\mu}{4}t^{-3/4}\,.
	\]
\end{theorem}

Together with Theorem~\ref{thm:trichotomy} on $\Z^2$, Theorem~\ref{thm:oriented-walk} proves both asymptotics conjectured in
\citet[Open Question~1]{DGJLS}. For every $T>0$, as $n\to\infty$, the variable
$n^{-1/4}\vec U_{\lfloor nT\rfloor}(0)$ converges in distribution to the
value of a Brownian optimal stopping problem. For $1\leq d\leq3$, the variables
$n^{-(4-d)/4}U_{\lfloor nT\rfloor}(0)$ have analogous limits.

\subsection{Outline and proof overview}\label{ssec:method}

Most of our main results rest on a comparison of the particle and divisible
sandpile odometers; Theorems~\ref{thm:subcritical-tail}
and~\ref{thm:nearest-counterexample} are proved without it. The estimates and scaling limits for
$u_n$ from \citet{BP} start from \eqref{eq:stopping-intro}, while Jensen's
inequality in \eqref{eq:parallel-intro} gives
$u_n\leq\E[U_n\mid\eta]$. We prove this bound in
Section~\ref{sec:divisible}.

Section~\ref{sec:routing} bounds $U_n$ above by introducing $w_n$ to record the
accumulated difference between the arrivals and their means:
\[
	w_0=0\,,\qquad
	w_{k+1}(x)=(Pw_k)(x)+\sum_y\bigl(I_{y,x}(U_k(y))-P(y,x)U_k(y)\bigr)\,.
\]
Then
\[
	U_{n+1}-w_{n+1}=\max\bigl\{\eta+P(U_n-w_n),\,-w_{n+1}\bigr\}\,,
	\qquad
	u_{n+1}=\max\{\eta+Pu_n,\,0\}\,.
\]
The first entries of the two maxima differ only through $U_n-w_n$ and $u_n$,
while their second entries are $-w_{n+1}$ and $0$. Since the maximum changes by at most the largest change in its entries, induction gives
\[
	|U_n(x)-u_n(x)|\leq2\,\E_x\max_{0\leq j\leq n}|w_{n-j}(X_j)|\,.
\]

Revealing each instruction when a particle first uses it allows us to express $w_n(0)$ as a sum
of bounded martingale differences. Exactly $U_{n-1}(y)$ instructions are read at
$y$, so the odometer itself controls the total variance, and a Bernstein
inequality gives, for every $r\geq2$,
\[
	\bigl(\E|w_n(0)|^r\bigr)^{1/r}\leq
	C\sqrt{r\kappa_d(n)\bigl(\E U_n(0)^r\bigr)^{1/r}}+Cr\,,
\]
where $\kappa_d(n)$ is defined to be $\sqrt n$ in dimension one, $\log(n+2)$ in dimension
two, and $1$ from dimension three on. Hence
\[
	\bigl(\E U_n(0)^r\bigr)^{1/r}\leq
	C\bigl(\bigl(\E u_n(0)^r\bigr)^{1/r}+r(n+1)^{2/r}\kappa_d(n)\bigr)\,.
\]

In Section~\ref{sec:critical} we prove a logarithmic lower bound on $\E
U_n(0)$ whose constant does not depend on the law of $\eta(0)$. The comparison with the divisible sandpile
gives only $\E U_n(0)\geq\E u_n(0)$, and the asymptotic behavior of $\E u_n(0)$ does not give such a bound: in
dimension four the constant in $\E u_n(0)\asymp\log n$ depends on the law of
$\eta(0)$, and from dimension five on $\E u_n(0)$ is of order $(\log n)^{2/d}$
when $\eta(0)$ is bounded below. Instead, we adapt the resampling argument of \citet{CRS} to discrete time. It
gives $S_t\geq1/(16t)$ for all sufficiently large $t$, and hence
$\E U_n(0)\geq c\log n-C$.

In Section~\ref{sec:bounds} we combine the two lower bounds with the upper
bound and prove Theorem~\ref{thm:master} and Corollary~\ref{cor:growth}. From
dimension four on we take $r=\lceil\log n\rceil$ in the moment bound on
$U_n(0)$, for which the factor $(n+1)^{2/r}$ remains bounded and the resulting error term is
of order $\log n$, the additive logarithm in \eqref{eq:master}; the rates
\eqref{eq:bp} then give Corollary~\ref{cor:growth}.

In Section~\ref{sec:four} we show that in dimension four, although both
odometers have order $\log n$, their ratio cannot be bounded uniformly over all
laws of $\eta(0)$. Since $\eta\mapsto u_n(0)$ is convex in each coordinate,
$\E u_n(0)$ increases when the i.i.d.\ one-site law is increased in convex
order. This gives $\E u_n(0)\leq C\log n/\log(e/\eps)$ when $\eta(0)$ takes the values
$1$ and $-1$ with probability $\eps/2$ each, while $\E U_n(0)\geq c\log n$ for every law;
the ratio is therefore at least $c\log(e/\eps)$.

Below dimension four the divisible odometer has a scaling limit, and the
discrepancy estimates transfer it to $U_n$. On the set where the limiting
odometer is positive, its time derivative solves the heat equation, because the
scenery does not depend on time and cancels in time differences. The strong
minimum principle then makes this derivative strictly positive there.
We prove Theorem~\ref{thm:nearest} this way in Section~\ref{sec:nearest}.
For Theorem~\ref{thm:nearest-counterexample}, a one-particle coupling gives a two-hole bound controlled by
\[
	\sum_{n\geq0}(n+1)\P_0(X_n=x)\,,
\]
which is at most $C(1+|x|)^{4-d}$ when $d\geq5$. A one-site comparison handles
bounded separation. Most holes are therefore isolated on the scale of their
mean spacing, and the mass transport principle gives that a positive fraction of sites is
closer to an isolated hole than to every active particle.

Below the critical density we argue directly, without the divisible sandpile.
Tag a particle at the origin and tilt the law of $\eta(0)$ to raise the density
of particles. The tagged particle, if it survives to time $t$, leaves no unfilled
hole in its range. Its survival indicator and the number of unfilled holes in
the range therefore have covariance equal to minus the product of their means.
Lemma~\ref{lem:product} bounds that covariance by three times the derivative of
the survival probability in the tilt parameter. Integration and the Donsker--Varadhan estimate give
\eqref{eq:sharpness}; see Section~\ref{sec:subcritical}.

At a distance $\delta$ below criticality, $S_t$ differs by at most $C\delta$
from its critical value. Summing until the critical value reaches order
$\delta$ gives the lower bound in \eqref{eq:near} from dimension four on; below
dimension four the lower bound comes from the divisible sandpile. For the upper bounds, the
tilt makes the tail summable beyond an explicit cutoff, while the stopping
argument before that cutoff reduces the estimate to optimizing
$CM^{(4-d)/4}-\delta M$ when $d\leq3$ and $C\log(M+2)-\delta M$ when $d\geq4$.
We carry this out in Section~\ref{sec:near}.

The oriented walk reaches each layer at a unique time. That property makes the
critical ambient dimension three rather than four, and it replaces
$\kappa_d(n)$ by $1$ in every dimension. In Section~\ref{sec:oriented} we chain over dyadic time
scales to bound $\E\vec u_n(0)$ by $Cn^{1/4}$ in dimension two. From dimension
three on we use that the arrivals at the origin are independent of $\eta(0)$, so
that $\E\vec U_{n+1}(0)-\E\vec U_n(0)\geq ce^{-C\E\vec U_n(0)}$, which
integrates to $\E\vec U_n(0)\geq c\log n$.

\subsection{Related work}\label{ssec:related}

\subsubsection{Annihilating systems}

In the annihilating system of \citet{BL}, particles of two kinds are placed on
$\Z^d$ according to independent Poisson configurations. Every particle
performs an independent continuous-time simple random walk at the same rate,
and two particles of opposite kind annihilate when they meet. Starting from
equal initial densities, they proved that the density of either kind at time
$t$ has order $t^{-d/4}$ for $d\leq4$ and $t^{-1}$ for $d\geq5$. Their proof
uses the equality of the two laws of motion. That equality fails here because
the holes never move, so their proof does not imply the matching exponents in
Corollary~\ref{cor:growth}.
\citet{BL01} describe the spatial segregation behind these rates.
\citet{Arratia} proved site recurrence when particles of a single kind
annihilate in pairs from an independent initial placement.

The decay $t^{-d/4}$ was predicted by \citet{OZ} and \citet{TW}, who traced it
to fluctuations in the initial densities. \citet{KR} gave a scaling theory for
unequal initial densities and biased motion. In the biased case they predict the
exponent $3/4$ in ambient dimension two and criticality in ambient dimension
three. Theorem~\ref{thm:oriented-walk} confirms the exponent and
Remark~\ref{rem:critical-dimension} the critical dimension. The
exponent $d/(d+2)$ of Theorem~\ref{thm:subcritical-tail} also governs a walk
among static traps \citep{BV}, where it comes from rare trap-free regions. See
\citet{ADS} for a survey. This tail is heavier than the one \citet{BL} obtain
when both kinds move, because the holes here do not move.

\citet{CRS} keep the annihilation rule but take the initial configuration
independent with a finite first moment and let the two kinds jump at different
rates. They proved that every site is visited infinitely often, and extended the Bramson--Lebowitz lower bound
on the density. Their site recurrence covers immobile holes except when the
holes outnumber the particles, in which case a positive density of holes is
never filled. \citet{AJLRS} also let particles of the same type coalesce, and again allow
two jump rates. With no cap on that coalescence they bound the expected
occupation time of the root below by a logarithm, at densities where an added
particle avoids annihilation with positive probability.

\citet{JJLS} work in continuous time and let the spots move at every speed up
to that of the cars, including zero. At the critical density, they proved that
the expected number of cars at a site at time $t$ is at least
$c(t\log t)^{-d/4}$ for $d\leq3$. This matches Corollary~\ref{cor:growth} up to
the logarithm. With the spots frozen, they proved upper and lower bounds for
the expected occupation time of a fixed site, at and near criticality, on $\Z$ and
on the bidirected regular tree. Their two bounds agree on the tree, and agree on
$\Z$ up to logarithmic factors. Near criticality, their lower bound holds on
$\Z^d$ for $d\leq3$ and has order $(1-2p)^{-(4-d)/d}$ divided by
$\log(1/(1-2p))$, so Theorem~\ref{thm:near} removes this logarithmic factor. Their matching
upper bound is on $\Z$ alone, and they record that finiteness for $d\geq2$ was
open.

\citet{BBJJ} proved that the occupation time of a set increases in increasing
convex order with the initial law, and deduced the analogous result for the
total particle lifespan in internal diffusion-limited aggregation.
Theorem~\ref{thm:four-sparse} applies this comparison to the divisible sandpile
odometer instead of the particle one.

\subsubsection{The parking model}

\citet{DGJLS} introduced the parking model on the lattice and proved the
almost sure statements recalled after Corollary~\ref{cor:growth} on every
transitive, unimodular and infinitely accessible graph. They also settled the
supercritical case: for $p>1/2$ a car parks with probability $(1-p)/p$, and the
number of visits to a site is almost surely infinite. On $\Z$, \citet{PRS}
proved that the expected number of rounds a car at the origin drives before it parks is finite for every $0<p<1/2$, and that at $p=1/2$ the expectation of that number truncated at $t$ lies between $ct^{3/4}(\log t)^{-1/4}$ and
$Ct^{3/4}$.

The parking model is related to activated random walk at infinite sleep rate.
At a finite sleep rate \citet{JR} proved that on $\Z$, below the critical
density, the total odometer has a stretched-exponential tail with
exponent $1/2$, and deduced that its mean is finite. \citet{KM} bound the
critical density as the sleep rate tends to zero and to infinity, and
\citet{LS} survey the conjectures at finite sleep rate.

Parking models have also been considered on trees, but in these models cars
move deterministically toward the root rather than by random walk. See
\citet{GP}, \citet{CH}, \citet{Contat}, \citet{ACCH} and
\citet{BBJ}.

\subsubsection{The divisible sandpile}

The divisible sandpile and its odometer are due to \citet{LevinePeres09}, the
optimal stopping representation to \citet{BPRS}, and the growth rates
\eqref{eq:bp} to \citet{BP}. \citet{LMPU} took the initial
masses independent and identically distributed and proved that the divisible sandpile
fails to stabilize at the critical density when they have positive finite
variance. On the torus, with the total mass balanced, there is no free
boundary: \citet{CHR} proved that the rescaled odometer then converges to a
bilaplacian Gaussian limit in every dimension, and \citet{CHR2} that heavy tails
replace the Gaussian limit by a stable random distribution. \citet{LP10} identified the scaling limit
of internal diffusion-limited aggregation with several sources.

\subsection{Notation}\label{ssec:conventions}

Let $\P_x$ and $\E_x$ be the law and expectation of a simple random walk started
at $x$ and independent of everything else; unless a subscript or a conditioning
is shown, $\E$ averages over all the randomness of the model. A subscript
$\infty$ denotes the pointwise limit in the number of rounds, and every such
limit below exists by monotonicity. Distances on $\Z^d$ are graph distances,
and we write $|x|$ for the distance from $x$ to the origin. For a walk $X$, let
$T_x\coloneqq\inf\{n\geq0:X_n=x\}$. Constants $c$ and $C$ are
positive and finite, may change from line to line, and depend only on the
dimension and the law of $\eta(0)$ unless stated otherwise. We write
$a\asymp b$ when $ca\leq b\leq Ca$ throughout the range of the parameter under
discussion.

\section*{Acknowledgments}
We thank Antal A. J\'arai for helpful discussions. We thank David Sivakoff for bringing this question to our attention. CP was supported by an EPSRC New Investigator Award (UKRI1019).

\medskip

We acknowledge ChatGPT 5.6 for help with the proof of
Lemma~\ref{lem:product}.

\section{Particle odometer}\label{sec:model}
In this section we describe the stack representation of the model. We attach
the randomness to the sites rather than to the particles, following
\citet{DiaconisFulton}. Each site carries an independent stack of
instructions, and a particle leaving a site reads the next unread instruction
there. An instruction is then independent of the event that it is ever read, so the
arrivals at a site may be averaged while $\eta$ is held fixed
(Lemma~\ref{lem:deferred}). Adding one particle to $\eta$ perturbs the state
by exactly one active particle or one unfilled hole, and never shortens the
journey of a particle already present (Lemmas~\ref{lem:one-particle}
and~\ref{lem:tagged-monotonicity}). Two conservation identities close the
section: the mass transport principle turns $\E U_n(0)$ into $\sum_{s<n}S_s$, and the
densities of active particles and of unfilled holes differ by $\E\eta(0)$ at
every time.

Let $P(x,y)\coloneqq\one\{x\sim y\}/(2d)$ act on functions by
$(Pf)(x)\coloneqq\sum_yP(x,y)f(y)$; it is symmetric and stochastic. Let the initial
configuration $\eta=(\eta(x))_{x\in\Z^d}$ have i.i.d.\ integer-valued
coordinates, and write $a^+\coloneqq\max\{a,0\}$ and
$a^-\coloneqq\max\{-a,0\}$. Cancelling
each particle against a hole at the same site leaves every site $x$ with
$\eta(x)^+$ particles or $\eta(x)^-$ holes, never both. The mean $\E\eta(0)$
vanishes at $p=1/2$ for the parking model and when the mean number of particles
per site is one in the particle--hole model.

Instead of following each particle along its own walk, we attach the randomness
to the sites. At each site $y$ take an independent stack of instructions
$\rho_1(y),\rho_2(y),\ldots$ with $\P(\rho_j(y)=x)=P(y,x)$, so that the $j$th
step from $y$ follows $\rho_j(y)$, and let
$I_{y,x}(m)\coloneqq\#\{1\leq j\leq m:\rho_j(y)=x\}$. A round has three stages: every active particle takes one step, each arrival receives an independent random
variable uniform on $[0,1]$, and at each site the arrivals fill the unfilled holes
there in increasing order of these variables until either the arrivals or the holes run out. Arrivals which do not fill a hole are active in the next round, and a particle which fills one never moves again.

The stack construction has the same law as the construction in which each particle is assigned its own independent walk and independent uniform variables. An instruction is never read until a particle is about to use it, so the unread instructions stay independent of everything that has happened.

\begin{lemma}\label{lem:parallel}
	Let $U_n(x)$ be the number of steps taken from $x$ during the first $n$
	rounds. Then $U_0=0$, and for every $n\geq0$ and every $x\in\Z^d$,
	\begin{equation}\label{eq:parallel}
		U_{n+1}(x)=\Bigl(\eta(x)+\sum_yI_{y,x}(U_n(y))\Bigr)^{\!+}\,.
	\end{equation}
\end{lemma}

\begin{proof}
	By the end of round $n$, exactly $\eta(x)^+ +\sum_yI_{y,x}(U_n(y))$ particles have started at $x$ or arrived there. Of these, the minimum of that number and $\eta(x)^-$ settle, and every other one takes a step from $x$ by the end of round $n+1$, so $U_{n+1}(x)$ equals the right side of \eqref{eq:parallel}.
\end{proof}

\begin{lemma}\label{lem:deferred}
	For every $n\geq0$, every $y$ and $x$ in $\Z^d$ and every $j\geq1$, the event
	$\{U_n(y)\geq j\}$ is measurable with respect to the initial configuration
	together with the instructions other than $\rho_j(y)$. Consequently
	\begin{equation}\label{eq:deferred}
		\E\bigl[\one\{U_n(y)\geq j\}\one\{\rho_j(y)=x\}\bigm|\eta\bigr]
		=P(y,x)\,\P\bigl(U_n(y)\geq j\bigm|\eta\bigr)\,.
	\end{equation}
\end{lemma}

\begin{proof}
	Fix every variable except $\rho_j(y)$ and compare two arbitrary values of
	that instruction. The evolutions agree until the $j$th departure from $y$,
	because neither has read $\rho_j(y)$ before that departure. Hence
	$\{U_n(y)\geq j\}$ has the same value for every possible value of
	$\rho_j(y)$. Since $\rho_j(y)$ is independent of all remaining variables,
	\eqref{eq:deferred} follows.
\end{proof}

Let $H_t(x)$ be the number of unfilled holes at $x$ after round $t$, and let
$A_t(x)\coloneqq U_{t+1}(x)-U_t(x)$ be the number of active particles there.
Then $A_t(x)H_t(x)=0$.

\begin{lemma}\label{lem:one-particle}
	Let $\eta$ and $\widetilde\eta$ agree at every site except one, where
	$\widetilde\eta$ exceeds $\eta$ by one, and couple the two processes by
	giving every particle they share the same walk and the same uniform
	variables. Then, for every $t\geq0$, either $\widetilde H_t=H_t$ and
	$\widetilde A_t-A_t$ is one at a single site and zero elsewhere, or
	$\widetilde A_t=A_t$ and $H_t-\widetilde H_t$ is one at a single site and
	zero elsewhere.
\end{lemma}

\begin{proof}
	Call the process started from $\eta$ the first and the process started from
	$\widetilde\eta$ the second, and induct over the rounds. Initially, if the
	site where the particle is added has an unfilled hole, then the first process
	has one additional unfilled hole; otherwise the second process has one
	additional active particle. Suppose the second
	process has one extra active particle and that particle reaches a site with
	$h$ unfilled holes together with $m$ common particles. If $m<h$, then all
	$m+1$ arrivals settle in the second process, leaving the first with one extra
	hole. If $m\geq h$, then both processes fill all $h$ holes, and the extra
	particle either stays active or displaces one common particle, which becomes
	the extra active particle. Suppose instead the first process has one extra
	unfilled hole, so the two have $h+1$ and $h$ holes there. If at most $h$
	common particles arrive, then the extra hole remains; if at least $h+1$
	arrive, then one common particle settles only in the first process and
	becomes the extra active particle of the second. The identity of the unmatched
	active particle changes exactly when the currently unmatched particle fills a
	hole that a common particle would otherwise fill.
\end{proof}

\begin{lemma}[Tagged-particle monotonicity]\label{lem:tagged-monotonicity}
	In the coupling of Lemma~\ref{lem:one-particle}, every particle present in
	both processes
	stays active in the process started from $\widetilde\eta$ at least as long as
	in the process started from $\eta$.
\end{lemma}

\begin{proof}
	Call the process started from $\eta$ the first and the process started from
	$\widetilde\eta$ the second, and induct over the rounds, using the same
	uniform variables for common arrivals. Every
	common particle active in the first system is then among the arrivals in the
	second, while at each site the second system has no more holes than the first.
	If such a
	particle does not settle in the first system, then its rank among the arrivals
	exceeds the number of holes there. Adding arrivals cannot lower its rank, and removing a
	hole cannot make it settle. It therefore remains active in the second system
	as well.
\end{proof}

Label the particles which start at the origin as
$1,\ldots,\eta(0)^+$, let $\tau_i$ be the round in which particle $i$ settles,
and let
\[
	S_t\coloneqq\E\bigl[\#\{1\leq i\leq\eta(0)^+:\tau_i>t\}\bigr]\,.
\]

\begin{lemma}\label{lem:transport}
	For every $t\geq0$ and every $n\geq0$,
	\begin{equation}\label{eq:transport}
		\E A_t(0)=S_t\,,
		\qquad
		\E U_n(0)=\sum_{s<n}S_s\,,
	\end{equation}
	and, for every $k\geq1$ with $\P(\eta(0)=k)>0$, the $k$ particles at the
	origin are exchangeable conditionally on $\eta(0)=k$, so that
	\begin{equation}\label{eq:S-expand}
		S_t=\sum_{k\geq1}k\,\P(\eta(0)=k)\,\P(\tau_1>t\mid\eta(0)=k)\,.
	\end{equation}
\end{lemma}

\begin{proof}
	Send each particle that is still active after round $t$ from its starting
	site to its position at that time. The transport is invariant under the
	translations of $\Z^d$, so the expected mass received at the origin, which is
	$\E A_t(0)$, equals the expected mass sent from the origin, which is $S_t$.
	Summing $\E A_s(0)=S_s$ over $s<n$ gives $\E U_n(0)=\sum_{s<n}S_s$, and
	conditioning the definition of $S_t$ on $\eta(0)$ gives \eqref{eq:S-expand}.
\end{proof}

Every arrival at the origin in round $t+1$ comes from a particle active at a
neighbor after round $t$. Translation invariance and $\sum_yP(y,0)=1$ therefore
make the expected number of arrivals in that round equal to $\E A_t(0)=S_t$.
Summing over $t<n$, the expected number of arrivals at the origin during the
first $n$ rounds is again $\E U_n(0)$.

\begin{lemma}\label{lem:activity-holes}
	Let $\eta$ be translation invariant with $\E|\eta(0)|<\infty$.
	Then, for every $t\geq0$,
	\begin{equation}\label{eq:activity-holes}
		\E A_t(0)-\E H_t(0)=\E\eta(0)\,.
	\end{equation}
\end{lemma}

\begin{proof}
	A particle active at the origin after round $t$ started within distance $t$
	of the origin. Thus $A_t(0)\leq\sum_{|y|\leq t}\eta(y)^+$, and every
	expectation in this proof is finite. If $j$ particles arrive at a site
	carrying $h$ unfilled holes,
	then the new counts are $(j-h)^+$ and $(h-j)^+$, whose difference is $j-h$.
	Writing $J_{t+1}(x)$ for the number of arrivals at $x$ in round $t+1$,
	\begin{equation}\label{eq:signed-recursion}
		A_{t+1}(x)-H_{t+1}(x)=J_{t+1}(x)-H_t(x)\,.
	\end{equation}
	Every arrival at $x$ in round $t+1$ comes from a particle active at a neighbor
	of $x$ after round $t$. Translation invariance therefore gives
	$\E J_{t+1}(0)=\sum_yP(y,0)\E A_t(y)=\E A_t(0)$. Taking expectations in
	\eqref{eq:signed-recursion} leaves $\E A_t(0)-\E H_t(0)$ unchanged
	from one round to the next. At $t=0$ it equals
	$\E\eta(0)^+-\E\eta(0)^-=\E\eta(0)$.
\end{proof}

Since every active particle moves in every round, $U_t(x)=\sum_{s<t}A_s(x)$,
and $\sum_yI_{y,x}(U_t(y))$ counts the arrivals at $x$ through round $t$.
Induction on \eqref{eq:signed-recursion} then gives the pathwise identity
\begin{equation}\label{eq:signed-count}
	A_t(x)-H_t(x)=\eta(x)+\sum_yI_{y,x}(U_t(y))-U_t(x)\,.
\end{equation}

When the mean vanishes, $\E A_t(0)=\E H_t(0)=S_t$.

\begin{lemma}\label{lem:density-compare}
	Let $\eta$ and $\widetilde\eta$ have finite first moments and admit a
	translation invariant coupling under which
	$\eta(x)\leq\widetilde\eta(x)$ for every $x$ almost surely. Let $S_t$ and
	$\widetilde S_t$ be the two expected numbers of particles which start at
	the origin and are still active after round $t$. Then, for every $t\geq0$,
	\[
		0\leq\widetilde S_t-S_t
		\leq\E\widetilde\eta(0)-\E\eta(0)\,.
	\]
\end{lemma}

\begin{proof}
	Lemma~\ref{lem:one-particle}, applied successively, gives
	$A_t\leq\widetilde A_t$ and $H_t\geq\widetilde H_t$ for finitely many added
	particles. For the stated coupling, $A_t(x)$ and $H_t(x)$ depend only on the
	initial configuration within distance $t$ of $x$. Thus
	$A_t(x)\leq\widetilde A_t(x)$ and
	$H_t(x)\geq\widetilde H_t(x)$ also hold for the full coupling.

	Taking expectations at the origin gives $S_t\leq\widetilde S_t$ by
	Lemma~\ref{lem:transport}. For the other inequality, apply
	\eqref{eq:activity-holes} to $\eta$ and to $\widetilde\eta$:
	\[
		S_t-\E\eta(0)=\E H_t(0)\geq\E\widetilde H_t(0)
		=\widetilde S_t-\E\widetilde\eta(0)\,. \qedhere
	\]
\end{proof}

For the parking model, assign an independent $\operatorname{Uniform}[0,1]$ variable to every site and place
a car where it falls below $p$; raising $p$ to $1/2$ only adds cars. Writing
$S_t^q$ for $S_t$ at car density $q$,
Lemma~\ref{lem:density-compare} gives, for every $t\geq0$,
\begin{equation}\label{eq:parking-compare}
	S_t^{1/2}-(1-2p)\leq S_t^p\leq S_t^{1/2}\,.
\end{equation}

\section{Divisible sandpile odometer and optimal stopping}\label{sec:divisible}

Replacing each arrival by its mean turns \eqref{eq:parallel} into the
deterministic recursion \eqref{eq:divisible}. Lemma~\ref{lem:stopping}
identifies this recursion with an optimal stopping problem. Because unused
instructions retain their original laws, Jensen's inequality gives $u_n\leq\E[U_n\mid\eta]$ in
Theorem~\ref{thm:comparison}, with no moment assumption. Retaining the
unfilled holes gives an exact recursion for $\E[U_n\mid\eta]$.

Replace the particles and holes by a divisible mass, so that $\eta(x)$ is the
signed mass at $x$. In each round, every site holding positive mass topples,
sending all of that mass in equal shares to its $2d$ neighbors. Let $u_n(x)$ denote the total mass emitted from $x$ during the first $n$
rounds. Then
\begin{equation}\label{eq:divisible}
    u_0=0\,,\qquad u_{n+1}=(\eta+Pu_n)^+\,.
\end{equation} 
This is the divisible sandpile of \citet{LevinePeres09}. In the normalization of \citet{BPRS}, the
odometer records mass emitted per neighbor and therefore equals $u_n/(2d)$ on
$\Z^d$. With the scenery $\eta$ fixed, run a walk from $x$, receive the reward
$\eta(X_j)$ at time $j$, and stop at a time of your choosing no later than $n$.
The largest expected total is $u_n(x)$: the
recursion \eqref{eq:divisible} is the dynamic programming equation of that
optimal stopping problem \citep[Chapter~I]{PeskirShiryaev}, as \citet{BPRS}
observed in the divisible sandpile setting. Multiplying Theorem~3.2 of \citet{BPRS} by $2d$ gives the
following.

\begin{lemma}[{\citealp{BPRS}}]\label{lem:stopping}
	Let $X$ be the simple random walk. For every configuration $\eta\in\mathbb R^{\Z^d}$, every $n\geq0$ and $x\in\Z^d$,
	\begin{equation}\label{eq:stopping}
		u_n(x)=\sup_{\sigma\leq n}\E_x\sum_{j=0}^{\sigma-1}\eta(X_j)\,,
	\end{equation}
	where the supremum is over stopping times for the
	natural filtration of $X$, bounded by $n$ and including $\sigma=0$; the
	stopping rule may depend on $\eta$.
\end{lemma}

Averaging the instructions while holding $\eta$ fixed already bounds $u_n$ above
by $\E[U_n\mid\eta]$.

\begin{theorem}[Conditional domination]\label{thm:comparison}
	Fix an integer-valued initial configuration $\eta$. For every $n\geq0$ and
	$x\in\Z^d$,
	\begin{equation}\label{eq:comparison}
		u_n(x)\leq\E[U_n(x)\mid\eta]\,.
	\end{equation}
\end{theorem}

\begin{proof}
	Let $v_n\coloneqq\E[U_n\mid\eta]$. Summing \eqref{eq:deferred} over $j$ gives
	$\E[I_{y,x}(U_n(y))\mid\eta]=P(y,x)v_n(y)$, so
	$\E[\sum_yI_{y,x}(U_n(y))\mid\eta]=\sum_yP(y,x)v_n(y)=(Pv_n)(x)$ by symmetry
	of $P$. The positive part is a convex function, so Jensen's inequality applied to
	\eqref{eq:parallel} gives $v_{n+1}\geq(\eta+Pv_n)^+$. The map $v\mapsto(\eta+Pv)^+$ is nondecreasing and $u_0=v_0=0$, so induction
	proves
	\eqref{eq:comparison}.
\end{proof}

Retaining the unfilled holes gives an identity in place of \eqref{eq:comparison}.
Using that $a^+=a+(-a)^+$ in \eqref{eq:parallel} and averaging over the instructions
alone gives
\[
	\E[U_{n+1}\mid\eta]=\eta+P\,\E[U_n\mid\eta]+\E[H_n\mid\eta]\,.
\]
The divisible recursion \eqref{eq:divisible} takes a positive part, whereas the
averaged particle recursion adds $\E[H_n\mid\eta]$.

\begin{remark}[Nonsymmetric walks]\label{rem:transpose}
	For a translation invariant kernel $P$, the incoming mean at $x$ is
	$(P^*v_n)(x)$, where $P^*(x,y)\coloneqq P(y,x)$. The proof of Theorem~\ref{thm:comparison} compares the particle odometer with
	\eqref{eq:divisible} driven by $P^*$; the stopping representation of that
	divisible sandpile uses the walk with kernel $P^*$.
\end{remark}

We next quote the estimates of \citet{BP}, multiplied by $2d$ to match our
normalization. The theorem below allows $\eta$ to be real-valued and not
only integer-valued. In Section~\ref{sec:four}, we use this generality to bound
$\E u_n(0)$ for an integer-valued $\eta$ taking nonzero values with small
probability, by comparison with independent Laplace variables.

\begin{theorem}[{\citealp[Theorem~1.3, Corollary~6.2, Theorem~6.6,
and (94)]{BP}}]
\label{thm:BP}
	Let $\eta=(\eta(x))_{x\in\Z^d}$ be i.i.d.\ real-valued with mean zero, positive finite
	variance, and $\E e^{\theta|\eta(0)|}<\infty$ for some $\theta>0$, and
	let $u_n$ obey $u_0=0$ and $u_{n+1}=(\eta+Pu_n)^+$. Then, for every $n\geq2$,
	$\E u_n(0)\asymp n^{(4-d)/4}$ for $d\leq3$ and $\E u_n(0)\asymp\log n$ for
	$d=4$. For $d\geq5$,
	\[
		c(\log n)^{2/d}\leq\E u_n(0)\leq C\log(n+1)\,.
	\]
	If moreover $\eta(0)$ is bounded below, then $\E u_n(0)\asymp(\log n)^{2/d}$
	for $d\geq5$. When $d\leq3$ the limit
	$\lim_{n\to\infty}n^{-(4-d)/4}\E u_n(0)$ exists and lies in $(0,\infty)$.
\end{theorem}

When $\eta(0)$ is not bounded below, the exact rate in dimensions at least five
depends on its lower tail. In the particle--hole model each site initially has
one hole, so $\eta\geq-1$ and the bounded-below statement applies.

\section{Pathwise comparison and moment bounds}\label{sec:routing}

Each instruction moves one particle to one neighbor, whereas
\eqref{eq:divisible} divides the corresponding mass among all neighbors. We
let $w$ record the accumulated difference between the actual arrivals and
their means. Lemma~\ref{lem:pathwise-comparison}
bounds the odometer difference by the largest value of $|w|$ seen by an
independent walk. Revealing each instruction when it is first used expresses
$w_n(0)$ as a sum of bounded martingale differences. The sum of their
variances is controlled by the odometer and by a truncated Green-function
estimate. The Green-function factor is at most $C\sqrt n$ in dimension one,
$C\log(n+2)$ in dimension two, and a constant from dimension three on.

More precisely, let $w_0\coloneqq0$ and
\begin{equation}\label{eq:error-recursion}
	w_{k+1}(x)=\frac1{2d}\sum_{y\sim x}w_k(y)
	+\sum_{y\sim x}\Bigl(I_{y,x}(U_k(y))-\frac1{2d}U_k(y)\Bigr)\,.
\end{equation}
The second sum counts the steps from the neighbors of $x$ which landed at $x$ in
the first $k$ rounds, minus the number their means would have sent. The first
term is one more application of $P$ to $w_k$.
Let $X$ be a simple random walk independent of $\eta$ and the instructions, and write
\[
	w_n^\star(x)\coloneqq\E_x\max_{0\leq j\leq n}|w_{n-j}(X_j)|\,,
\]
where $\E_x$ averages over the walk alone, so that $w_n^\star(x)$ is random.

\begin{lemma}\label{lem:pathwise-comparison}
	For every $n\geq0$ and every $x\in\Z^d$,
	\begin{equation}\label{eq:pathwise-error}
		|U_n(x)-w_n(x)-u_n(x)|\leq w_n^\star(x)\,,
	\end{equation}
	and therefore
	\begin{equation}\label{eq:pathwise-comparison}
		|U_n(x)-u_n(x)|\leq2\,w_n^\star(x)\,.
	\end{equation}
\end{lemma}

\begin{proof}
	Let $V_k\coloneqq U_k-w_k$ and $D_k\coloneqq|V_k-u_k|$. Lemma~\ref{lem:parallel} and
	\eqref{eq:error-recursion} give
	\[
		V_{k+1}=(\eta+PV_k)\vee(-w_{k+1}),
		\qquad
		u_{k+1}=(\eta+Pu_k)\vee0\,.
	\]
	Since
	\[
		|(a_1\vee a_2)-(b_1\vee b_2)|
		\leq\max\{|a_1-b_1|,|a_2-b_2|\}\,,
	\]
	we have, pointwise,
	\begin{equation}\label{eq:bellman-difference}
		D_{k+1}\leq\max\{PD_k,|w_{k+1}|\}\,.
	\end{equation}
	The definition of $w_k^\star$ and the Markov property give
	\[
		w_{k+1}^\star\geq|w_{k+1}|,
		\qquad
		w_{k+1}^\star\geq Pw_k^\star\,.
	\]
	Since $D_0=w_0^\star=0$, induction in
	\eqref{eq:bellman-difference} proves $D_n\leq w_n^\star$, which is
	\eqref{eq:pathwise-error}. Finally, we have
	$|w_n(x)|\leq w_n^\star(x)$ by the term $j=0$ in the maximum, and
	\eqref{eq:pathwise-comparison} follows.
\end{proof}

Let $g_m(x)\coloneqq\sum_{j<m}P^j(0,x)$ be the Green function of the walk truncated
at time $m$, so that $g_m(x)=0$ when $|x|\geq m$. An instruction at $y$ first
used in round $s$ and pointing at $z$ contributes
$g_{n-s}(z)-(Pg_{n-s})(y)$ to $w_n(0)$, as follows by iterating
\eqref{eq:error-recursion}. Let
\[
\Gamma_m(y)\coloneqq\sum_zP(y,z)\bigl(g_m(z)-(Pg_m)(y)\bigr)^2\,,
\]
the variance of $g_m$ at a uniform neighbor of $y$. Let
$\kappa_d(n)\coloneqq n^{1/2}$ when $d=1$,
$\kappa_d(n)\coloneqq\log(n+2)$ when $d=2$, and
$\kappa_d(n)\coloneqq1$ when $d\geq3$.
For every $m\geq1$, every $y$, and every neighbor $z$ of $y$,
\begin{equation}\label{eq:green-gradient}
	|g_m(y)-g_m(z)|\leq C(1+|y|)^{1-d}\,,
	\qquad\text{hence}\qquad
	\Gamma_m(y)\leq C(1+|y|)^{2-2d}\,.
\end{equation}
For $d\geq2$, pair the two parity classes in the sum defining $g_m$. The
first-difference local central limit estimate and the Gaussian bound
\citep[Section~2.3]{LawlerLimic} give, uniformly in $m$,
\[
	|g_m(y)-g_m(z)|
	\leq C\sum_{j\geq0}(j+1)^{-(d+1)/2}
	\exp\{-c|y|^2/(j+1)\}+C(1+|y|)^{-d}\,,
\]
where the last term bounds the single terminal summand left by the parity
pairing. Splitting the displayed sum at $j=|y|^2$ gives
\eqref{eq:green-gradient}. In dimension one the gradient is
bounded, as the proof of Lemma~\ref{lem:gamma-sum} computes exactly. We also use the increment bound
\begin{equation}\label{eq:increment}
	\sup_{m\geq1}\sup_y\sup_{z\sim y}|g_m(z)-(Pg_m)(y)|\leq C\,,
\end{equation}
which follows by writing
$g_m(z)-(Pg_m)(y)=(g_m(z)-g_m(y))+((I-P)g_m)(y)$: the first bracket is bounded
by \eqref{eq:green-gradient} and the second is
$\one\{y=0\}-P^m(0,y)$, of absolute value at most one. In dimensions one and
two the bound $(1+|y|)^{2-2d}$ is not summable, and the truncation at time $m$
has to be used.

\begin{lemma}\label{lem:gamma-sum}
	For every $n\geq1$,
	\begin{equation}\label{eq:gamma-sum}
		\sum_y\sup_{m\leq n}\Gamma_m(y)\leq C\kappa_d(n)\,.
	\end{equation}
\end{lemma}

\begin{proof}
	\smallskip
	\noindent\emph{Step 1.} We prove \eqref{eq:gamma-sum} when $d\geq3$.
	For $d\geq3$ the pointwise bound \eqref{eq:green-gradient} bounds
	$\Gamma_m(y)$ by $C(1+|y|)^{2-2d}$ uniformly in $m$, and
	$\sum_y(1+|y|)^{2-2d}<\infty$ because $2-2d<-d$.

	\smallskip
	\noindent\emph{Step 2.} We prove \eqref{eq:gamma-sum} when $d=2$.
	For $d=2$, \eqref{eq:green-gradient}
	reads $\Gamma_m(y)\leq C(1+|y|)^{-2}$. Since $g_m$ vanishes outside
	$\{|z|<m\}$, the quantity $\Gamma_m(y)$ vanishes once every neighbor of $y$
	lies outside that ball, hence once $|y|\geq m+1$, so summing over the shells
	$|y|\asymp2^j$ with $2^j\leq n+1$ leaves $C\log(n+2)$.

	\smallskip
	\noindent\emph{Step 3.} We prove \eqref{eq:gamma-sum} when $d=1$.
	For $d=1$ the pointwise bound is only $\Gamma_m(y)\leq C$, which would give
	$O(n)$, so we compute the gradient exactly. Let
	$D_m(y)\coloneqq g_m(y)-g_m(y+1)$.
	Then $(I-P)g_m(y)=\tfrac12(D_m(y)-D_m(y-1))$, so
	$(I-P)g_m(y)=\one\{y=0\}-P^m(0,y)$ becomes
	$D_m(y)-D_m(y-1)=2(\one\{y=0\}-P^m(0,y))$, while $D_m(y)=0$ for $y\geq m$.
	Summing the identity for $D_m(y)-D_m(y-1)$ over the integers larger than $y$
	leaves
	$D_m(y)=2\,\P_0(X_m>y)$ for every $y\geq0$. A uniform neighbor of $y$ takes
	the two values $y\pm1$, so for $y\geq1$,
	\[
		\Gamma_m(y)=\tfrac14\bigl(D_m(y-1)+D_m(y)\bigr)^2
		\leq4\,\P_0(X_m\geq y)^2\,,
	\]
	while $\Gamma_m(0)=0$ because $g_m$ is symmetric. Since
	$\E_0e^{\theta X_m}=(\cosh\theta)^m\leq e^{m\theta^2/2}$, taking
	$\theta=y/m$ bounds $\P_0(X_m\geq y)$ by $e^{-y^2/(2m)}$, which is at most
	$e^{-y^2/(2n)}$ for $m\leq n$. Hence $\sup_{m\leq n}\Gamma_m(y)\leq
	4e^{-y^2/n}$ for every $y$, by symmetry also for $y<0$, and summing over $y$
	leaves $C\sqrt n$.
\end{proof}

Reveal each instruction in the round when it is first used. Conditional on
the previous rounds, unused instructions remain independent, and the error is a
sum of martingale differences.

\begin{lemma}\label{lem:exposure}
	Let $\mathcal G_0\coloneqq\sigma(\eta)$ and, for $k\geq1$, let $\mathcal G_k$ be
	generated by $\mathcal G_{k-1}$ together with the indicators
	$\one\{U_k(y)\geq j,\ \rho_j(y)=x\}$ over all $x$ and $y$ in $\Z^d$ and all
	$j\geq1$. Then
	$U_{k+1}$ is measurable with respect to $\mathcal G_k$, and conditionally on
	$\mathcal G_k$ the instructions $\rho_j(y)$ with $j>U_k(y)$ are independent
	with law $P(y,\cdot)$.
\end{lemma}

\begin{proof}
	We proceed by induction on $k$. For $k=0$ the recursion reads $U_1=\eta^+$, and the
	instructions are independent of $\eta$. Suppose that $U_k$ is
	$\mathcal G_{k-1}$-measurable and that, conditionally on
	$\mathcal G_{k-1}$, the instructions $\rho_j(y)$ with
	$j>U_{k-1}(y)$ are independent with law $P(y,\cdot)$. Then the index set
	$\{(y,j):U_{k-1}(y)<j\leq U_k(y)\}$ is $\mathcal G_{k-1}$-measurable as well; conditionally on
	$\mathcal G_{k-1}$ the instructions in that set are independent with the
	law $P(y,\cdot)$. Passing to $\mathcal G_k$ reveals exactly those
	instructions, so, conditionally on $\mathcal G_k$, the instructions
	$\rho_j(y)$ with $j>U_k(y)$ remain independent with law $P(y,\cdot)$.
	Finally $I_{y,x}(U_k(y))$ is $\mathcal G_k$-measurable, so
	\eqref{eq:parallel} makes $U_{k+1}$ $\mathcal G_k$-measurable.
	Formally, the conditional-independence assertion is checked first for each finite cylinder
	of unused instructions and then extended to the full unused family by the
	monotone-class theorem.
\end{proof}

\begin{lemma}\label{lem:w-martingale}
	Fix $n\geq2$. There are a filtration $(\mathcal F_i)_{i\geq0}$ and random
	variables $(\xi_i)_{i\geq1}$, only finitely many of which are nonzero almost
	surely, such that $w_n(0)=\sum_i\xi_i$ and, for every $i\geq1$,
	\[
		\E[\xi_i\mid\mathcal F_{i-1}]=0\,,\qquad
		|\xi_i|\leq\max_{m<n}\max_y\max_{z\sim y}|g_m(z)-(Pg_m)(y)|\,.
	\]
	Moreover
	\begin{equation}\label{eq:qv}
		\sum_i\E[\xi_i^2\mid\mathcal F_{i-1}]
		=\sum_{s=1}^{n-1}\sum_yA_{s-1}(y)\Gamma_{n-s}(y)\,.
	\end{equation}
\end{lemma}

\begin{proof}
	Iterating \eqref{eq:error-recursion} one instruction at a time,
	\[
		w_n(0)=\sum_{s=1}^{n-1}\sum_y\sum_{j=U_{s-1}(y)+1}^{U_s(y)}
		\bigl[g_{n-s}(\rho_j(y))-(Pg_{n-s})(y)\bigr]\,,
	\]
	since an instruction first used in round $s$ occurs in the errors of rounds
	$s$ through $n-1$ and the associated kernels sum to $g_{n-s}$. As
	$g_m(z)=0$ for $|z|\geq m$, the bracket vanishes unless $|y|\leq n-s$, so
	finitely many terms are nonzero. For each $s$, condition on
	$\mathcal G_{s-1}$. The finite set
	\[
		\{(y,j):|y|\leq n-s,\ U_{s-1}(y)<j\leq U_s(y)\}\,,
	\]
	is then known. Reveal these instructions one at a time in a fixed order and
	then pass to $\mathcal G_s$. Concatenating the finite blocks over
	$1\leq s<n$ and padding the sequence with zeros gives $(\xi_i)$.
	Lemma~\ref{lem:exposure} shows that each instruction has conditional law
	$P(y,\cdot)$ when it is revealed. Thus an increment arising from an
	instruction first read at $y$ in round $s$ has conditional mean zero,
	absolute value at most
	$\max_{m<n}\max_v\max_{z\sim v}|g_m(z)-(Pg_m)(v)|$, and conditional variance
	$\Gamma_{n-s}(y)$.
\end{proof}

\begin{lemma}[{\citealp[Theorems~4.1 and~3.3]{Pinelis}}]\label{lem:bernstein}
	Let $(\mathcal F_i)_{i=0}^k$ be a filtration, let $(\xi_i)_{i=1}^k$ be
	martingale differences for it, let $r\geq2$, and let $a>0$. If
	$|\xi_i|\leq a$ for every $1\leq i\leq k$, then
	\[
		\Bigl(\E\Bigl|\sum_{i=1}^k\xi_i\Bigr|^r\Bigr)^{1/r}\leq
		C\Bigl(\sqrt r\,\Bigl(\E\Bigl[\sum_{i=1}^k
		\E[\xi_i^2\mid\mathcal F_{i-1}]\Bigr]^{r/2}\Bigr)^{1/r}+ra\Bigr)\,,
	\]
	with $C$ universal. If instead there is $v>0$ such that, almost surely,
	$\sum_{i=1}^k\E[|\xi_i|^q\mid\mathcal F_{i-1}]\leq\tfrac{q!}{2}a^{q-2}v$
	for every integer $q\geq2$, then
	$(\E|\sum_{i=1}^k\xi_i|^r)^{1/r}\leq C(\sqrt{rv}+ra)$, again with $C$ universal.
\end{lemma}

Every step from the origin by round $n$ is taken by a particle which began within distance $n$ and takes at most $n$ steps, so
\begin{equation}\label{eq:apriori-finite}
	U_n(0)\leq n\sum_{|y|\leq n}\eta(y)^+\,.
\end{equation}
Whenever $\eta(0)^+$ has an exponential moment, the finite sum on the
right has moments of every order.

\begin{proposition}\label{prop:w-moment}
	Let $\eta=(\eta(x))_{x\in\Z^d}$ have i.i.d.\ integer-valued coordinates,
	with $\E e^{\theta\eta(0)^+}<\infty$ for some $\theta>0$. There is a
	dimension-dependent $C<\infty$ such that, for every $n\geq1$ and $r\geq2$,
	\[
		\max_{m\leq n}\bigl(\E|w_m(0)|^r\bigr)^{1/r}
		\leq C\Bigl(\sqrt{r\kappa_d(n)\bigl(\E U_n(0)^r\bigr)^{1/r}}
		+r\Bigr)\,,
	\]
	\[
		\bigl(\E w_n^\star(0)^r\bigr)^{1/r}
		\leq(n+1)^{1/r}\max_{m\leq n}\bigl(\E|w_m(0)|^r\bigr)^{1/r}\,.
	\]
\end{proposition}

\begin{proof}
	For $m\leq n$ the sum $\sum_{s=1}^{m-1}A_{s-1}(y)$ is $U_{m-1}(y)$, so the
	predictable quadratic variation \eqref{eq:qv} is at most
	$\sum_y(\sup_{j\leq n}\Gamma_j(y))U_n(y)$, and Minkowski's inequality,
	translation invariance and \eqref{eq:gamma-sum} give
	\[
		\Bigl(\E\Bigl(\sum_y\bigl(\sup_{j\leq n}\Gamma_j(y)\bigr)U_n(y)\Bigr)^{r/2}\Bigr)^{2/r}
		\leq C\kappa_d(n)\bigl(\E U_n(0)^{r/2}\bigr)^{2/r}\,.
	\]
	The increment bound
	\eqref{eq:increment} controls $|\xi_i|$. Almost surely only finitely many
	$\xi_i$ are nonzero, but their number is unbounded, so apply
	Lemma~\ref{lem:bernstein} to $\xi_1,\ldots,\xi_k$ and let $k\to\infty$ using
	Fatou's lemma. With $(\E U_n(0)^{r/2})^{2/r}\leq(\E U_n(0)^r)^{1/r}$, this
	proves the first estimate. Jensen's inequality and a union bound over the
	$n+1$ times prove the second:
	$\E w_n^\star(0)^r\leq\sum_{j\leq n}\E\,\E_0|w_{n-j}(X_j)|^r
	\leq(n+1)\max_{m\leq n}\E|w_m(0)|^r$ by translation invariance.
\end{proof}

The martingale representation also shows that $\E[w_n(0)\mid\eta]=0$. Taking
expectations in \eqref{eq:pathwise-error} gives the upper bound, and
Theorem~\ref{thm:comparison} gives the lower bound:
\[
	0\leq\E U_n(0)-\E u_n(0)\leq\E w_n^\star(0)\,.
\]

\section{The logarithmic lower bound}\label{sec:critical}

Section~\ref{sec:divisible} bounds $\E U_n(0)$ below by $\E u_n(0)$, but from
dimension four on this need not give a logarithmic lower bound uniform over the
law of $\eta(0)$. We obtain such a bound by adapting the resampling argument of
\citet{CRS} to discrete time. Run the original and resampled configurations
with shared steps and match their common particles. For each pair of opposite
labels, cancellation by time $t$ requires a simple random walk to hit zero
within $2t$ steps. Summing these hitting probabilities bounds the expected
number of canceled labels.

\begin{lemma}\label{lem:critical-density}
	Let $\eta=(\eta(x))_{x\in\Z^d}$ be independent copies of a nonconstant integer-valued
	$\eta(0)$ with mean zero. There is a universal $c>0$ such that
	$\liminf_{t\to\infty}tS_t\geq c$, and there is
	$C<\infty$ such that, for every $n\geq2$,
	\begin{equation}\label{eq:q-lower-log}
		\E U_n(0)\geq c\log n-C\,.
	\end{equation}
\end{lemma}

\begin{proof}
	\smallskip
	\noindent\emph{Step 1.} We construct the coupled particle systems.
	Let $\gamma>0$ be the expected absolute difference between two independent
	variables with the law of $\eta(0)$.
	Fix $\eps$ with $0\leq\eps\leq\gamma$. Independently at each site, replace
	$\eta(x)$ by an independent copy with probability $\eps/\gamma$. The
	expected absolute change at each site is $\eps$, and the resampled
	configuration has the same law as $\eta$. We use tildes for the process
	started from the resampled configuration. The change at each site is
	symmetric, so its positive and negative parts have mean $\eps/2$.

	Match $A_t(x)\wedge\widetilde A_t(x)$ active particles at every site and give
	each matched pair one common step. Matched arrivals cancel in
	\eqref{eq:signed-recursion}. Thus the signed difference between the two
	processes at $x$ after round $t+1$ counts the unmatched particles arriving at
	$x$ together with the unmatched holes already there. Neither process has an
	active particle and an unfilled hole at the same site, so all unmatched
	particles and holes at one site have the same sign. Each unmatched particle
	takes one fresh step in round $t+1$, while each unmatched hole waits. Cancel
	opposite unmatched particles and holes in pairs whenever they reach the same
	site.

	\smallskip
	\noindent\emph{Step 2.} We compare each pair of oppositely signed labels with
	a simple random walk.
	Label every initially unmatched particle or hole by its initial site and
	sign, and give it an independent continuous priority. When matching replaces
	the unmatched particle or hole carrying a label, transfer the label and
	priority to the replacement. Use the priorities for all matching and
	cancellation choices. This rule is translation-equivariant.

	Let $\mathcal G$ be generated by the initial labels and their priorities. Fix
	opposite labels $\alpha$ and $\beta$ initially at $x$ and $y$. Until either
	label disappears, reveal chronologically every step taken by the unmatched
	particle bearing either label. When both move in one round, first reveal the
	step of the particle bearing $\alpha$. Starting from $x-y$, append $v$ when the particle
	bearing $\alpha$ moves by $v$, and append $-v$ when the particle bearing
	$\beta$ moves by $v$. Immediately before each instruction is revealed, the
	particle reading it and its active status are determined by the revealed past,
	while the instruction is fresh. Conditional on $\mathcal G$, the appended
	increments are therefore independent and uniform over the neighbors of the
	origin. After either label disappears, append independent increments to
	obtain a full simple random walk.

	Before the labels disappear, this walk records the position of the particle or
	hole bearing $\alpha$ minus the position of the particle or hole bearing
	$\beta$ after each round. If the two labels cancel each other, then their positions coincide at
	cancellation, so the walk has hit zero. At most $2t$ increments
	have then been appended by round $t$, and hence
	\[
		\P\bigl(\alpha\text{ and }\beta\text{ cancel each other by }t
		\bigm|\mathcal G\bigr)
		\leq\P_{x-y}(T_0\leq2t)\,.
	\]

	\smallskip
	\noindent\emph{Step 3.} We prove $S_t\geq1/(16t)$ for
	$t\geq1/(2\gamma)$ and \eqref{eq:q-lower-log}.
	We use
	\[
		\sum_{z\neq0}\P_z(T_0\leq2t)
		=\E_0|R_{2t}|-1\leq2t\,,
	\]
	where the equality follows from $\P_z(T_0\leq m)=\P_0(T_z\leq m)$.
	All unmatched particles or holes created at one site have the same sign, and
	the collections created at distinct sites are independent. For each pair of
	distinct sites, the expected number of opposite-sign pairs is $\eps^2/2$.
	Hence the expected number of labels created at the origin that disappear by
	time $t$ is at most
	$\tfrac{\eps^2}{2}(\E_0|R_{2t}|-1)$. The expected number of labels created at
	the origin that remain is therefore at least
	$\eps-\tfrac{\eps^2}{2}(\E_0|R_{2t}|-1)$. By the mass transport principle, this is also the
	expected number of unmatched particles and holes present at the origin. Every
	such particle or hole is counted among the active particles and unfilled holes
	of the two processes. Since both processes have the original mean-zero law,
	their total expected number of active particles and unfilled holes at the
	origin is
	$\E A_t(0)+\E H_t(0)+\E\widetilde A_t(0)+\E\widetilde H_t(0)=4S_t$, so that
	for every $0\leq\eps\leq\gamma$,
	\[
		4S_t\geq\eps-\frac{\eps^2}{2}\bigl(\E_0|R_{2t}|-1\bigr)
		\geq\eps-t\eps^2\,.
	\]
	For $t\geq1/(2\gamma)$, take $\eps=1/(2t)$. Then $\eps-t\eps^2=1/(4t)$, so
	$S_t\geq1/(16t)$. Thus $\liminf_{t\to\infty}tS_t\geq1/16$. Summing over
	$1/(2\gamma)\leq t<n$ gives \eqref{eq:q-lower-log}.
\end{proof}

\begin{corollary}\label{cor:critical}
	Let $\eta=(\eta(x))_{x\in\Z^d}$ be independent copies of a nonconstant integer-valued
	$\eta(0)$ with mean zero. Then there are $c,C>0$, with $c$ universal, such
	that, for every $n\geq2$,
	\[
		\E U_n(0)\geq\max\bigl\{\E u_n(0),\ c\log n-C\bigr\}\,.
	\]
\end{corollary}

\begin{proof}
	Theorem~\ref{thm:comparison} gives
	$\E U_n(0)\geq\E u_n(0)$, and Lemma~\ref{lem:critical-density} gives
	$\E U_n(0)\geq c\log n-C$.
\end{proof}

\section{Growth of the expected odometer}\label{sec:bounds}

The proof begins with the pathwise comparison and absorbs its square-root term
by Young's inequality. This leaves the corresponding moment of $u_n(0)$ and
$r(n+1)^{2/r}\kappa_d(n)$.
Lemma~\ref{lem:u-concentration} then bounds the moment of $u_n(0)$ by its mean
and the two Green-function terms in
\eqref{eq:green-norms}. We take $r=8$ below dimension four, where all error
terms have the order of $\E u_n(0)$. From dimension four on, we take $r$ of
order $\log n$, which keeps $(n+1)^{2/r}$ bounded and uses $\kappa_d(n)=1$.

Assume the hypotheses of Theorem~\ref{thm:master}.

The Green estimates collected in \citet[Section~3.1]{BP} give
\begin{equation}\label{eq:green-norms}
	\|g_n\|_2\asymp
	\begin{cases}
		n^{3/4}&d=1\\
		n^{1/2}&d=2\\
		n^{1/4}&d=3\\
		\sqrt{\log n}&d=4\\
		1&d\geq5
	\end{cases}
	\qquad
	\max_xg_n(x)\asymp
	\begin{cases}
		n^{1/2}&d=1\\
		\log n&d=2\\
		1&d\geq3
	\end{cases}
	\,.
\end{equation}

\begin{lemma}\label{lem:u-concentration}
	Let $K$ be a finite-range, translation-invariant transition kernel on
	$\Z^d$. Let $g_n^K(z)\coloneqq\sum_{j<n}K^j(0,z)$, and let
	$v_0=0$ and $v_{n+1}=(\eta+Kv_n)^+$ for i.i.d.\ $\eta(x)$ satisfying
	$\E e^{\theta|\eta(0)|}<\infty$ for some $\theta>0$. Then there is
	$C<\infty$, depending on $K$ and the law of $\eta(0)$, such that, for every
	$n\geq1$ and $r\geq2$,
	\begin{equation}\label{eq:u-concentration}
		\bigl(\E|v_n(0)-\E v_n(0)|^r\bigr)^{1/r}
		\leq C\bigl(\sqrt r\,\|g_n^K\|_2+r\|g_n^K\|_\infty\bigr)\,.
	\end{equation}
\end{lemma}

Equation \eqref{eq:u-concentration} follows from the general-kernel concentration estimate \citep[Remark~3.4]{BP}. In this section,
we use $K=P$, $v=u$ and $g_n^P=g_n$.

\begin{theorem}[Critical upper bound]\label{thm:upper}
	Let $\eta=(\eta(x))_{x\in\Z^d}$ be independent copies of a nonconstant integer-valued
	$\eta(0)$ which has mean zero and satisfies $\E e^{\theta|\eta(0)|}<\infty$ for
	some $\theta>0$. There is $C<\infty$ such that, for every $n\geq2$,
	\begin{equation}\label{eq:target}
		\E U_n(0)\leq C\bigl(\E u_n(0)+\log n\bigr)\,.
	\end{equation}
	More generally, for every $r\geq2$,
	\begin{equation}\label{eq:critical-moment}
		\bigl(\E U_n(0)^r\bigr)^{1/r}
		\leq C\bigl(\E u_n(0)+\sqrt r\,\|g_n\|_2+r\max_xg_n(x)
		+r(n+1)^{2/r}\kappa_d(n)\bigr)\,.
	\end{equation}
\end{theorem}

\begin{proof}
	\smallskip
	\noindent\emph{Step 1.} We prove \eqref{eq:critical-moment}.
	Fix $n\geq1$ and $r\geq2$. By \eqref{eq:apriori-finite}, every
	moment in this proof is finite. Taking $r$-th moments in
	\eqref{eq:pathwise-comparison} and using Proposition~\ref{prop:w-moment},
	\[
		\bigl(\E U_n(0)^r\bigr)^{1/r}\leq\bigl(\E u_n(0)^r\bigr)^{1/r}
		+C(n+1)^{1/r}
		\Bigl(\sqrt{r\kappa_d(n)\bigl(\E U_n(0)^r\bigr)^{1/r}}+r\Bigr)\,.
	\]
	Young's inequality absorbs half of $(\E U_n(0)^r)^{1/r}$, and
	$\kappa_d(n)\geq1$ leaves
	\[
		\bigl(\E U_n(0)^r\bigr)^{1/r}\leq
		C\bigl(\bigl(\E u_n(0)^r\bigr)^{1/r}
		+r(n+1)^{2/r}\kappa_d(n)\bigr)\,.
	\]
	Combining this with Lemma~\ref{lem:u-concentration} yields
	\eqref{eq:critical-moment}.

	\smallskip
	\noindent\emph{Step 2.} We prove \eqref{eq:target} when $d\leq3$.
	Take $r=8$. By \eqref{eq:green-norms},
	$\sqrt r\,\|g_n\|_2+r\max_xg_n(x)\leq C\|g_n\|_2$. The term
	$r(n+1)^{2/r}\kappa_d(n)$ has order
	\[
		n^{3/4}\ (d=1)\,,\qquad n^{1/4}\log n\ (d=2)\,,\qquad n^{1/4}\ (d=3)\,.
	\]
	Theorem~\ref{thm:BP} bounds $\|g_n\|_2$ and the last displayed term by
	$C\E u_n(0)$.

	\smallskip
	\noindent\emph{Step 3.} We prove \eqref{eq:target} when $d\geq4$.
	Take
	$r=2\vee\lceil\log(n+1)\rceil$, so that $(n+1)^{2/r}\leq e^2$
	and $\kappa_d(n)=1$. Then
	$\sqrt r\,\|g_n\|_2+r\max_xg_n(x)+r(n+1)^{2/r}\kappa_d(n)\leq C\log n$.
\end{proof}

\begin{proof}[Proof of Theorem~\ref{thm:master}]
	The upper bound is \eqref{eq:target}. Corollary~\ref{cor:critical} bounds
	$\E U_n(0)$ below by $c(\E u_n(0)+\log n)-C$, which is the required lower
	bound for all sufficiently large $n$. The remaining values are covered by
	decreasing $c$, since
	$\E U_n(0)\geq\E U_1(0)=\E\eta(0)^+>0$.
\end{proof}

\begin{proof}[Proof of Corollary~\ref{cor:growth}]
	Insert Theorem~\ref{thm:BP} into \eqref{eq:master}. For $d\leq3$,
	$\E u_n(0)\asymp n^{(4-d)/4}$ dominates
	$\log n$; for $d=4$ $\E u_n(0)\asymp\log n$; for $d\geq5$ the
	term $\E u_n(0)$ is at most $C\log n$. For \eqref{eq:activity}, recall that each
	indicator that a particle has not settled is nonincreasing in $t$, hence so
	is $S_t$.
	When $d\leq3$, let $\beta=(4-d)/4$. Comparing the partial sums over
	$[0,n)$ and $[n,mn)$, for a fixed sufficiently large $m$, gives
	$S_t\asymp(t+1)^{\beta-1}$. When $d\geq4$, monotonicity gives
	$(t+1)S_t\leq\sum_{s\leq t}S_s=\E U_{t+1}(0)\leq C\log(t+2)$. The lower
	bound in \eqref{eq:activity} follows from Lemma~\ref{lem:critical-density} for
	all large $t$.
	On the event $\eta(0)>0$, select one particle at the origin. With positive
	probability, no other site in its assigned walk through time $t$ initially
	contains a hole, so that particle remains active. Hence $S_t>0$ at every
	fixed time, and decreasing the constant covers the finitely many smaller
	times.
\end{proof}

\begin{proposition}\label{prop:everyone-settles}
	Let $\eta=(\eta(x))_{x\in\Z^d}$ have i.i.d.\ integer-valued coordinates.
	Suppose that $\eta(0)$ is nonconstant, $\E\eta(0)=0$, and
	$\E e^{\theta|\eta(0)|}<\infty$ for some $\theta>0$. Then the following hold
	together on one event of probability one: every particle settles after
	finitely many rounds; every hole is filled after finitely many rounds; and
	infinitely many distinct particles leave every site, so that
	$U_\infty(x)=\infty$ for every $x\in\Z^d$.
\end{proposition}

\begin{proof}
	\smallskip
	\noindent\emph{Step 1.} We prove that every particle settles and every hole is
	filled.
	Since $S_t$ is nonincreasing, $(t+1)S_t\leq\E U_{t+1}(0)=o(t)$ by
	Corollary~\ref{cor:growth}, so $S_t\to0$. The expected
	number of particles which start at the origin and never settle is
	$\lim_tS_t=0$, so almost surely there is none. Translation invariance and
	countability extend this to every particle. Since the mean vanishes,
	$\E H_t(0)=S_t$, and $H_t(0)$ decreases, so $\E H_\infty(0)=0$ and every
	hole at the origin is filled almost surely. Translation invariance and
	countability extend this to every hole.

	\smallskip
	\noindent\emph{Step 2.} We prove that every site has infinite odometer.
	The estimate
	\eqref{eq:critical-moment} bounds the second moment:
	\[
		\bigl(\E U_n(0)^2\bigr)^{1/2}\leq C\bigl(\E u_n(0)+\log n\bigr)\,.
	\]
	Indeed, take $r=8$ when $d\leq3$ and $r=2\vee\lceil\log(n+1)\rceil$ when
	$d\geq4$. Theorem~\ref{thm:master} and the Paley--Zygmund inequality
	therefore produce $c>0$ such that
	\[
		\P\bigl(U_n(0)\geq c(\E u_n(0)+\log n)\bigr)\geq c\,,
	\]
	for every $n\geq2$. Since $\E u_n(0)+\log n$ tends to infinity, for every
	$M\geq1$ there is $n$ with $c(\E u_n(0)+\log n)\geq M$.
	Hence $\P(U_\infty(0)\geq M)\geq c$ for every $M$, and letting $M$
	tend to infinity gives $\P(U_\infty(0)=\infty)\geq c$.
	The event that some site has infinite odometer is invariant under the
	translations of $\Z^d$, which act ergodically on $\eta$ and the instructions,
	so it has probability one. Almost surely, the stack at each $y$ sends infinitely
	many instructions to each neighbor $x$. If $U_\infty(y)=\infty$, then every
	instruction at $y$ is eventually used, so $x$ receives infinitely many
	arrivals. At most $\eta(x)^-<\infty$ of them settle at $x$, and every other
	arrival leads to a departure. Connectedness of $\Z^d$ shows that every site
	has infinite odometer. Since every particle takes finitely many steps,
	infinitely many distinct particles leave every site.
\end{proof}

\section{Quenched odometer comparison}\label{sec:four}

In dimensions at most three the two odometers agree to leading order, while in
dimensions at least five their expectations have different orders when
$\eta(0)$ is bounded below; these claims follow from
Section~\ref{sec:bounds} and \eqref{eq:bp}. In dimension four both
expectations are of order $\log n$ for each fixed initial law, but the upper
comparison constant cannot be uniform in that law. Indeed, the logarithmic
coefficient in \eqref{eq:q-lower-log} is law-independent, whereas
\eqref{eq:stopping} shows that $u_n(0)$, as a function of $\eta$, is
coordinatewise convex and positively
homogeneous. Theorem~\ref{thm:four-sparse} exploits this distinction with a
sparse symmetric law.

\begin{proposition}\label{prop:discrepancy}
	Under the assumptions of Theorem~\ref{thm:master}, suppose $d\leq3$. For
	$n\geq1$, let $r=2\vee\lceil\log(n+1)\rceil$. Then
	\begin{equation}\label{eq:discrepancy}
		\bigl(\E|U_n(0)-u_n(0)|^r\bigr)^{1/r}\leq C
		\begin{cases}
			n^{5/8}[\log(n+1)]^{3/4}&d=1\,,\\
			n^{1/4}[\log(n+1)]^{5/4}&d=2\,,\\
			n^{1/8}[\log(n+1)]^{3/4}&d=3\,.
		\end{cases}
	\end{equation}
	Consequently, for every $\eps>0$, there is $c>0$ such that, for all
	sufficiently large $n$,
	\begin{equation}\label{eq:discrepancy-tail}
		\P\bigl(|U_n(0)-u_n(0)|>\eps\E u_n(0)\bigr)
		\leq e^{-c(\log n)^2}\,.
	\end{equation}
\end{proposition}

\begin{proof}
	Take $r=2\vee\lceil\log(n+1)\rceil$ in
	\eqref{eq:critical-moment}. Since $d\leq3$, \eqref{eq:green-norms} and
	Theorem~\ref{thm:BP} bound $(\E U_n(0)^r)^{1/r}$ by
	$C\E u_n(0)\sqrt{\log(n+1)}$. Insert this into
	Proposition~\ref{prop:w-moment} and then
	\eqref{eq:pathwise-comparison}. Substituting the three values of
	$\kappa_d(n)$ and $\E u_n(0)\asymp n^{(4-d)/4}$ produces
	\eqref{eq:discrepancy}. In each dimension its right side divided by
	$\E u_n(0)$ is at most $n^{-c}$ times a power of $\log n$, so Markov's
	inequality with exponent $r$ yields \eqref{eq:discrepancy-tail}.
\end{proof}

For the dimension-four construction, write $u_n(0;\xi)$ for the solution of
\eqref{eq:divisible} at the origin with $\eta$ replaced by $\xi$.

\begin{theorem}[Rare nonzero values in dimension four]\label{thm:four-sparse}
	Let $d=4$ and, for $0<\eps\leq1/2$, let $(\eta(x))_{x\in\Z^4}$ be independent
	and identically distributed, with $\eta(0)$ taking the values $1$ and
	$-1$ with probability $\eps/2$ each and $0$ otherwise. There is
	$c>0$, independent of $\eps$, such that
	\begin{equation}\label{eq:four-sparse}
		\liminf_{n\to\infty}\frac{\E U_n(0)}{\E u_n(0)}
		\geq c\log(e/\eps).
	\end{equation}
\end{theorem}

\begin{proof}
	\smallskip
	\noindent\emph{Step 1.} We compare the two initial laws in convex order.
	The representation \eqref{eq:stopping} shows that
	$\xi\mapsto u_n(0;\xi)$ is convex and positively homogeneous: it is a
	supremum of linear functions of the coordinates $\xi(x)$ with $|x|<n$.
	Let $(\zeta(x))_{x\in\Z^4}$ be independent with the symmetric Laplace
	density $s\mapsto\tfrac12e^{-|s|}$, and put
	$b_\eps=1/\log(e/\eps)$. The law of $\eta(0)$ is dominated by that of $b_\eps\zeta(0)$ in convex
	order.
	Indeed, for $0\leq t\leq1$,
	\[
		\E(\eta(0)-t)^+=\frac\eps2(1-t),
		\qquad
		\E(b_\eps\zeta(0)-t)^+
		=\frac{b_\eps}{2}e^{-t/b_\eps}.
	\]
	For $0\leq t<1$, the ratio
	$b_\eps e^{-t/b_\eps}/(1-t)$ is minimized at
	$t=1-b_\eps\in(0,1)$, where its value is
	$e^{1-1/b_\eps}=\eps$. Thus
	$\E(\eta(0)-t)^+\leq\E(b_\eps\zeta(0)-t)^+$ holds on
	$[0,1]$. For $t\geq1$ its left side is zero, and symmetry together with
	equality of the means gives it for $t<0$.

	\smallskip
	\noindent\emph{Step 2.} We prove \eqref{eq:four-sparse}.
	Replacing the finitely many coordinates one at a time, then using
	independence and positive homogeneity of
    $\xi\mapsto u_n(0;\xi)$, gives
	\[
		\E u_n(0;\eta)
		\leq b_\eps\,\E u_n(0;\zeta)
		\leq Cb_\eps\log(n+1),
	\]
	where we used Theorem~\ref{thm:BP} for the second inequality. On the other hand,
	\eqref{eq:q-lower-log} gives
	$\E U_n(0)\geq c\log n-C_\eps$. Dividing and letting $n\to\infty$ proves
	\eqref{eq:four-sparse}.
\end{proof}

This law occurs in the particle--hole model: the number of particles at a site,
$\eta(0)+1$, takes the values $0$, $1$ and $2$. Thus, even within that model,
the ratio cannot be bounded uniformly over the initial law in dimension four.

\begin{proof}[Proof of Theorem~\ref{thm:trichotomy}]
	For $d\leq3$, the right side of \eqref{eq:discrepancy} is $o(\E u_n(0))$,
	and, since $2\vee\lceil\log(n+1)\rceil\to\infty$, this gives convergence in
	every fixed $L^q$. The summable bound
	\eqref{eq:discrepancy-tail} upgrades it to almost sure convergence by the
	Borel--Cantelli lemma. Taking expectations leaves
	$\E U_n(0)/\E u_n(0)\to1$, and Theorem~\ref{thm:BP} makes
	$n^{-(4-d)/4}\E u_n(0)$ converge to a limit in $(0,\infty)$, proving part
	\textup{(i)}.

	In dimension four Theorem~\ref{thm:BP} gives $\E u_n(0)\asymp\log n$, so
	both sides of \eqref{eq:master} are of order $\E u_n(0)$ and
	$\E U_n(0)/\E u_n(0)$ stays between two positive constants. Given $B>0$,
	choose $\eps$ in Theorem~\ref{thm:four-sparse} so that
	$c\log(e/\eps)\geq B$. This proves
	part~\textup{(ii)}.

	In dimensions five and higher, the assumption that $\eta(0)$ is bounded
	below and Theorem~\ref{thm:BP} give $\E u_n(0)=o(\log n)$. Hence
	$\E u_n(0)+\log n\asymp\log n$, so
	$\E U_n(0)\asymp\log n$ by Theorem~\ref{thm:master} and
	$\E U_n(0)/\E u_n(0)\to\infty$. Moreover
	$\E|U_n(0)-u_n(0)|\geq\E U_n(0)-\E u_n(0)\geq c\log n$ by
	\eqref{eq:q-lower-log}, while
	$\E|U_n(0)-u_n(0)|\leq\E U_n(0)+\E u_n(0)\leq C\log n$.
	This proves part~\textup{(iii)}.
\end{proof}

\section{The nearest particle and hole}\label{sec:nearest}

\subsection{Below dimension four}

Below dimension four, the rescaled scenery converges to spatial white noise, while the two rescaled odometers converge to the associated Brownian optimal-stopping value. The
rescaled excess of active particles over unfilled holes converges to the time
derivative of that value. This derivative solves the heat equation and is
strictly positive wherever the value is positive.

\begin{proposition}[Spatial scaling]\label{prop:spatial-scaling}
	Let $d\leq3$, and let $(\eta(x))_{x\in\Z^d}$ be independent and identically
	distributed, integer-valued and nonconstant, with $\E\eta(0)=0$ and
	$\E e^{\theta|\eta(0)|}<\infty$ for some $\theta>0$. For $R>0$,
	$s\geq0$, and $x\in\mathbb R^d$, let
	\[
		\overline U_R(s,x)\coloneqq
		R^{d/2-2}U_{\lfloor sR^2\rfloor}(\lfloor Rx\rfloor),
		\qquad
		\overline u_R(s,x)\coloneqq
		R^{d/2-2}u_{\lfloor sR^2\rfloor}(\lfloor Rx\rfloor),
	\]
	where the floor is taken coordinatewise. For $R>0$ and
	$\varphi\in C_c^\infty(\mathbb R^d)$, let
	\[
		\langle\eta_R,\varphi\rangle
		\coloneqq R^{-d/2}\sum_y\eta(y)\varphi(y/R)
	\]
	and
	\[
		\langle\nu_R,\varphi\rangle
		\coloneqq R^{-d/2}\sum_y
		\bigl(A_{\lfloor R^2\rfloor}(y)-H_{\lfloor R^2\rfloor}(y)\bigr)
		\varphi(y/R).
	\]
	Let $\mathcal W$ be a mean-zero spatial white noise with covariance
	\[
		\E\langle\mathcal W,\varphi\rangle\langle\mathcal W,\psi\rangle
		=\Var\eta(0)\int_{\mathbb R^d}\varphi(x)\psi(x)\,dx
		\qquad(\varphi,\psi\in C_c^\infty(\mathbb R^d))\,.
	\]
	Let $\mathcal U$ be the continuous Brownian optimal-stopping value driven by
	$\mathcal W$, with $\mathcal U(0,\cdot)=0$. Then, jointly,
	\begin{equation}\label{eq:space-time-scaling}
		(\eta_R,\overline u_R,\overline U_R)
		\Longrightarrow(\mathcal W,\mathcal U,\mathcal U)
	\end{equation}
	as $R\to\infty$, with the first coordinate converging as a random
	distribution and the last two locally uniformly on
	$(0,\infty)\times\mathbb R^d$.
	Let
	$\mathcal O\coloneqq\{\mathcal U>0\}$ and
	$\mathcal L\coloneqq(2d)^{-1}\Delta$. Then, in
	the sense of distributions on $\mathcal O$,
	\begin{equation}\label{eq:continuum-equation}
		\partial_s\mathcal U=\mathcal L\mathcal U+\mathcal W.
	\end{equation}
	For every $\varphi\in C_c^\infty(\mathbb R^d)$, jointly with
	\eqref{eq:space-time-scaling},
	\begin{equation}\label{eq:signed-density-limit}
		\langle\nu_R,\varphi\rangle\Longrightarrow
		\langle\mathcal W+\mathcal L\mathcal U(1,\cdot),\varphi\rangle.
	\end{equation}
	Moreover, on $\mathcal O$ the distribution
	$v=\partial_s\mathcal U$ is a smooth function and
	\begin{equation}\label{eq:strict-time-derivative}
		v(s,x)>0\qquad ((s,x)\in\mathcal O).
	\end{equation}
\end{proposition}

\begin{proof}
	\smallskip
	\noindent\emph{Step 1.} We prove \eqref{eq:space-time-scaling} and
	\eqref{eq:continuum-equation}.
	The argument proving the parabolic scaling limit in \citet{BP}, with the
	time variable and scenery retained, gives the joint convergence of
	$(\eta_R,\overline u_R)$. Its estimates are uniform on compact time
	intervals. A union bound in \eqref{eq:discrepancy-tail} over the
	$O(R^{d+2})$ space--time mesh points in a compact set transfers this
	convergence to $\overline U_R$. On a compact subset of $\mathcal O$, local
	uniform positivity makes
	the positive part in the divisible recursion inactive for all large $R$.
	Test $u_{n+1}-u_n=\eta+(P-I)u_n$ against a smooth compactly supported
	function and pass to the limit to obtain \eqref{eq:continuum-equation}.

	\smallskip
	\noindent\emph{Step 2.} We prove \eqref{eq:signed-density-limit}.
	Let $t\coloneqq\lfloor R^2\rfloor$ and
	$\varphi_R(y)\coloneqq\varphi(y/R)$. The identity \eqref{eq:signed-count}
	gives
	\begin{equation}\label{eq:signed-density-decomposition}
		\langle\nu_R,\varphi\rangle
		=\langle\eta_R,\varphi\rangle
		+R^{-d/2}\sum_yU_t(y)(P-I)\varphi_R(y)+M_R(\varphi),
	\end{equation}
	and
	\[
		M_R(\varphi)=R^{-d/2}\sum_y\sum_{j\leq U_t(y)}
		\left[\varphi\bigl(\rho_j(y)/R\bigr)-(P\varphi_R)(y)\right].
	\]
	Reveal the instructions in the order in which the particles read them. By Lemma~\ref{lem:exposure}, the summands in $M_R$ are martingale differences. Since each term is bounded by $C_\varphi/R$,
	Corollary~\ref{cor:growth} gives
	\[
		\E M_R(\varphi)^2
		\leq CR^{-d-2}
		\sum_{\operatorname{dist}(y,R\operatorname{supp}\varphi)\leq1}
		\E U_t(y)
		\leq CR^{-d/2}.
	\]
	Thus $M_R(\varphi)$ tends to zero in probability. By Taylor expansion and the
	locally uniform convergence of $\overline U_R(1,\cdot)$, the middle term
	in \eqref{eq:signed-density-decomposition} converges to
	$\langle\mathcal L\mathcal U(1,\cdot),\varphi\rangle$. The joint
	scenery--odometer convergence in \eqref{eq:space-time-scaling} now proves
	\eqref{eq:signed-density-limit}.

	\smallskip
	\noindent\emph{Step 3.} We prove that $v=\partial_s\mathcal U$ is smooth and
	satisfies $(\partial_s-\mathcal L)v=0$ on $\mathcal O$. Since $\mathcal U$ is nondecreasing in time, $v$ is nonnegative. On a cylinder compactly contained in $\mathcal O$, the time differences
	\[
		\frac{\mathcal U(s+h,x)-\mathcal U(s,x)}h
	\]
	are nonnegative and, by \eqref{eq:continuum-equation}, solve the heat equation: the time-independent white noise cancels. Letting $h\downarrow0$ shows that $(\partial_s-\mathcal L)v=0$ distributionally. Interior parabolic regularity therefore represents $v$ by a smooth nonnegative function on $\mathcal O$.

	\smallskip
	\noindent\emph{Step 4.} We prove \eqref{eq:strict-time-derivative}.
	Suppose that $v(s_0,x_0)=0$ at a point of $\mathcal O$, and let
	$\tau=\inf\{s:\mathcal U(s,x_0)>0\}$. Continuity and monotonicity give
	$\mathcal U(\tau,x_0)=0$. Since
	$(\tau,s_0]\times\{x_0\}\subset\mathcal O$, the strong minimum principle
	gives $v(s,x_0)=0$ for $\tau<s\leq s_0$. Hence $\mathcal U(s_0,x_0)=
	\mathcal U(\tau+\varepsilon,x_0)$ for every
	$0<\varepsilon<s_0-\tau$. Letting
	$\varepsilon\downarrow0$ gives $\mathcal U(s_0,x_0)=
	\mathcal U(\tau,x_0)=0$, a contradiction. This proves
	\eqref{eq:strict-time-derivative}.
\end{proof}

\begin{proof}[Proof of Theorem~\ref{thm:nearest}]
	Let $R=\sqrt t$ and use the notation of
	Proposition~\ref{prop:spatial-scaling}. Fix $K\geq1$ and $\eps>0$.
	The parabolic scaling limit and critical-scale lower-tail estimate of
	\citet{BP} imply that $\mathcal U(1,0)>0$ almost surely. By continuity and
	\eqref{eq:strict-time-derivative}, there are $r,a>0$ and a nonzero smooth
	$\varphi\geq0$, supported in the $\ell^1$-ball of radius $r$, such that with
	probability at least $1-\eps$,
	\[
		\inf_{|x|_1\leq(K+2)r}\mathcal U(1,x)>2a,
		\qquad \int\varphi(x)v(1,x)\,dx>2a.
	\]
	Proposition~\ref{prop:spatial-scaling} shows that, with limiting probability
	at least $1-\eps$, $U_t$ is positive on the lattice ball of radius
	$(K+1)rR$ and $\langle\nu_R,\varphi\rangle>a$.
	If $U_t(x)>0$, then a particle has left $x$,
	so every hole there has been filled and $H_t(x)=0$. The ball of
	radius $(K+1)rR$
	therefore contains no unfilled hole. On the support of $\varphi$ we have
	$A_t-H_t=A_t$, so
	$\langle\nu_R,\varphi\rangle>a$ implies that an active particle lies within
	distance $rR$. Thus
	every unfilled hole is more than $K$ times as far from the
	origin as the nearest active particle with limiting probability at least
	$1-\eps$. Letting
	$\eps\downarrow0$ completes the proof.
\end{proof}

\subsection{From dimension five on}

We now prove Theorem~\ref{thm:nearest-counterexample}. Fix $d\geq5$; every
constant below depends only on $d$. For $p\in(0,1/4]$, let
$(\eta(x))_{x\in\Z^d}$ be i.i.d.\ with
\[
	\P(\eta(0)=1)=\P(\eta(0)=-1)=p,
	\qquad \P(\eta(0)=0)=1-2p.
\]
Let
\[
	G(x)=\sum_{n\geq0}\P_0(X_n=x),\qquad g=G(0),\qquad
	q_x(y)=\frac{G(y-x)}g=\P_y(T_x<\infty)\,.
\]
The Chapman--Kolmogorov equations and the Gaussian bound for the simple random walk
\citep[Section~2.3]{LawlerLimic} give
\begin{equation}\label{eq:nearest-bubble-decay}
	\sum_yG(y-x)G(y-z)
	=\sum_{n\geq0}(n+1)\P_0(X_n=z-x)
	\leq C(1+|x-z|)^{4-d}\,.
\end{equation}
Let
\[
	h_t=\P(H_t(0)=1),\qquad m_t=\E U_t(0)\,.
\]
Since $\eta\geq-1$, no site ever holds more than one unfilled hole, and
Lemmas~\ref{lem:activity-holes} and~\ref{lem:transport} give
\begin{equation}\label{eq:nearest-one-point-identities}
	\E A_t(0)=\E H_t(0)=h_t,\qquad
	m_t=\sum_{s<t}h_s\,.
\end{equation}

\begin{lemma}\label{lem:nearest-one-point}
	There is $C<\infty$, depending only on $d$, such that, for every
	$t\geq0$, $x\in\Z^d$, and $r\geq2$,
	\begin{equation}\label{eq:nearest-odometer-moment}
		\bigl(\E U_t(x)^r\bigr)^{1/r}\leq C(m_t+r),
		\qquad
		m_t\leq C\log(1/h_t)\,.
	\end{equation}
	Moreover, $h_t\downarrow0$.
\end{lemma}

\begin{proof}
	\smallskip
	\noindent\emph{Step 1.} We prove the first bound in
	\eqref{eq:nearest-odometer-moment}. Raise one coordinate of $\eta$ from $-1$
	to $1$.
	Apply Lemma~\ref{lem:one-particle} once for each added particle. Whenever the one-particle discrepancy is represented by an active particle, that particle uses a fresh instruction at each departure. Its trajectory is therefore a simple random walk. Changing $\eta(y)$ therefore changes
	$\E[U_t(x)\mid\eta]$ by at most $2G(y-x)$, and $\sum_yG(y-x)^2<\infty$, so the bounded difference inequality gives
	\begin{equation}\label{eq:nearest-scenery-moment}
		\bigl(\E\bigl|\E[U_t(x)\mid\eta]-m_t\bigr|^r\bigr)^{1/r}
		\leq C\sqrt r\,.
	\end{equation}
	Now reveal the used instructions in the order in which the particles read them.
	Compare each instruction with the deletion of the particle which reads it.
	This bounds the corresponding Doob increment by $g$. For an instruction at
	$y$, it bounds the conditional variance by
	$\frac1{2d}\sum_{z\sim y}G(z-x)^2$. Thus the predictable quadratic variation
	is at most $\sum_yU_t(y)\,\frac1{2d}\sum_{z\sim y}G(z-x)^2$. The weights sum
	to $\sum_zG(z-x)^2<\infty$. Lemma~\ref{lem:bernstein},
	Minkowski's inequality, translation invariance and
	\eqref{eq:nearest-scenery-moment} give
	\[
		\bigl(\E U_t(x)^r\bigr)^{1/r}
		\leq m_t+C\sqrt r
		+C\sqrt{r\bigl(\E U_t(0)^r\bigr)^{1/r}}+Cr\,.
	\]
	Young's inequality absorbs the square root, proving the first bound in
	\eqref{eq:nearest-odometer-moment}. A permanent sink only lowers the
	odometers, so this bound also holds for the process with a sink.

	\smallskip
	\noindent\emph{Step 2.} We prove \eqref{eq:nearest-hole-sink} and
	\eqref{eq:nearest-sink-mean}.
	Empty the origin, make it a permanent sink, and let $U^{\rm sink}_t$ be its
	odometer. Let
	$J$ count the particles killed there through round $t$; its predictable
	compensator is
	\[
		\lambda\coloneqq\frac1{2d}\sum_{y\sim0}U^{\rm sink}_t(y)\,.
	\]
	The origin's hole remains unfilled through round $t$ exactly when
	$\eta(0)=-1$ and no particle arrives. Before the first arrival, the original
	and sink processes agree. The sink process is independent of $\eta(0)$.
	Hence
	\begin{equation}\label{eq:nearest-hole-sink}
		h_t=p\,\P(J=0)\,.
	\end{equation}
	In an auxiliary coupled process, set
	\[
		\eta(0)=-1-\#\{y:|y|\leq t\}.
	\]
	Since $\eta^+\leq1$, its holes at the origin cannot be exhausted by round
	$t$, so off the origin it agrees with the sink process through that round.
	Its initial configuration is below the original one, and
	Lemma~\ref{lem:one-particle} gives $U_t^{\rm sink}\leq U_t$,
	$A_t^{\rm sink}\leq A_t$ and $H_t^{\rm sink}\geq H_t$ off the origin.
	Thus $\E\lambda\leq m_t$.
	Conversely, let $f(y)\coloneqq\E U_t^{\rm sink}(y)$. For $y\ne0$, averaging
	\eqref{eq:signed-count} identifies $(Pf)(y)-f(y)$ with the expected number of
	active particles minus unfilled holes at $y$ in the sink process. The same
	comparison therefore gives
	\[
		(Pf)(y)-f(y)\leq\E[A_t(y)-H_t(y)]=0\,.
	\]
	Moreover $f(0)=0$, while finite propagation gives $f(y)=m_t$ when $|y|>t$.
	Stop the walk when it hits the origin or leaves the ball of radius $t$.
	The optional stopping theorem gives $f(y)\geq m_t\P_y(T_0=\infty)$. Averaging
	over the neighbors of the origin and using
	$\frac1{2d}\sum_{y\sim0}\P_y(T_0=\infty)=1/g$ leaves
	\begin{equation}\label{eq:nearest-sink-mean}
		m_t/g\leq\E\lambda\leq m_t\,.
	\end{equation}

	\smallskip
	\noindent\emph{Step 3.} We prove the second bound in
	\eqref{eq:nearest-odometer-moment} and show that $h_t\downarrow0$. We first
	show that
	\begin{equation}\label{eq:nearest-sink-tail}
		\P(\lambda<\E\lambda/2)\leq Ce^{-cm_t}\,.
	\end{equation}
	Let $\bar\lambda\coloneqq\E[\lambda\mid\eta]$. Raising $\eta(z)$ by one,
	for $z\ne0$, changes the expected number of entries into the sink by at most
	\[
		\E_z\sum_{n<T_0}P(X_n,0)=q_0(z)\,.
	\]
	The one-particle comparison therefore bounds the oscillation of
	$\bar\lambda$ in $\eta(z)$ by $2q_0(z)$; the oscillation at the origin is
	zero. Since $\sum_zq_0(z)^2<\infty$, the bounded difference inequality and
	\eqref{eq:nearest-sink-mean} give
	$\P(\bar\lambda<3\E\lambda/4)\leq Ce^{-cm_t^2}$. Next reveal the instructions
	the particles read. The resulting martingale $\lambda-\bar\lambda$ has increments
	bounded by a constant and
	predictable quadratic variation at most
	\[
		\sum_yU^{\rm sink}_t(y)\,\frac1{2d}\sum_{z\sim y}q_0(z)^2\,.
	\]
	The weights in this sum add up to $\sum_zq_0(z)^2<\infty$. Minkowski's
	inequality and the sink version of \eqref{eq:nearest-odometer-moment}
	therefore bound its $L^{r/2}$ norm by $C(m_t+r)$.
	For $m_t$ larger than a constant, let $r$ be a sufficiently small fixed
	multiple of $m_t$, rounded down to an integer. Lemma~\ref{lem:bernstein} then
	gives
	\[
		\bigl(\E|\lambda-\bar\lambda|^r\bigr)^{1/r}
		\leq C\sqrt{r(m_t+r)}+Cr\leq m_t/(8g)\,,
	\]
	and $\E\lambda/4\geq m_t/(4g)$, so Markov's inequality bounds
	$\P(|\lambda-\bar\lambda|\geq\E\lambda/4)$ by
	$2^{-r}\leq Ce^{-cm_t}$. Since
	$\lambda\geq\bar\lambda-|\lambda-\bar\lambda|$, the two tail bounds prove
	\eqref{eq:nearest-sink-tail}, after enlarging $C$ to cover the bounded
	range of $m_t$.

	Reveal chronologically the finitely many instructions which can affect the
	sink through round $t$, each only after the preceding history determines that
	it is read. Multiply the indicator that no revealed instruction has entered
	the sink by the exponential of the accumulated conditional entrance
	probabilities. This is a nonnegative supermartingale because
	$(1-a)e^a\leq1$, and those probabilities sum to $\lambda$. Hence
	$\E[\one\{J=0\}e^\lambda]\leq1$, and splitting along
	\eqref{eq:nearest-sink-tail} and \eqref{eq:nearest-sink-mean} gives
	\[
		\P(J=0)\leq Ce^{-cm_t}+e^{-m_t/(2g)}\leq Ce^{-cm_t}\,.
	\]
	With \eqref{eq:nearest-hole-sink} and $h_t\leq p\leq1/4$ this proves the
	second bound in \eqref{eq:nearest-odometer-moment}. The sequence $h_t$
	decreases because holes never reappear. If $m_t$ stays bounded, then
	\eqref{eq:nearest-one-point-identities} gives $\sum_t h_t<\infty$, so
	$h_t\to0$. If
	$m_t\to\infty$, then \eqref{eq:nearest-hole-sink} and the bound
	$\P(J=0)\leq Ce^{-cm_t}$ give $h_t\to0$.
\end{proof}

\begin{proposition}\label{prop:nearest-two-hole}
	There is $C<\infty$, depending only on $d$, such that, for every $t\geq0$ and
	distinct $x,z\in\Z^d$,
	\begin{equation}\label{eq:nearest-two-hole}
		\P\bigl(H_t(x)=H_t(z)=1\bigr)
		\leq Ch_t^2\exp\bigl\{C\log(1/h_t)(1+|x-z|)^{4-d}\bigr\}\,.
	\end{equation}
\end{proposition}

\begin{proof}
	Let $h\coloneqq h_t$.
	\smallskip
	\noindent\emph{Step 1.} We prove \eqref{eq:nearest-one-particle-bound} and
	\eqref{eq:nearest-factor-supermartingale}.
	By finite propagation, only initial values within distance $t$ of $x$ or
	$z$ can affect the two holes. Reveal these values first. Then reveal only the
	instructions needed to determine the two hole counts, in chronological order.
	Reveal each instruction only after the history up to its use determines its
	site and index, and break ties in a fixed order. Since $\eta^+\leq1$, the number of variables
	revealed has a deterministic upper bound. Immediately before revealing an
	instruction, the past determines its site and index, while its destination is
	conditionally uniform by Lemma~\ref{lem:exposure}.

	Condition on the variables already revealed, and fix a site $v\in\Z^d$. Let
	$a_0$ and $a_1$ be the conditional probabilities that the hole at $x$ is
	unfilled, without and with one added particle at $v$, respectively.
	Couple the two remaining evolutions as in
	Lemma~\ref{lem:one-particle}, using the construction in which each particle has
	its own walk. First condition on the evolution without the added particle
	through round $t$. Under this conditioning, the event defining $a_0$ is fixed.
	Whenever the one-particle discrepancy is active, it is represented either by
	the added particle or by a particle already settled in the conditioned
	evolution. Its next step is therefore unused by that evolution and remains
	independent and uniform. While the discrepancy is an unmatched hole, it does not move. After these waiting periods are removed, its path is a simple random
	walk from $v$. Its conditional probability of visiting $x$ by round $t$ is therefore at most $q_x(v)$. If the hole at $x$ is unfilled in the conditioned
	evolution and the discrepancy does not visit $x$, then it remains unfilled after the particle is added. Averaging over the conditioned evolution and using Lemma~\ref{lem:one-particle} for the upper bound gives
	\begin{equation}\label{eq:nearest-one-particle-bound}
		(1-q_x(v))\,a_0\leq a_1\leq a_0\,.
	\end{equation}
	The same bound holds with $z$ in place of $x$.

	Before revealing each remaining variable, let $a$ and $b$ be the conditional
	probabilities of the two hole events. Choose a factor $f\geq1$, determined by
	the variables already revealed, such that the conditional mean of their
	product after revealing the variable is at most $fab$. If either probability
	is zero, take $f=1$. Let $D$ be the product of these factors. Dividing
	the product of the two conditional probabilities by the factors accumulated so
	far gives a nonnegative supermartingale. Both initial probabilities equal
	$h$, while their product at the end is the indicator of the two-hole event.
	Therefore
	\begin{equation}\label{eq:nearest-factor-supermartingale}
		\E\frac{\one\{H_t(x)=H_t(z)=1\}}{D}\leq h^2\,.
	\end{equation}

	\smallskip
	\noindent\emph{Step 2.} We show that the factors from the initial values
	multiply to at most $\exp\{C\sum_vq_x(v)q_z(v)\}$. Let
	\[
		k_y\coloneqq\frac1{2d}\sum_{v\sim y}q_x(v)q_z(v)\,.
	\]
	Equation \eqref{eq:nearest-bubble-decay} gives
	\begin{equation}\label{eq:nearest-kernel-sum}
		\sum_yk_y=\sum_vq_x(v)q_z(v)
		\leq C(1+|x-z|)^{4-d}\,.
	\end{equation}

	Consider the initial value at $y\notin\{x,z\}$. If the conditional
	probability of either hole is zero when $\eta(y)=0$, then
	\eqref{eq:nearest-one-particle-bound} makes that probability zero for all
	three possible values. In that case, take $f=1$. Otherwise, normalize each
	conditional probability by its value when $\eta(y)=0$. Sampling $\eta(y)$
	then gives ratios $1+\delta$ and $1+\eps$. The maximum principle gives
	$q_x(y)\leq1-1/g$ and $q_z(y)\leq1-1/g$.
	Equation \eqref{eq:nearest-one-particle-bound} gives
	\[
		\E(1+\delta)\geq1-p,\qquad
		\E|\delta|\leq Cpq_x(y)
		\quad\text{and}\quad
		\E|\delta\eps|\leq Cpq_x(y)q_z(y)\,.
	\]
	The corresponding bounds hold with $\eps$ and $q_z(y)$ in place of
	$\delta$ and $q_x(y)$.
	The ratio that determines the factor is
	\[
		\frac{\E[(1+\delta)(1+\eps)]}
		{\E(1+\delta)\E(1+\eps)}
		=1+\frac{\Cov(\delta,\eps)}
		{\E(1+\delta)\E(1+\eps)}
		\leq e^{Cq_x(y)q_z(y)}\,.
	\]
	Thus we may take $f\leq\exp\{Cq_x(y)q_z(y)\}$.

	The event $H_t(x)=1$ requires $\eta(x)=-1$. The average conditional
	probability of a hole at $z$ is at least $1-q_z(x)$ times its value under
	$\eta(x)=-1$, because
	\[
		p+(1-2p)(1-q_z(x))+p(1-q_z(x))^2\geq1-q_z(x)\,.
	\]
	Thus the factor at $x$ is at most
	$(1-q_z(x))^{-1}\leq e^{gq_z(x)}$. Interchanging $x$ and $z$ gives the
	corresponding bound for the factor at $z$. The
	terms $v=x$ and $v=z$ in $\sum_vq_x(v)q_z(v)$ show that the two factors
	multiply to at most $\exp\{C\sum_vq_x(v)q_z(v)\}$.

	\smallskip
	\noindent\emph{Step 3.} We show that the factors from the instructions
	multiply to at most $\exp\{C\sum_yk_yU_t(y)\}$. Consider
	an instruction based at $y$ as we reveal it. If deleting the particle
	which reads it leaves either conditional probability zero, then restoring the
	particle leaves that probability zero, and we take $f=1$. If changing the
	instruction cannot affect one of the two holes, then the corresponding
	conditional probability is constant and again we take $f=1$.

	In the remaining case, normalize the two conditional probabilities by their
	values after deleting the particle which reads the instruction. For a destination $v$, the
	normalized values are $1-\alpha_v$ and $1-\beta_v$.
	Equation \eqref{eq:nearest-one-particle-bound} gives $0\leq\alpha_v\leq q_x(v)$ and
	$0\leq\beta_v\leq q_z(v)$. Let $\overline\alpha$ and $\overline\beta$ be
	the averages of $\alpha_v$ and $\beta_v$ over the neighbors of $y$. The
	corresponding neighbor averages of $q_x(v)$ and $q_z(v)$ are at most
	$1-1/g$. Thus both $1-\overline\alpha$ and $1-\overline\beta$ are at least
	$1/g$. Expanding the product gives
	\[
		1+\frac{\frac1{2d}\sum_{v\sim y}\alpha_v\beta_v
		-\overline\alpha\,\overline\beta}
		{(1-\overline\alpha)(1-\overline\beta)}
		\leq e^{Ck_y}\,.
	\]
	Thus we may take $f\leq e^{Ck_y}$. We reveal at most $U_t(y)$ instructions
	based at $y$, so Steps~2 and~3 give
	\begin{equation}\label{eq:nearest-factor-product}
		D\leq
		\exp\Bigl\{C\sum_vq_x(v)q_z(v)+C\sum_yk_yU_t(y)\Bigr\}\,.
	\end{equation}

	\smallskip
	\noindent\emph{Step 4.} We prove \eqref{eq:nearest-two-hole}. Forcing every site within
	distance $t$ of the origin to be a hole shows that $h>0$. Let
	$\ell\coloneqq\log(1/h)$. The event $H_t(0)=1$ requires $\eta(0)=-1$, so
	$h\leq p\leq1/4$ and $\ell\geq\log4$. Fix $A\geq1$, to be chosen after
	\eqref{eq:nearest-two-hole-split}. Equations
	\eqref{eq:nearest-factor-supermartingale}
	and \eqref{eq:nearest-factor-product} give
	\begin{align}
		\P\bigl(H_t(x)=H_t(z)=1\bigr)
		&\leq h^2\exp\Bigl\{CA\ell\sum_yk_y\Bigr\}\notag\\
		&\quad+h\,\P\left(
			\sum_yk_yU_t(y)\geq A\ell\sum_yk_y
			\,\middle|\,H_t(z)=1\right)\,.
		\label{eq:nearest-two-hole-split}
	\end{align}

	Let $r\coloneqq\lceil4\ell\rceil$. H\"older's inequality, Minkowski's
	inequality, translation invariance, and
	Lemma~\ref{lem:nearest-one-point} give
	\[
		\Bigl(\E\Bigl[\Bigl(\sum_yk_yU_t(y)\Bigr)^r
		\Bigm|H_t(z)=1\Bigr]\Bigr)^{1/r}
		\leq h^{-1/(2r)}
		\Bigl(\E\Bigl(\sum_yk_yU_t(y)\Bigr)^{2r}\Bigr)^{1/(2r)}
		\leq C\ell\sum_yk_y\,.
	\]
	Here $h^{-1/(2r)}\leq e^{1/8}$. Choose $A$ sufficiently large. Markov's
	inequality then bounds the conditional probability in
	\eqref{eq:nearest-two-hole-split} by $e^{-r}\leq h^4$. Hence
	\[
		\P\bigl(H_t(x)=H_t(z)=1\bigr)
		\leq h^2e^{CA\ell\sum_yk_y}+h^5
		\leq Ch^2e^{C\ell\sum_yk_y}\,.
	\]
	The estimate \eqref{eq:nearest-kernel-sum} proves
	\eqref{eq:nearest-two-hole}.
\end{proof}

\begin{lemma}\label{lem:nearest-close-pair}
	For every $t\geq0$ and distinct $x,z\in\Z^d$,
	\begin{equation}\label{eq:nearest-close-pair}
		\P\bigl(H_t(x)=H_t(z)=1\bigr)\leq2p\,h_t\,.
	\end{equation}
\end{lemma}

\begin{proof}
	On the two-hole event no particle has reached $z$. Change only
	$\eta(z)=-1$ to $0$. The evolution off $z$ is unchanged, and the image
	still has a hole at $x$. The map is injective and multiplies the weight by
	$(1-2p)/p$, so the two-hole probability is at most
	$\frac p{1-2p}\,h_t\leq2p\,h_t$.
\end{proof}

\begin{proof}[Proof of Theorem~\ref{thm:nearest-counterexample}]
	\smallskip
	\noindent\emph{Step 1.} We show that holes with no other hole within
	distance $4R$ have density at least $3h_t/4$. Let $h\coloneqq h_t$ and
	$\ell\coloneqq\log(1/h)$. Choose $L$ so large that
	$C(1+L)^{4-d}\leq1/4$, and then choose $p,b>0$ so small that
	$CL^dp+Cb^d\leq1/8$. Let $p_d\coloneqq p$ in
	\eqref{eq:nearest-counterexample-law}, and let
	$R\coloneqq\lfloor bh^{-1/d}\rfloor$.
	Lemma~\ref{lem:nearest-one-point} gives $h\to0$,
	so eventually $R\geq2d$ and $\ell^{1/(d-4)}<R$.

	A hole is isolated when no other hole lies within distance $4R$. The
	density of the holes that are not isolated is at most the sum of the two-hole
	probabilities over $0<|x|\leq4R$. Lemma~\ref{lem:nearest-close-pair}
	makes the sum over $|x|<L$ at most $CL^dph$. By the choice of $L$,
	Proposition~\ref{prop:nearest-two-hole}
	bounds each summand with $L\leq|x|<\ell^{1/(d-4)}$ by $Ch^{7/4}$. There
	are at most $C\ell^{d/(d-4)}$ such sites, so this range contributes $o(h)$.
	For $|x|\geq\ell^{1/(d-4)}$, the exponential factor in
	\eqref{eq:nearest-two-hole} is bounded. Each summand is at most $Ch^2$, and
	there are at most $CR^d\leq Cb^dh^{-1}$ sites. Thus this range contributes at
	most $Cb^dh$. Combining the three ranges bounds the density of holes which are
	not isolated by $(CL^dp+Cb^d+o(1))h\leq h/4$ for all large $t$.
	Isolated holes therefore have density at least $3h/4$ for all large $t$.

	\smallskip
	\noindent\emph{Step 2.} We show that isolated holes whose radius-$2R$ ball
	contains at most one active site have density at least $h/4$.
	Their balls of radius $2R$ are disjoint. Every active site in one of these
	balls sends unit mass to its center. Since
	$\P(A_t(0)>0)\leq\E A_t(0)=h$, applying the mass transport principle shows that the isolated holes
	with at least two active sites in their ball have density at most $h/2$.
	Isolated holes with at most one
	such site therefore have density at least $h/4$.

	\smallskip
	\noindent\emph{Step 3.} We prove \eqref{eq:nearest-counterexample}.
	Fix one of the holes from Step~2. If its radius-$2R$ ball contains an active
	site $a$, then choose the closed coordinate orthant based at the hole and
	opposite $a$. Use the positive sign at every tied coordinate. If the ball
	contains no active site, then choose a fixed coordinate orthant. Its
	intersection with the ball of radius $R$ contains at least $cR^d$ sites.

	Let $y$ be the hole and let $w$ lie in this intersection. If the active site
	$a$ exists, then
	$|w_i-a_i|=|w_i-y_i|+|a_i-y_i|$ for every coordinate $i$. Since $a\ne y$,
	this gives $|w-a|>|w-y|$. Every active site $v$ farther than $2R$ from $y$
	satisfies $|w-v|\geq|v-y|-R>R\geq|w-y|$. Thus every site in the chosen
	intersection is strictly closer to $y$ than to every active particle. The
	balls of radius $R$ about distinct isolated holes are disjoint, so mass
	transport gives
	\[
		\P(\text{the origin is closer to an unfilled hole than to an active
		particle at time }t)\geq\frac h4\,cR^d\geq c\,,
	\]
	since $R^d\geq cb^dh^{-1}$ eventually. Unfilled holes and active
	particles exist at every time almost surely, since both densities are
	positive and the translations of $\Z^d$ act ergodically. Taking the
	lower limit proves
	\eqref{eq:nearest-counterexample}.
\end{proof}

\section{Subcritical phase}\label{sec:subcritical}

Below the critical density a tagged particle settles almost surely, and the
question is how fast. Both bounds in \eqref{eq:sharpness} are exponential in
the size of the range of that particle's walk, and the Donsker--Varadhan
estimate converts them into stretched exponentials in $t$. The lower bound
needs only independence. For the upper bound, we tilt the initial law to raise
the density of particles and use Lemma~\ref{lem:product} to control the
resulting covariance.

Let $\eta=(\eta(x))_{x\in\Z^d}$ be i.i.d.\ and integer-valued, with a finite
first moment, negative mean, $\P(\eta(0)>0)>0$, and
$\E e^{\theta\eta(0)}<\infty$ for some $\theta>0$. Fix $k\geq1$ with
$\P(\eta(0)=k)>0$ and condition on $\eta(0)=k$. Follow particle $1$ among the
$k$ particles at the origin. Let $X_0=0,X_1,\ldots$ be its assigned walk and
$R_t=\{X_0,\ldots,X_t\}$ its range through time $t$. We bound
$\P(\tau_1>t)$ above and below by exponentials in $|R_t|$.

\begin{lemma}\label{lem:range-lower}
	For every $k\geq1$ with
	$\P(\eta(0)=k)>0$ and every $t\geq0$,
	\[
		\P(\tau_1>t\mid\eta(0)=k)\geq\E_0\bigl[\P(\eta(0)\geq0)^{|R_t|-1}\bigr]\,,
	\]
	and consequently
	$S_t\geq\E[\eta(0)^+]\,\E_0[\P(\eta(0)\geq0)^{|R_t|-1}]$.
\end{lemma}

\begin{proof}
	Conditioning on $\eta(0)=k\geq1$ leaves no hole at the origin. If no other
	site of $R_t$ carries a hole initially, then particle $1$ meets none and
	cannot settle by time $t$. Conditionally on the assigned walk these
	requirements concern distinct sites other than the origin, so they are
	independent and each has probability $\P(\eta(0)\geq0)$. The bound is
	uniform in $k$, so \eqref{eq:S-expand} gives the stated lower bound on $S_t$.
\end{proof}

For the upper bound we raise the density of particles. Tilting the law of
$\eta(0)$ by $e^{\lambda\eta(0)}$ raises its mean. Differentiating
$\P(\tau_1>t)$ in $\lambda$ then produces a covariance between the event
$\{\tau_1>t\}$ and the number of unfilled holes left in the range of the walk.
Whenever particle $1$ arrives at a site without settling, all holes present
there are filled in that round, and holes never reappear. Hence the product of
the survival indicator and the number of unfilled holes in the range is zero.
Their covariance is therefore minus the product of their expectations.
Integrating the resulting differential inequality over $\lambda$ gives the
factor $e^{-a|R_t|}$.

Choose $0<\lambda_1<\theta$ so small that, for $0\leq\lambda\leq\lambda_1$, the
law
\[
	\P_\lambda(\eta(0)=j)
	=\frac{e^{\lambda j}\P(\eta(0)=j)}{\E e^{\lambda\eta(0)}}\,,
\]
has nonpositive mean. Such a choice exists by continuity. Let $\E_\lambda$
be the expectation for the corresponding product law, with the walks and
uniform variables unchanged, let $\Cov_\lambda$ denote its covariance, and let
$\delta(\lambda)\coloneqq-\E_\lambda\eta(0)$.

\begin{lemma}\label{lem:product}
	Fix $0\leq\lambda\leq\lambda_1$ and finitely many sites carrying the tilted
	law, and condition on the remaining randomness. Let $F$ and $Z$ be functions
	of the counts, walks, and uniform variables at these sites, each
	invariant under relabeling the particles at a site. Suppose $F$ takes values in $[0,1]$ and
	does not decrease when a particle is added. Suppose $Z$ is nonnegative and
	integrable, does not increase when a particle is added, and changes by at
	most one when a single particle is added or deleted. Then
	\begin{equation}\label{eq:cov}
		|\Cov_\lambda(F,Z)|\leq3\,\partial_\lambda\E_\lambda F\,.
	\end{equation}
\end{lemma}

\begin{proof}
	Fix $\lambda$ and suppress the subscript until Step~3.

	Adding a particle at a site means raising its count by one, which cancels a
	hole when the count is negative and adds a particle with a fresh independent
	walk and uniform variables when it is not. Deleting a particle is the inverse
	operation. Consider first a single random site and let $Y$ be its signed
	count. For every $k\in\Z$, prescribe signed count $k$ at this site and, when
	$k>0$, give its $k$ particles independent walks and uniform variables. Let $f(k)$
	and $z(k)$ be the resulting expectations of $F$ and $Z$. For $k$ in the
	support of $Y$, these equal $\E[F\mid Y=k]$ and $\E[Z\mid Y=k]$.
	The conditional covariance identity splits $\Cov(F,Z)$ into
	$\Cov(f(Y),z(Y))$ and $\E[\Cov(F,Z\mid Y)]$. Here $f$ is nondecreasing while
	$z$ is nonincreasing. Moreover, $z$ is one-Lipschitz because adding or deleting one
	particle changes $Z$ by at most one.

	\smallskip
	\noindent\emph{Step 1.} We prove
	\begin{equation}\label{eq:cov-counts}
		|\Cov(f(Y),z(Y))|\leq\Cov(f(Y),Y).
	\end{equation}
	Let $Y_1$ be an independent copy of $Y$. Then
	\[
		2\bigl|\Cov(f(Y),z(Y))\bigr|
		=\Bigl|\E\bigl[(f(Y)-f(Y_1))(z(Y)-z(Y_1))\bigr]\Bigr|
		\leq\E\bigl[|f(Y)-f(Y_1)|\,|Y-Y_1|\bigr]\,,
	\]
	where the equality is the covariance identity and the inequality uses that
	$z$ is one-Lipschitz. Since $f$ is nondecreasing,
	$\E[|f(Y)-f(Y_1)|\,|Y-Y_1|]=2\Cov(f(Y),Y)$.

	\smallskip
	\noindent\emph{Step 2.} We prove
	\begin{equation}\label{eq:cov-particle-randomness}
		\E|\Cov(F,Z\mid Y)|\leq2\Cov(f(Y),Y).
	\end{equation}
	For $k\leq0$, the site carries no particles, so the conditional covariance
	vanishes. Fix $k\geq1$. Let $\mathcal H_0$ be the trivial $\sigma$-field and,
	for $1\leq i\leq k$, let $\mathcal H_i$ be generated by the simple random
	walks and tie-breaking variables of particles $1,\ldots,i$.
	Writing $D_iF=\E[F\mid\mathcal H_i]-\E[F\mid\mathcal H_{i-1}]$ and likewise for
	$Z$, the martingale covariance decomposition gives
	\[
		\Cov(F,Z\mid Y=k)=\sum_{i=1}^k\E\bigl[D_iF\,D_iZ\bigr]\,,
		\qquad\text{so}\qquad
		\bigl|\Cov(F,Z\mid Y=k)\bigr|\leq\sum_{i=1}^k\E\bigl[|D_iF|\,|D_iZ|\bigr]\,.
	\]
	Let $F_i$ be the value after particle $i$ is deleted. Symmetrization and
	monotonicity give
	\[
		\E|D_iF|
		\leq2\E(F-F_i)
		=2\bigl(f(k)-f(k-1)\bigr)\,.
	\]
	The deletion comparison also gives $|D_iZ|\leq1$, and hence
	\[
		\bigl|\Cov(F,Z\mid Y=k)\bigr|\leq2k\bigl(f(k)-f(k-1)\bigr)\,.
	\]
	The summation by parts formula gives
	\[
		\Cov(f(Y),Y)=\sum_j(f(j)-f(j-1))\,
		\E[(Y-\E Y)\one\{Y\geq j\}].
	\]
	Every summand is nonnegative, and for $j\geq1$ the expectation in it is at
	least $j\P(Y=j)$ because $\E Y\leq0$. Comparing term by term shows that the
	mean of the absolute conditional covariances is at most
	$2\Cov(f(Y),Y)$. This proves \eqref{eq:cov-particle-randomness}.
	Together with \eqref{eq:cov-counts}, it gives
	$|\Cov(F,Z)|\leq3\Cov(f(Y),Y)$ for one site.

	\smallskip
	\noindent\emph{Step 3.} We prove \eqref{eq:cov} for the finite family of
	tilted sites.
	For finitely many sites $x_1,\ldots,x_N$, reveal the counts, walks, and uniform variables one site at a time. The martingale covariance decomposition writes $\Cov_\lambda(F,Z)$ as a sum of conditional one-site covariances. Conditional averaging over the unrevealed sites preserves the required monotonicity and Lipschitz properties,
	so Steps~1 and~2 give
	\[
		|\Cov_\lambda(F,Z)|
		\leq3\sum_i\Cov_\lambda(F,\eta(x_i)).
	\]
	Finally the tilted law has density proportional
	to $\prod_{i=1}^Ne^{\lambda\eta(x_i)}$ on these coordinates, so
	$\partial_\lambda\E_\lambda F=\sum_{i=1}^N\Cov_\lambda(F,\eta(x_i))$, the
	differentiation under the expectation being justified by $\lambda_1<\theta$.
\end{proof}

\begin{theorem}[Subcritical settling time]\label{thm:subcritical}
	Under the assumptions and notation above, let
	\[
		a\coloneqq\frac13\int_0^{\lambda_1}\delta(\lambda)\,d\lambda\,.
	\]
	Then $a>0$ and, for every $t\geq0$ and every $k\geq1$ in the support of
	$\eta(0)$,
	\begin{equation}\label{eq:range-upper}
		\P\bigl(\tau_1>t\mid\eta(0)=k,\ X_0,\ldots,X_t\bigr)
		\leq\exp\Bigl\{\tfrac13\int_0^{\lambda_1}\E_\lambda|\eta(0)|
		\,d\lambda\Bigr\}\,e^{\lambda_1(k-1)/3-a|R_t|}\,.
	\end{equation}
	Consequently $S_t\leq C\E_0e^{-a|R_t|}$.
\end{theorem}

\begin{proof}
	Fix $X_0,\ldots,X_t$. Prescribe the $k$ particles at the origin and
	temporarily condition on all their walks and uniform variables. Remove
	particle $1$, and let $Z$ be the number of unfilled holes in $R_t$ at time
	$t$ in the resulting process. Let $F$ indicate that particle $1$ has not
	settled by time $t$ when it is put back. Let $\widehat\E_\lambda$ denote
	expectation when the origin carries the remaining $k-1$ prescribed particles
	and every other site has the tilted law.

	\smallskip
	\noindent\emph{Step 1.} We prove
	\begin{equation}\label{eq:range-hole-count}
		FZ=0,
		\qquad
		\widehat\E_\lambda Z\geq\delta(\lambda)|R_t|
		-\E_\lambda|\eta(0)|-(k-1).
	\end{equation}
	On $\{F=1\}$, particle $1$ fills no hole, so removing it does not change the
	holes through time $t$. Every hole at a site it visits is filled by the end of
	that round. Thus no site of $R_t$ carries an unfilled hole at time $t$, and
	$FZ=0$.

	Under $\E_\lambda$, \eqref{eq:activity-holes} gives
	$\E_\lambda H_t(0)\geq\delta(\lambda)$, so by translation invariance the
	expected number of unfilled holes in $R_t$ at time $t$ is at least
	$\delta(\lambda)|R_t|$. If the origin instead has value $j$, then passing from that value to the
	prescribed configuration takes at most $|j|+k-1$ single-particle changes,
	each changing $Z$ by at most one by Lemma~\ref{lem:one-particle}. Averaging
	over $j$ proves the second inequality in \eqref{eq:range-hole-count}.

	\smallskip
	\noindent\emph{Step 2.} We prove
	\begin{equation}\label{eq:survival-differential}
		3\varphi'(\lambda)\geq
		\bigl(\delta(\lambda)|R_t|-\E_\lambda|\eta(0)|-(k-1)\bigr)
		\varphi(\lambda).
	\end{equation}
	Let $\varphi(\lambda)\coloneqq\widehat\E_\lambda F$.
	Both $F$ and $Z$ depend only on the finitely many sites within distance $t$
	of $R_t$. They are invariant under relabeling the particles at those sites.
	By Lemma~\ref{lem:tagged-monotonicity}, $F$ is nondecreasing when a particle
	is added; by Lemma~\ref{lem:one-particle}, $Z$ is nonincreasing and changes by
	at most one. Thus Lemma~\ref{lem:product} applies. The identity $FZ=0$ shows
	that the covariance under $\widehat\E_\lambda$ is
	$-\varphi(\lambda)\widehat\E_\lambda Z\leq0$. Equation \eqref{eq:cov} and
	\eqref{eq:range-hole-count} now give \eqref{eq:survival-differential}.

	\smallskip
	\noindent\emph{Step 3.} We prove \eqref{eq:range-upper} and
	$S_t\leq C\E_0e^{-a|R_t|}$.
	Since $\delta(0)=-\E\eta(0)>0$ and $\delta$ is continuous, the constant
	$a$ in the theorem is positive.
	If $\varphi(0)=0$, then \eqref{eq:range-upper} is immediate. Otherwise,
	integrate the logarithmic derivative in
	\eqref{eq:survival-differential} over $[0,\lambda_1]$ and use
	$\varphi(\lambda_1)\leq1$. This gives \eqref{eq:range-upper}. Averaging over
	the conditioned variables and then summing over $k$ in \eqref{eq:S-expand}
	gives the last assertion because
	$k e^{\lambda_1k/3}\leq C e^{\lambda_1k}$ and
	$\E e^{\lambda_1\eta(0)}<\infty$. Therefore
	$S_t\leq C\E_0e^{-a|R_t|}$.
\end{proof}

\begin{proof}[Proof of Theorem~\ref{thm:subcritical-tail}]
	Lemma~\ref{lem:range-lower} bounds $S_t$ below by a positive constant times
	$\E_0e^{-c|R_t|}$ for a constant $c>0$, while Theorem~\ref{thm:subcritical}
	bounds it above by $C\E_0e^{-a|R_t|}$. The Donsker--Varadhan estimate
	for the range \citep[Theorem~1]{DonskerVaradhan} gives \eqref{eq:sharpness}.
	Summing the upper bound in \eqref{eq:sharpness} and using
	\eqref{eq:transport} gives $\E U_\infty(0)<\infty$. Also $S_t\to0$, so the
	expected number of particles starting at the origin which never settle is
	zero; translation invariance and countability extend this to every particle.
	Moreover, $H_t(0)\downarrow H_\infty(0)$ and
	$0\leq H_t(0)\leq\eta(0)^-$. The dominated convergence theorem and
	\eqref{eq:activity-holes} give
	\[
		\E H_\infty(0)=-\E\eta(0),
		\qquad
		\E[H_m(0)-H_\infty(0)]=S_m.
	\]
\end{proof}

\begin{remark}[The constant for the parking model]\label{rem:tilt}
	For the parking model, let
	\[
		\lambda_1\coloneqq\frac12\log\frac{1-p}{p}.
	\]
	Since
	$\frac d{d\lambda}\log\E e^{\lambda\eta(0)}=-\delta(\lambda)$,
	\[
		a=-\frac16\log(4p(1-p))\asymp(1-2p)^2
		\qquad\text{as }p\uparrow\tfrac12.
	\]
\end{remark}

\begin{remark}[Stabilization and decay of correlations]\label{rem:stabilization-mixing}
	The identities
	$\E[U_\infty(0)-U_m(0)]=\sum_{t\geq m}S_t$ and
	$\E[H_m(0)-H_\infty(0)]=S_m$ give
	\[
		\E[U_\infty(0)-U_m(0)]+\E[H_m(0)-H_\infty(0)]
		\leq Ce^{-cm^{d/(d+2)}}.
	\]
	For bounded local functions $F$ and $G$ with disjoint supports, let
	\[
		m\coloneqq\bigl\lfloor\operatorname{dist}(\operatorname{supp}F,
		\operatorname{supp}G)/3\bigr\rfloor\,.
	\]
	Finite propagation makes $F(U_m,H_m)$ and $G(U_m,H_m)$ independent.
	Markov's inequality and a union bound therefore give
	\[
		|\Cov(F(U_\infty,H_\infty),G(U_\infty,H_\infty))|
		\leq C\|F\|_\infty\|G\|_\infty
		\bigl(|\operatorname{supp}F|+|\operatorname{supp}G|\bigr)
		e^{-cm^{d/(d+2)}}.
	\]
\end{remark}

\begin{remark}[A finite first moment does not suffice]\label{rem:first-moment}
	Fix $\beta>1$ and let $\eta(0)$ be supported on $\{-1\}$ together with
	the positive integers. Let $q>0$ be small enough that
	$\E\eta(0)<0$, let $\P(\eta(0)=k)\coloneqq qk^{-\beta-1}$ for $k\geq1$,
	and let the remaining mass be at $-1$. By time $t$ the
	particles at the origin can fill at most $Ct^d$ holes. On the event
	$\{\eta(0)\geq2Ct^d\}$, at least half remain active. Thus
	\eqref{eq:S-expand} gives $S_t\geq c(1+t)^{-d(\beta-1)}$, which is not
	summable when $1<\beta\leq1+1/d$. Thus a finite first moment alone does not
	ensure $\E U_\infty(0)<\infty$.
\end{remark}

\section{Near criticality}\label{sec:near}

For an initial law of mean $-\delta$, $S_t$ differs by at most $C\delta$ from
its critical value, so it follows the critical decay until that decay reaches
order $\delta$. Two estimates are needed that the
critical theory does not supply. The first is a time beyond which $S_t$
is summable; it comes from the tilt of Section~\ref{sec:subcritical},
available only up to a parameter of order $\delta$, together with a lower tail
estimate for the range. The second is a bound on the divisible odometer of a
scenery with a small negative drift. This becomes an optimization over the expected duration of the stopping rule: a rule of expected duration $M$ collects at most the critical growth at time $M$, but incurs a drift cost $\delta M$. The delicate point is that such a rule may
run much longer than $M$.

For the rest of this section, let $\eta_\delta$ satisfy the assumptions of
Theorem~\ref{thm:near}, and let $S_t^\delta$ be the expected number of particles
from the origin not settled by time $t$. Let $R_t$ be the range through time
$t$ of a simple random walk from the origin, independent of the model. Thus
\[
	\E U_\infty^\delta(0)=\sum_{t\geq0}S_t^\delta\,.
\]

\begin{lemma}\label{lem:near-tilt}
	For all sufficiently small $\delta>0$ and every $t\geq0$,
	\[
		S_t^\delta\leq C\E_0e^{-c\delta^2|R_t|}.
	\]
\end{lemma}

\begin{proof}
	Let $\psi_\delta(\lambda)\coloneqq\log\E e^{\lambda\eta_\delta(0)}$.

	\smallskip
	\noindent\emph{Step 1.} We prove that $\psi_\delta'$ has a unique positive
	zero $\lambda_\delta$ in a fixed neighborhood of zero and that
	$\lambda_\delta\asymp\delta$.
	The coupling assumption gives $\eta_\delta(0)\to\eta_0(0)$ in probability,
	and the uniform exponential moment makes
	$\psi_\delta''\to\psi_0''$ uniformly on a neighborhood of zero. Since
	$\psi_0''(0)=\Var\eta_0(0)>0$, on a fixed neighborhood of zero
	\[
		c\leq\psi_\delta''\leq C
	\]
	for every small $\delta$. As $\psi_\delta'(0)=-\delta$, there is a unique $\lambda_\delta>0$ in that fixed neighborhood with
	$\psi_\delta'(\lambda_\delta)=0$, and
	$\lambda_\delta\asymp\delta$. Also
	$\E\eta_\delta(0)^+\to\E\eta_0(0)^+>0$, so
	Theorem~\ref{thm:subcritical} applies.

	\smallskip
	\noindent\emph{Step 2.} We prove the bound in the statement.
	Take $\lambda_1=\lambda_\delta$ in Theorem~\ref{thm:subcritical}. In its
	notation,
	\[
		a_\delta=\frac13\int_0^{\lambda_\delta}-\psi_\delta'(s)\,ds
		=-\frac13\psi_\delta(\lambda_\delta)\asymp\delta^2.
	\]
	The exponential-moment assumption and $\lambda_\delta\asymp\delta$ bound
	uniformly the prefactor obtained when \eqref{eq:range-upper} is averaged over
	the particles at the origin. This proves the result.
\end{proof}

The terms in $\E U_\infty^\delta(0)$ begin to decay on a scale which diverges as
$\delta\downarrow0$. We use the next range bound only beyond an explicit
cutoff; before it we compare with the subcritical divisible odometer.

\begin{proposition}\label{prop:resolvent}
	There are $c,C>0$ such that, for $0<a\leq1$ and $t\geq1$,
	\begin{equation}\label{eq:resolvent-pointwise}
		\E_0e^{-a|R_t|}\leq C
		\begin{cases}
			\exp\{-ca^{2/3}t^{1/3}\}&d=1\,,\\
			\exp\{-c\sqrt{at/\log(t+2)}\}&d=2,\ t\geq e/a\,,\\
			\exp\{-c\sqrt{at}\}&d\geq3\,.
		\end{cases}
	\end{equation}
	Moreover, with $\Lambda=\log(e/a)$ and
	\begin{equation}\label{eq:range-threshold}
		T=C
		\begin{cases}
			a^{-2}\Lambda^3&d=1\,,\\
			a^{-1}\Lambda^3&d=2\,,\\
			a^{-1}\Lambda^2&d\geq3\,,
		\end{cases}
	\end{equation}
	the sum $\sum_{t>T}\E_0e^{-a|R_t|}$ is at most one.
\end{proposition}

\begin{proof}
	\smallskip
	\noindent\emph{Step 1.} We prove that, for every integer $m\geq2$,
	\begin{equation}\label{eq:range-lower-tail}
		\P_0(|R_t|<m)\leq C
		\begin{cases}
			\exp\{-ct/m^2\}&d=1\,,\\
			\exp\{-ct/[m\log(m+1)]\}&d=2\,,\\
			\exp\{-ct/m\}&d\geq3\,.
		\end{cases}
	\end{equation}
	Let $L_t(x)$ denote the number of visits to $x$ by time $t$. Then
	the Cauchy--Schwarz inequality and the local central limit theorem give
	\[
		\E_0|R_t|\geq\frac{(t+1)^2}{\E_0\sum_xL_t(x)^2}
		\geq c
		\begin{cases}
			\sqrt t&d=1\,,\\ t/\log(t+2)&d=2\,,\\ t&d\geq3\,.
		\end{cases}
	\]
	Writing $T_x$ for the first visit to $x$, the strong Markov property gives
	\[
		\E_0|R_t|^2
		=\sum_{x,y}\P_0(T_x\leq t,T_y\leq t)
		\leq2\sum_x\P_0(T_x\leq t)\sum_y\P_x(T_y\leq t)
		=2(\E_0|R_t|)^2.
	\]
	Thus for some
	\[
		n\leq C
		\begin{cases}
			m^2&d=1\,,\\ m\log(m+1)&d=2\,,\\ m&d\geq3\,,
		\end{cases}
	\]
	the Paley--Zygmund inequality gives $\P_0(|R_n|\geq m)\geq c$. The ranges traced within
	successive blocks of $n$ steps, translated to their starting points, are
	independent. If $|R_t|<m$, each block range has fewer than $m$ sites, so $\P_0(|R_t|<m)$ is at most
	$(1-c)^{\lfloor t/n\rfloor}$, proving \eqref{eq:range-lower-tail}.

	\smallskip
	\noindent\emph{Step 2.} We prove \eqref{eq:resolvent-pointwise}.
	For every integer $m\geq2$,
	\[
		\E_0e^{-a|R_t|}\leq\P_0(|R_t|<m)+e^{-am}\,.
	\]
	Take $m$ of order $(t/a)^{1/3}$ for $d=1$, of order $\sqrt{t/a}$ for
	$d\geq3$, and of order $\sqrt{t/[a\log(t+2)]}$ for $d=2$.
	In dimension two, $t\geq e/a$ ensures that
	$\log(m+1)\leq\log(t+2)$. These choices give
	\eqref{eq:resolvent-pointwise}. When the
	corresponding choice of $m$ lies outside $[2,t+1]$, the exponential factor
	on the right of \eqref{eq:resolvent-pointwise} is bounded below by a positive
	constant. Enlarging $C$ and using $\E_0e^{-a|R_t|}\leq1$ proves the estimate.

	\smallskip
	\noindent\emph{Step 3.} We prove
	$\sum_{t>T}\E_0e^{-a|R_t|}\leq1$.
	For $d\geq3$, comparison with an integral gives
	\[
		\sum_{t>T}\E_0e^{-a|R_t|}
		\leq Ca^{-1}(1+\sqrt{aT})e^{-c\sqrt{aT}},
	\]
	which is at most one for the value of $T$ in \eqref{eq:range-threshold} once
	the constant there is fixed sufficiently large. The substitution $s=v^3a^{-2}$ gives the
	same conclusion when $d=1$. For $d=2$, divide $(T,\infty)$ into the
	intervals $[2^jT,2^{j+1}T)$. Since the constant in
	\eqref{eq:range-threshold} is fixed, $\log T\leq C\Lambda$. Since
	$C\Lambda+j\leq C2^{j/3}\Lambda$,
	\[
		c\sqrt{\frac{a2^jT}{\log(2^{j+1}T+2)}}
		\geq c\,2^{j/3}\Lambda
	\]
	where the constant $c$ on the right of the last display can be chosen as large as desired by increasing the constant in \eqref{eq:range-threshold}. The $j$th interval contributes at most $2^jT\exp\{-c2^{j/3}\Lambda\}$, and these bounds sum to at most one.
\end{proof}

Let $\xi_\delta\coloneqq\eta_\delta+\delta$, so that $\xi_\delta$ has mean zero. For
mean-zero $\xi$, let
\[
	u_m(x;\xi)\coloneqq\sup_{\sigma\leq m}\E_x\sum_{j<\sigma}\xi(X_j).
\]
Let
\[
	\phi_d(s)\coloneqq
	\begin{cases}
		(s+1)^{(4-d)/4}&d\leq3\,,\\
		\log(s+2)&d\geq4.
	\end{cases}
\]

Let $\chi$ be a symmetric random sign. There are $a_0>0$ and a fixed symmetric Laplace variable $\zeta$ such that, for all sufficiently small $\delta$ and every convex $f$ for which the expectations are finite,
\begin{equation}\label{eq:near-convex}
	\E f(a_0\chi)\leq\E f(\xi_\delta(0))\leq\E f(\zeta).
\end{equation}
Indeed, $\E|\xi_\delta(0)-\eta_0(0)|\leq(C+1)\delta$ and
$\E\eta_0(0)^+>0$. Choose $a_0$ so small that, for $|t|\leq a_0$,
\[
	\E(\eta_0(0)-t)^+>\E(a_0\chi-t)^++2a_0.
\]
The same inequality without $2a_0$ holds with $\xi_\delta(0)$ in place of
$\eta_0(0)$ when $\delta$ is small. For $t\geq a_0$, $\E(a_0\chi-t)^+=0$; for $t\leq-a_0$, $\E(\xi_\delta(0)-t)^+\geq-t=\E(a_0\chi-t)^+$. Conversely, the
uniform exponential moment gives
\[
	\E(\xi_\delta(0)-t)^++\E(-\xi_\delta(0)-t)^+\leq Ce^{-ct}\qquad(t\geq0).
\]
For a symmetric Laplace variable of scale $b$,
$\E(\zeta-t)^+=(b/2)e^{-t/b}$ for $t\geq0$, so a fixed large $b$ dominates this bound. If $t<0$, the mean-zero property and the comparison at $-t$ give
\[
	\E(\xi_\delta(0)-t)^+
	=-t+\E(-\xi_\delta(0)+t)^+
	\leq-t+\E(\zeta+t)^+=\E(\zeta-t)^+.
\]
The functions $(x-t)^+$ characterize convex order, proving
\eqref{eq:near-convex}.

\begin{lemma}\label{lem:mean-horizon}
	Conditionally on $\xi_\delta$, let $\sigma$ be a bounded stopping time for the
	walk. If $M=\E\sigma$, then
	\[
		\E\sum_{j<\sigma}\xi_\delta(X_j)\leq C\phi_d(M).
	\]
\end{lemma}

\begin{proof}
	Choose $q>1$ so that $1-1/q>(4-d)/4$ when $d\leq3$; when
	$d\geq4$, choose a fixed $q>1$.

	\smallskip
	\noindent\emph{Step 1.} We prove
	\begin{equation}\label{eq:mean-horizon-moment}
		\bigl(\E u_m(0;\xi_\delta)^q\bigr)^{1/q}\leq C_q\phi_d(m).
	\end{equation}
    The representation \eqref{eq:stopping} shows that
    $\xi\mapsto u_m(0;\xi)$ is the supremum of a family of linear
    functionals that includes the zero functional. It is therefore convex and
    nonnegative. Hence, for $r\geq2$, the map $\xi\mapsto u_m(0;\xi)^r$ is convex in every coordinate.
	Replacing the finitely many relevant coordinates one at a time by independent
	copies of $\zeta$ and applying \eqref{eq:near-convex} gives
	\[
		\bigl(\E u_m(0;\xi_\delta)^r\bigr)^{1/r}
		\leq\bigl(\E u_m(0;\zeta)^r\bigr)^{1/r}.
	\]
	Theorem~\ref{thm:BP} bounds $\E u_m(0;\zeta)$ by $C\phi_d(m)$.
	Lemma~\ref{lem:u-concentration} and
	\eqref{eq:green-norms} therefore give
	\begin{equation}\label{eq:near-centered-moment}
		\bigl(\E u_m(0;\xi_\delta)^r\bigr)^{1/r}
		\leq C\left(\phi_d(m)+\sqrt r\,\|g_m\|_2
		+r\max_xg_m(x)\right).
	\end{equation}
	Taking $r=2\vee q$ and using monotonicity of moments proves
	\eqref{eq:mean-horizon-moment}.

	\smallskip
	\noindent\emph{Step 2.} We prove the bound in the statement.
	Let $N=1\vee\lceil M\rceil$, set $s_0=0$ and $\ell_0=N$, and set
	$s_k=\ell_k=2^{k-1}N$ for $k\geq1$. These pairs describe the blocks
	$[0,N)$ and $[2^{k-1}N,2^kN)$. Conditional on $\xi_\delta$ and the walk through
	time $s_k$, on $\{\sigma>s_k\}$ the conditional expected reward collected
	from time $s_k$ through time $\sigma\wedge(s_k+\ell_k)-1$ is at most
	$u_{\ell_k}(X_{s_k};\xi_\delta)$. Hence
	\begin{equation}\label{eq:stopping-blocks}
		\E\sum_{j<\sigma}\xi_\delta(X_j)
		\leq\sum_{k\geq0}\E\left[
		\one\{\sigma>s_k\}u_{\ell_k}(X_{s_k};\xi_\delta)\right].
	\end{equation}
	The walk is independent of $\xi_\delta$, so stationarity gives
	$(\E u_{\ell_k}(X_{s_k};\xi_\delta)^q)^{1/q}\leq C_q\phi_d(\ell_k)$ by
	\eqref{eq:mean-horizon-moment}. Also
	$\P(\sigma>s_k)\leq M/s_k\leq2^{1-k}$ for $k\geq1$.
	H\"older's inequality in \eqref{eq:stopping-blocks} now gives
	\[
		\E\sum_{j<\sigma}\xi_\delta(X_j)
		\leq C\phi_d(N)+C\sum_{k\geq1}
		2^{-(k-1)(1-1/q)}\phi_d(2^{k-1}N)
		\leq C\phi_d(M).
	\]
	The last series converges by the choice of $q$.
\end{proof}

\begin{proposition}\label{prop:near-divisible}
	  For all
	sufficiently small $\delta>0$,
	\begin{equation}\label{eq:near-divisible-upper}
		\E u_\infty^\delta(0)\leq C
		\begin{cases}
			\delta^{-3}&d=1\,,\\
			\delta^{-1}&d=2\,,\\
			\delta^{-1/3}&d=3\,,\\
			\log(e/\delta)&d\geq4.
		\end{cases}
	\end{equation}
	The reverse inequality holds when $d\leq4$, and when $d\geq5$ the lower bound
	is $c[\log(e/\delta)]^{2/d}$. If in addition $\eta_\delta(0)$ is bounded
	uniformly in $\delta$, then $\E u_\infty^\delta(0)\asymp[\log(e/\delta)]^{2/d}$
	when $d\geq5$.
\end{proposition}

\begin{proof}
	\smallskip
	\noindent\emph{Step 1.} We prove \eqref{eq:near-divisible-upper}.
	For each $n$, let $\sigma_n$ be the first $j\leq n$ for which $u_{n-j}^\delta(X_j)=0$. Dynamic programming shows that $\sigma_n$ attains the
	supremum in \eqref{eq:stopping}. Let $M_n\coloneqq\E\sigma_n$. Since
	$\eta_\delta=\xi_\delta-\delta$, Lemma~\ref{lem:mean-horizon} gives
	\[
		\E u_n^\delta(0)\leq C\phi_d(M_n)-\delta M_n.
	\]
	Taking the supremum over $M\geq0$ gives the upper bound in
	\eqref{eq:near-divisible-upper}. For $d\leq3$, the supremum has order at most
	$\delta^{-(4-d)/d}$; for $d\geq4$ it is at most
	$C\log(e/\delta)$. The resulting bound on $\E u_n^\delta(0)$ is uniform in $n$, so the monotone convergence theorem applies.

	\smallskip
	\noindent\emph{Step 2.} We prove the lower bounds in the statement.
	The comparison \eqref{eq:near-convex}, positive homogeneity, and
	Theorem~\ref{thm:BP} give
	\[
		\E u_m(0;\xi_\delta)\geq c
		\begin{cases}
			m^{(4-d)/4}&d\leq3\,,\\
			\log(m+1)&d=4\,,\\
			[\log(m+1)]^{2/d}&d\geq5
		\end{cases}
	\]
	for all large $m$, uniformly in $\delta$. An optimizer for
	$u_m(0;\xi_\delta)$ is a competitor for the $\eta_\delta$-problem, and hence
	\[
		u_\infty^\delta(0)\geq u_m(0;\xi_\delta)-\delta m.
	\]
	Taking expectations and optimizing over $m$ gives the matching lower bounds
	for $d\leq4$ and the stated lower bound for $d\geq5$.

	\smallskip
	\noindent\emph{Step 3.} We prove the matching upper bound for bounded laws
	when $d\geq5$. Choose $B$ so that $|\xi_\delta(0)|\leq B$ almost surely for every sufficiently small $\delta$. Since $\xi_\delta(0)$ has mean zero, it is
	dominated in convex order by $B\chi$. Applying the coordinatewise
	convex comparison to $u_m(0;\cdot)^2$, then using Theorem~\ref{thm:BP} and Lemma~\ref{lem:u-concentration}, gives
	\[
		\bigl(\E u_m(0;\xi_\delta)^2\bigr)^{1/2}
		\leq C[\log(m+2)]^{2/d}.
	\]
	Repeating Step~2 of Lemma~\ref{lem:mean-horizon} with $q=2$ bounds the
	expected reward of a stopping time of mean $M$ by
	$C[\log(M+2)]^{2/d}$. Subtracting $\delta M$ and optimizing over $M$
	proves the last assertion.
\end{proof}

\begin{proof}[Proof of the lower bounds in Theorem~\ref{thm:near}]
	Theorem~\ref{thm:comparison} gives $\E U_\infty^\delta(0)\geq\E u_\infty^\delta(0)$.
	Proposition~\ref{prop:near-divisible} now gives the claimed bounds when
	$d\leq3$. For $d\geq4$, take independent copies across
	sites of the coupling in Theorem~\ref{thm:near}. Applying
	Lemma~\ref{lem:density-compare} first to
	$\min\{\eta_\delta,\eta_0\}$ and then to
	$\max\{\eta_\delta,\eta_0\}$ gives
	\[
		S_t^\delta\geq S_t^0-\E|\eta_\delta(0)-\eta_0(0)|
		\geq S_t^0-C\delta.
	\]
	Lemma~\ref{lem:critical-density} gives $S_t^0\geq c/(t+1)$ for all large
	$t$. Hence $S_t^\delta\geq c/(t+1)$ from that point up to $c/\delta$,
	and summing proves $\E U_\infty^\delta(0)\geq c\log(e/\delta)$.
\end{proof}

\begin{proof}[Proof of the upper bounds in Theorem~\ref{thm:near}]
	\smallskip
	\noindent\emph{Step 1.} We prove \eqref{eq:near-tail} for the cutoff
	\eqref{eq:near-cutoff}.
	Let $L=\log(e/\delta)$. Take
	\begin{equation}\label{eq:near-cutoff}
		N=\left\lceil C
		\begin{cases}
			\delta^{-4}L^3&d=1\,,\\
			\delta^{-2}L^3&d=2\,,\\
			\delta^{-2}L^2&d\geq3.
		\end{cases}\right\rceil
	\end{equation}
	With $C$ sufficiently large, Lemma~\ref{lem:near-tilt} and
	Proposition~\ref{prop:resolvent} give
	\begin{equation}\label{eq:near-tail}
		\sum_{t\geq N}S_t^\delta\leq C,
		\qquad \E U_\infty^\delta(0)\leq\E U_N^\delta(0)+C\,.
	\end{equation}

	\smallskip
	\noindent\emph{Step 2.} We prove \eqref{eq:near-routing-mean} and
	\eqref{eq:near-routing-moment}.
	Theorem~\ref{thm:comparison}, Lemma~\ref{lem:pathwise-comparison}, and
	Proposition~\ref{prop:w-moment} give, for $n\geq2$ and $r\geq2$,
	\begin{align}
		0\leq\E U_n^\delta(0)-\E u_n^\delta(0)
		&\leq C(n+1)^{1/r}
		\left(\sqrt{r\kappa_d(n)\bigl(\E U_n^\delta(0)^r\bigr)^{1/r}}+r\right),
		\label{eq:near-routing-mean}\\
		\bigl(\E U_n^\delta(0)^r\bigr)^{1/r}
		&\leq C\left(\bigl(\E u_n^\delta(0)^r\bigr)^{1/r}
		+r(n+1)^{2/r}\kappa_d(n)\right).
		\label{eq:near-routing-moment}
	\end{align}
	The second bound follows from the pathwise estimate by Young's inequality.
	The constants are uniform in $\delta$.

	\smallskip
	\noindent\emph{Step 3.} We prove the upper bounds in \eqref{eq:near}.
	Since $\eta_\delta\leq\xi_\delta$, monotonicity and
	\eqref{eq:near-centered-moment} bound $(\E u_n^\delta(0)^r)^{1/r}$. Take
	$r=2\vee\lceil\log(N+1)\rceil$. Then $r\asymp L$ and both
	$(N+1)^{1/r}$ and $(N+1)^{2/r}$ are bounded. Equations
	\eqref{eq:near-centered-moment}, \eqref{eq:green-norms}, and
	\eqref{eq:near-routing-moment} give
	\[
		\bigl(\E U_N^\delta(0)^r\bigr)^{1/r}\leq C
		\begin{cases}
			\delta^{-3}L^3&d=1\,,\\
			\delta^{-1}L^2&d=2\,,\\
			\delta^{-1/2}L&d=3\,,\\
			L&d\geq4.
		\end{cases}
	\]
	Since $\kappa_1(N)=\sqrt N$, $\kappa_2(N)=\log(N+2)$, and
	$\kappa_d(N)=1$ for $d\geq3$, \eqref{eq:near-routing-mean} bounds
	$\E U_N^\delta(0)-\E u_N^\delta(0)$ by
	$C\delta^{-5/2}L^3$, $C\delta^{-1/2}L^2$, $C\delta^{-1/4}L$, and $CL$
	in dimensions $1$, $2$, $3$, and $d\geq4$, respectively. The first
	three are $o(\delta^{-3})$, $o(\delta^{-1})$, and
	$o(\delta^{-1/3})$, respectively. Since $u_N^\delta\leq u_\infty^\delta$, combining the last bounds on $\E U_N^\delta(0)-\E u_N^\delta(0)$
	with \eqref{eq:near-divisible-upper} and \eqref{eq:near-tail} proves the
	upper bounds in \eqref{eq:near}.
\end{proof}

\section{Oriented walk}\label{sec:oriented}

The oriented walk reaches the layer at distance $\ell$ from its starting layer
at time $\ell$ and at no other time. Its truncated Green function is therefore
a sum of probability measures with disjoint supports, so it is bounded by one
and the square of its norm is the sum over the layers of the squared norms of
the layer laws. That sum has order $\sqrt n$ in dimension two and $\log n$ in
dimension three, and is bounded from dimension four on, so the critical
ambient dimension is three rather than four. The same disjointness makes the
variances of Section~\ref{sec:routing} telescope over the layers and sum to
one, so that $\kappa_d(n)$ is replaced by one in every dimension, and it makes
the odometers on a layer depend only on lower layers, which supplies the
logarithmic lower bound in place of the coupling of
Section~\ref{sec:critical}.

Fix $d\geq2$, let $e_1,\ldots,e_d$ be the coordinate vectors in $\Z^d$, and
let $\vec P$ be the transition matrix with
\[
	\vec P(x,x-e_i)=\frac1{d}\qquad(1\leq i\leq d)
\]
and all other entries zero. The particles take the opposite steps, so their
transition matrix is the transpose $\vec P^*$. Let $\P_x^{\vec P}$ and
$\E_x^{\vec P}$ denote the law and expectation of the $\vec P$-walk started
at $x$. By Remark~\ref{rem:transpose}, the solution of
\begin{equation}\label{eq:oriented-divisible}
	\vec u_0=0\,,\qquad \vec u_{n+1}=(\eta+\vec P\vec u_n)^+
\end{equation}
satisfies $\E[\vec U_n(x)\mid\eta]\geq \vec u_n(x)$.

Let $\vec p_\ell(x)=\vec P^\ell(0,x)$. This measure is supported on the layer
$\{x:\sum_i x_i=-\ell\}$, which the $\vec P$-walk visits only at time $\ell$.
The local central limit theorem on this $(d-1)$-dimensional layer gives
\begin{equation}\label{eq:oriented-heat}
	\sum_x\vec p_\ell(x)^2\asymp(\ell+1)^{-(d-1)/2}.
\end{equation}
Since the supports of the $\vec p_\ell$ are disjoint, let
$g_n=\sum_{\ell<n}\vec p_\ell$. Then $0\leq g_n\leq1$ and
\begin{equation}\label{eq:oriented-green}
	\|g_n\|_2^2=\sum_{\ell<n}\|\vec p_\ell\|_2^2\asymp \vec\kappa_d(n),
	\qquad
	\vec\kappa_d(n)=
	\begin{cases}
		\sqrt n&d=2\,,\\
		\log(n+1)&d=3\,,\\
		1&d\geq4.
	\end{cases}
\end{equation}

When $d=2$, the number of steps in direction $-e_1$ identifies the support of
$\vec p_\ell$ with $\{0,\ldots,\ell\}$. Under this identification $\vec p_\ell$
is the binomial law with $\ell$ trials and success probability $1/2$. For an
integer $q$, let $\vec p_\ell(\cdot-q)$ denote its translate.

\begin{lemma}\label{lem:shift}
	If $d=2$, then for every integer $q$,
	\[
		\sum_{\ell\geq0}\bigl\|\vec p_\ell(\cdot-q)-\vec p_\ell\bigr\|_2^2=4|q|\,.
	\]
\end{lemma}

\begin{proof}
	Both sides are unchanged when $q$ is replaced by $-q$, and both vanish at
	$q=0$, so assume $q\geq1$.
	On a layer the Fourier transform of $\vec p_\ell$ is
	$((1+e^{i\theta})/2)^\ell$, of modulus squared $\cos^{2\ell}(\theta/2)$.
	Summing the geometric series and applying Parseval's identity turns the left
	side into
	\[
		\frac1{2\pi}\int_{-\pi}^{\pi}\frac{|e^{iq\theta}-1|^2}
		{1-\cos^2(\theta/2)}\,d\theta
		=\frac1{2\pi}\int_{-\pi}^{\pi}
		\frac{4\sin^2(q\theta/2)}{\sin^2(\theta/2)}\,d\theta\,,
	\]
	and the integrand is four times
	$\bigl|\sum_{0\leq m<q}e^{im\theta}\bigr|^2$, whose normalized integral is
	$q$.
\end{proof}

\begin{theorem}[Divisible odometer]\label{thm:oriented}
	Let $\eta=(\eta(x))_{x\in\Z^d}$ be i.i.d.\ with mean zero, and let $\vec u_n$
	obey \eqref{eq:oriented-divisible}. The following bounds hold for $n\geq1$.
	If $d=2$ and $\E|\eta(0)|^r<\infty$ for some $r>4$, then
	\[
		\E \vec u_n(0)\leq Cn^{1/4}.
	\]
	If $d\geq3$ and $\E e^{\theta|\eta(0)|}<\infty$ for some $\theta>0$, then
	\[
		\E \vec u_n(0)\leq C\log(n+1).
	\]
	If $\eta(0)$ has positive finite variance, then
	$\E \vec u_n(0)\geq c\sqrt{\vec\kappa_d(n)}$.
\end{theorem}

\begin{proof}[Proof of the lower bound]
	The deterministic stopping time $n$, and also $\sigma=0$, in the
	$\vec P$-walk version of \eqref{eq:stopping} give
	$\vec u_n(0)\geq(\sum_xg_n(x)\eta(x))^+$. When $d\leq3$,
	$\|g_n\|_2\to\infty$ while $\|g_n\|_\infty\leq1$, so the Lindeberg central
	limit theorem and uniform integrability give
	$\E \vec u_n(0)\geq c\|g_n\|_2\geq c\sqrt{\vec\kappa_d(n)}$ for all large $n$;
	monotonicity covers the remaining $n$. When
	$d\geq4$, the lower bound follows from
	$\vec u_n(0)\geq\vec u_1(0)=\eta(0)^+$ and
	$\vec\kappa_d(n)\asymp1$.
\end{proof}

\begin{proof}[Proof of the upper bound]
	\smallskip
	\noindent\emph{Step 1.} We prove the logarithmic bound when $d\geq3$.
	Suppose that $d\geq3$. Apply \citet[Lemmas~2.5 and~3.2]{BP} with $P$
	replaced by $\vec P$. For
	$\Phi_m(x)=\sum_y g_m(y-x)\eta(y)$, their stopping identity and
	concentration estimate give, with
	$r=2\vee\lceil\log(n+1)\rceil$,
	\[
		\E \vec u_n(0)\leq\E\E_0^{\vec P}\max_{j\leq n}|\Phi_{n-j}(X_j)|
		\leq C\bigl(\sqrt{r \vec\kappa_d(n)}+r\bigr)
		\leq C\log(n+1).
	\]
	Equation~\eqref{eq:oriented-green} gives
	$\|g_n\|_2^2\asymp\vec\kappa_d(n)$, while $g_n\leq1$ gives
	$\|g_n\|_\infty\leq1$.

	\smallskip
	\noindent\emph{Step 2.} We prove \eqref{eq:oriented-increment} when $d=2$.
	Suppose that $d=2$. Let $\Phi_m=\sum_{\ell<m}\vec P^\ell\eta$, so that
	$\Phi_m=\eta+\vec P\Phi_{m-1}$, and let $X$ be the $\vec P$-walk from the
	origin. The martingale identity in \citet[Lemma~2.5]{BP}, with $\vec P$
	in place of $P$, gives
	\[
		\vec u_n(0)\leq\Phi_n(0)+\E_0^{\vec P}\max_{0\leq j\leq n}|Z_j|,
		\qquad Z_j=\Phi_{n-j}(X_j).
	\]
	Fix an exponent $r>4$ for which $\E|\eta(0)|^r<\infty$. For
	$0\leq j,k\leq n$,
	\begin{equation}\label{eq:oriented-increment}
		\bigl(\E|Z_j-Z_k|^r\bigr)^{1/r}\leq C|k-j|^{1/4}\,.
	\end{equation}
	By symmetry, assume $0\leq j\leq k\leq n$. Condition on the walk and let
	$h=k-j$. Let $D$ be the number of $-e_1$
	steps between times $j$ and $k$. The $h$ layers between $X_j$ and $X_k$
	contribute at most $C\sqrt h$ to the sum of the squared coefficients of
	$Z_j-Z_k$. On the remaining layers, convolution with an independent
	$L\sim\operatorname{Bin}(h,1/2)$ and Lemma~\ref{lem:shift} bound this sum by
	$4\E(|L-D|\mid D)$. Hence the full sum of squared coefficients is at most
	\[
		C\sqrt h+4\E(|L-D|\mid D)
		\leq C\bigl(\sqrt h+|D-h/2|\bigr).
	\]
	Conditional Rosenthal's inequality bounds the $L^r$ norm of $Z_j-Z_k$ by
	a constant times the square root of this coefficient sum. Since
	$D\sim\operatorname{Bin}(h,1/2)$,
	\[
		\bigl(\E|D-h/2|^{r/2}\bigr)^{2/r}\leq C\sqrt h.
	\]
	Averaging over $D$ proves \eqref{eq:oriented-increment}.

	\smallskip
	\noindent\emph{Step 3.} We prove the $L^r$ maximal bound
	\eqref{eq:oriented-dyadic-max} by summing the dyadic increments.
	For dyadic $n$, \eqref{eq:oriented-increment} gives
	\begin{equation}\label{eq:oriented-dyadic-max}
		\Bigl(\E\max_{j\leq n}|Z_j|^r\Bigr)^{1/r}
		\leq C\sum_{m\geq1}2^{m/r}\bigl(n2^{-m}\bigr)^{1/4}
		\leq Cn^{1/4}.
	\end{equation}
	The random variable $\Phi_n(0)$ has mean zero, so the stopping identity
	gives $\E\vec u_n(0)\leq Cn^{1/4}$. For general $n$, compare with the least
	dyadic integer greater than or equal to $n$ and use monotonicity.
\end{proof}

\begin{proposition}\label{prop:oriented-scaling}
	Let $d=2$, and let $\eta=(\eta(x))_{x\in\Z^2}$ have i.i.d.\ coordinates.
	Suppose that $\eta(0)$ is nonconstant, $\E\eta(0)=0$, and
	$\E e^{\theta|\eta(0)|}<\infty$ for some $\theta>0$. Let $\vec P(x,x-e_i)=1/2$ for $i=1,2$, and let $\vec u_0=0$ and $\vec u_{n+1}=(\eta+\vec P\vec u_n)^+$. Then, for every
	$T>0$, $n^{-1/4}\vec u_{\lfloor nT\rfloor}(0)$ converges in distribution to a random variable $\mathcal U(T)$ defined in the proof. Moreover, $\mathcal U(T)\stackrel d=T^{1/4}\mathcal U(1)$, $\mu\coloneqq\E\mathcal U(1)\in(0,\infty)$, and
	\[
		\lim_{n\to\infty}n^{-1/4}\E \vec u_n(0)=\mu\,.
	\]
\end{proposition}

\begin{proof}
	We adapt the proof of the parabolic scaling limit in \citet{BP}, with
	space--time white noise in place of spatial white noise. We record the
	limiting object and the
	estimates that change for the directed kernel.

	\smallskip
	\noindent\emph{Step 1.} We define $\mathcal U(T)$ and prove
	\[
		n^{-1/4}\vec u_{\lfloor nT\rfloor}(0)
		\Longrightarrow\mathcal U(T).
	\]
	Let $B$ be Brownian motion with $\Var(B_t)=t/4$, let $q_t$ be its
	transition density, and let $W$ be space--time white noise on
	$[0,\infty)\times\mathbb R$, independent of $B$. Let
	\[
		\mathcal Z_T(s,x)\coloneqq
		\sqrt{\Var\eta(0)}\int_s^T\int_{\mathbb R}
		q_{r-s}(x,y)\,W(dr,dy).
	\]
	Conditional on $W$, let
	\[
		\mathcal U(T)\coloneqq\mathcal Z_T(0,0)
		+\sup_{\tau\leq T}\E_0[-\mathcal Z_T(\tau,B_\tau)],
	\]
	where the supremum is over stopping times of $B$ bounded by $T$, and $\E_0$
	averages only over $B$.

	Identify layer $\ell$ with the sites $(-j,j-\ell)$, $j\in\mathbb Z$. The
	rescaled scenery
	\[
		n^{-3/4}\sum_{\ell\geq0}\sum_{j\in\mathbb Z}
		\eta(-j,j-\ell)\,
		\delta_{(\ell/n,(j-\ell/2)/\sqrt n)}
	\]
	converges in law as a random distribution to
	$\sqrt{\Var\eta(0)}W$. The binomial local central limit
	theorem gives convergence of the convolved potentials, and the rescaled
	$\vec P$-walk converges to $B$. The cutoff and stability estimates in the
	proof of the parabolic scaling limit in \citet{BP} use only heat-kernel bounds and
	bounds on translated kernel differences. Equation~\eqref{eq:oriented-heat}
	and Lemma~\ref{lem:shift} give those bounds here. The cited cutoff and
	stability estimates therefore prove the displayed convergence.

	\smallskip
	\noindent\emph{Step 2.} We prove
	$\mathcal U(T)\stackrel d=T^{1/4}\mathcal U(1)$ and
	$n^{-1/4}\E\vec u_n(0)\to\mu\in(0,\infty)$.
	Parabolic scaling gives
	$\mathcal U(T)\stackrel d=T^{1/4}\mathcal U(1)$. Equation
	\eqref{eq:oriented-dyadic-max}, together with
	\[
		\bigl(\E|\Phi_n(0)|^r\bigr)^{1/r}
		\leq C\|g_n\|_2\leq Cn^{1/4},
	\]
	bounds $n^{-1/4}\vec u_n(0)$ in $L^r$ for a fixed $r>4$. Uniform
	integrability gives convergence of the means, and the lower bound in
	Theorem~\ref{thm:oriented} gives $\mu>0$.
\end{proof}

\begin{remark}[Critical dimension]\label{rem:critical-dimension}
	On $\Z^d$ the layers have dimension $d-1$ and
	$\|\vec p_\ell\|_2^2\asymp(\ell+1)^{-(d-1)/2}$. Accordingly, as
	$T\to\infty$,
	\[
		\int_0^T(s+1)^{-(d-1)/2}\,ds
	\]
	has order $\sqrt T$ when $d=2$, order $\log(T+1)$ when $d=3$, and remains
	bounded when $d\geq4$. Thus the critical ambient dimension is three,
	and the proof of Proposition~\ref{prop:oriented-scaling} applies only in
	$1+1$ dimensions. Under
	the hypotheses of Theorem~\ref{thm:oriented-walk}, the particle odometer has
	order $n^{1/4}$ when $d=2$ and $\log n$ when $d\geq3$.
\end{remark}

\begin{proof}[Proof of Theorem~\ref{thm:oriented-walk}]
	\smallskip
	\noindent\emph{Step 1.} We prove that
	Proposition~\ref{prop:w-moment} holds with $\kappa_d(n)$ replaced by one and
	prove \eqref{eq:oriented-u-concentration}.
	Although each instruction has law $\vec P^*(y,\cdot)$, its mean arrivals are
	propagated by $\vec P$. Let $w_0=0$ and
	\[
		w_{k+1}(x)=(\vec Pw_k)(x)
		+\sum_y\bigl(I_{y,x}(\vec U_k(y))-\vec P^*(y,x)\vec U_k(y)\bigr)\,.
	\]
	Since $\sum_y\vec P^*(y,x)\vec U_k(y)=(\vec P\vec U_k)(x)$,
	Lemma~\ref{lem:pathwise-comparison} applies with $\vec P$ in the recursion,
	and Lemma~\ref{lem:w-martingale} applies with $\vec P^*$ as the instruction
	law. For the Green function in \eqref{eq:oriented-green}, let
	\[
		\Gamma_m(y)=\sum_z\vec P^*(y,z)\bigl(g_m(z)-(\vec P^*g_m)(y)\bigr)^2.
	\]

	The $\vec P$-walk reaches the layer $\{x:\sum_i x_i=-\ell\}$ at time $\ell$ and at
	no other time. Thus $0\leq g_m\leq1$ and
	$|g_m(z)-(\vec P^*g_m)(y)|\leq1$ whenever $\vec P^*(y,z)>0$. For $y$ on layer
	$\ell\geq1$,
	\[
		\sup_m\Gamma_m(y)=\frac1{d}\sum_{i=1}^{d}
		\bigl(\vec p_{\ell-1}(y+e_i)-\vec p_\ell(y)\bigr)^2,
	\]
	while $\Gamma_m(y)=0$ for every $m<\ell$. Summing the variance on each layer
	and then telescoping gives
	\[
		\sum_y\sup_m\Gamma_m(y)
		=\sum_{\ell\geq1}
		\bigl(\|\vec p_{\ell-1}\|_2^2-\|\vec p_\ell\|_2^2\bigr)
		=1
	\]
	because $\|\vec p_0\|_2^2=1$ and
	$\|\vec p_\ell\|_2^2\to0$ by \eqref{eq:oriented-heat}. Therefore the conclusions of
	Proposition~\ref{prop:w-moment} apply with $\kappa_d(n)$ replaced by one.
	Lemma~\ref{lem:u-concentration}, applied to $\vec P$, gives
	\begin{equation}\label{eq:oriented-u-concentration}
		\bigl(\E|\vec u_n(0)-\E \vec u_n(0)|^r\bigr)^{1/r}
		\leq C\bigl(\sqrt{r\vec\kappa_d(n)}+r\bigr).
	\end{equation}

	\smallskip
	\noindent\emph{Step 2.} We prove the order of the odometer when $d=2$.
	Take $r=8$ in Lemma~\ref{lem:pathwise-comparison} and
	Proposition~\ref{prop:w-moment}, with $\kappa_d(n)$ replaced by one.
	Theorem~\ref{thm:oriented} gives
	\[
		\bigl(\E \vec U_n(0)^8\bigr)^{1/8}
		\leq Cn^{1/4}+Cn^{1/8}
		\Bigl(\bigl(\E \vec U_n(0)^8\bigr)^{1/16}+1\Bigr),
	\]
	so Young's inequality yields
	$(\E \vec U_n(0)^8)^{1/8}\leq Cn^{1/4}$. The inequality
	$\E \vec U_n(0)\geq\E \vec u_n(0)$ and the lower bound in
	Theorem~\ref{thm:oriented} give the matching lower bound.

	\smallskip
	\noindent\emph{Step 3.} We identify the odometer and activity asymptotics
	when $d=2$.
	To compare $\E\vec U_n(0)$ with $\E\vec u_n(0)$, take
	$r=2\vee\lceil\log(n+1)\rceil$. Proposition~\ref{prop:w-moment}, with
	$\kappa_d(n)$ replaced by one, and
	Lemma~\ref{lem:pathwise-comparison} give
	$(\E \vec U_n(0)^r)^{1/r}\leq Cn^{1/4}\sqrt{\log(n+1)}$, so the same
	proposition bounds $(\E|w_n(0)|^r)^{1/r}$ by
	\[
		C\bigl(n^{1/8}[\log(n+1)]^{3/4}+\log(n+1)\bigr)\,.
	\]
	Since $(n+1)^{1/r}\leq e$, Lemma~\ref{lem:pathwise-comparison} gives
	\begin{equation}\label{eq:oriented-particle-error}
		\E|\vec U_n(0)-\vec u_n(0)|
		\leq Cn^{1/8}[\log(n+1)]^{3/4}=o(n^{1/4}).
	\end{equation}
	Theorem~\ref{thm:oriented} gives
	$\E\vec u_n(0)\asymp n^{1/4}$, so
	$\E\vec U_n(0)/\E\vec u_n(0)\to1$. Proposition~\ref{prop:oriented-scaling}
	now gives $\E\vec U_n(0)\sim\mu n^{1/4}$. Equation
	\eqref{eq:oriented-particle-error} also gives
	$n^{-1/4}(\vec U_n(0)-\vec u_n(0))\to0$ in $L^1$, so
	$n^{-1/4}\vec U_n(0)$ converges in distribution to $\mathcal U(1)$.

	Since $\vec S_t$ decreases and
	$\E\vec U_n(0)=\sum_{s<n}\vec S_s$, for every $0<\eps<1$,
	\[
		\frac{\E\vec U_{\lfloor(1+\eps)t\rfloor}(0)-\E\vec U_t(0)}
		{\lfloor(1+\eps)t\rfloor-t}
		\leq\vec S_t\leq
		\frac{\E\vec U_t(0)-\E\vec U_{\lfloor(1-\eps)t\rfloor}(0)}
		{t-\lfloor(1-\eps)t\rfloor}
	\]
	for all sufficiently large integers $t$. Multiply by $t^{3/4}$, let
	$t\to\infty$, and then let $\eps\downarrow0$. This gives
	$\vec S_t\sim(\mu/4)t^{-3/4}$.

	\smallskip
	\noindent\emph{Step 4.} We prove the upper bound when $d\geq3$.
	Take $r=2\vee\lceil\log(n+1)\rceil$. Theorem~\ref{thm:oriented},
	\eqref{eq:oriented-u-concentration}, and
	Lemma~\ref{lem:pathwise-comparison} give
	\[
		\bigl(\E \vec U_n(0)^r\bigr)^{1/r}\leq C\log(n+1)
		+C\Bigl(\sqrt{r\bigl(\E \vec U_n(0)^r\bigr)^{1/r}}+r\Bigr).
	\]
	Another use of Young's inequality gives $\E \vec U_n(0)\leq C\log(n+1)$.

	\smallskip
	\noindent\emph{Step 5.} We prove the logarithmic lower bound when $d\geq3$.
	Let $a_n=\E \vec U_n(0)$ and let
	\[
		N_n=\sum_{i=1}^{d}I_{-e_i,0}\bigl(\vec U_n(-e_i)\bigr)
	\]
	be the number of arrivals at the origin through round $n$. Once an oriented
	path leaves a layer, it never returns. Induction in the particle recursion shows
	that the odometers on $\{x:\sum_i x_i=m\}$ depend only on $\eta$
	on layers with coordinate sum at most $m$ and stacks based on layers with
	coordinate sum less than $m$. Hence
	$(\vec U_n(-e_i))_{i=1}^{d}$ is independent of the stacks based on the layer
	of the $-e_i$ and of $\eta(0)$. With $q=(d-1)/d$, conditional independence gives
	\[
		\P\bigl(N_n=0\mid \vec U_n(-e_1),\ldots,\vec U_n(-e_d)\bigr)
		=q^{\sum_i\vec U_n(-e_i)}.
	\]
	The variable $N_n$ is measurable with respect to the scenery and the stacks
	based at sites with coordinate sum at most $-1$. It is therefore independent
	of $\eta(0)$.
	Conditional expectation and translation invariance give
	$\E N_n=d^{-1}\sum_i\E \vec U_n(-e_i)=a_n$. The particle recursion says
	$\vec U_{n+1}(0)=(\eta(0)+N_n)^+$. Since $x^+=x+x^-$ and
	$\E[\eta(0)+N_n]=a_n$,
	\[
		a_{n+1}-a_n
		=\E\bigl(\eta(0)+N_n\bigr)^-
		\geq \E\eta(0)^-\,\P(N_n=0),
	\]
	where the inequality uses the independence of $N_n$ and $\eta(0)$.
	Jensen's inequality yields
	\[
		\P(N_n=0)\geq q^{da_n}=e^{-\lambda a_n},
		\qquad \lambda=-d\log q>0.
	\]
	Since $\E\eta(0)^->0$, it follows that
	$a_{n+1}-a_n\geq ce^{-\lambda a_n}$. Therefore
	\[
		e^{\lambda a_{n+1}}-e^{\lambda a_n}
		=e^{\lambda a_n}\bigl(e^{\lambda(a_{n+1}-a_n)}-1\bigr)
		\geq \lambda e^{\lambda a_n}(a_{n+1}-a_n)\geq c.
	\]
	Summing from $a_0=0$ gives $a_n\geq c\log(n+1)$, which matches the upper
	bound when $d\geq3$.
\end{proof}

\bibliographystyle{plainnat}
\bibliography{refs}

\begin{thebibliography}{99}

\bibitem[Ahn, Junge, Lyu, Reeves, Richey and Sivakoff(2025)]{AJLRS}
Sungwon Ahn, Matthew Junge, Hanbaek Lyu, Lily Reeves, Jacob Richey, and David Sivakoff,
\emph{Diffusion-limited annihilating-coalescing systems},
Electron. J. Probab. \textbf{30} (2025);
\href{https://doi.org/10.1214/25-EJP1286}{doi:10.1214/25-EJP1286};
\href{https://arxiv.org/abs/2305.19333}{arXiv:2305.19333}.


\bibitem[Aldous, Contat, Curien and H\'enard(2023)]{ACCH}
David Aldous, Alice Contat, Nicolas Curien, and Olivier H\'enard,
\emph{Parking on the infinite binary tree},
Probab. Theory Related Fields \textbf{187} (2023), 481--504;
\href{https://doi.org/10.1007/s00440-023-01189-6}{doi:10.1007/s00440-023-01189-6};
\href{https://arxiv.org/abs/2205.15932}{arXiv:2205.15932}.

\bibitem[Athreya, Drewitz and Sun(2019)]{ADS}
Siva Athreya, Alexander Drewitz, and Rongfeng Sun,
\emph{Random walk among mobile/immobile traps: a short review},
Sojourns in Probability Theory and Statistical Physics III,
Springer Proc. Math. Stat. \textbf{300} (2019);
\href{https://doi.org/10.1007/978-981-15-0302-3_1}{doi:10.1007/978-981-15-0302-3\_1};
\href{https://arxiv.org/abs/1703.06617}{arXiv:1703.06617}.


\bibitem[Bahl, Barnet, Johnson and Junge(2022)]{BBJJ}
Riti Bahl, Philip Barnet, Tobias Johnson, and Matthew Junge,
\emph{Diffusion-limited annihilating systems and the increasing convex order},
Electron. J. Probab. \textbf{27} (2022), no.~84, 1--19;
\href{https://doi.org/10.1214/22-EJP808}{doi:10.1214/22-EJP808};
\href{https://arxiv.org/abs/2104.12797}{arXiv:2104.12797}.

\bibitem[Bahl, Barnet and Junge(2021)]{BBJ}
Riti Bahl, Philip Barnet, and Matthew Junge,
\emph{Parking on supercritical Galton-Watson trees},
ALEA Lat. Am. J. Probab. Math. Stat. \textbf{18} (2021), 1801--1815;
\href{https://doi.org/10.30757/alea.v18-67}{doi:10.30757/alea.v18-67};
\href{https://arxiv.org/abs/1912.13062}{arXiv:1912.13062}.


\bibitem[Bou-Rabee and Panagiotis(2026)]{BP}
Ahmed Bou-Rabee and Christoforos Panagiotis,
\emph{Quantitative explosion and percolation of the divisible sandpile},
preprint, 2026.

\bibitem[Bou-Rabee, Peres and Sava-Huss(2026)]{BPRS}
Ahmed Bou-Rabee, Yuval Peres, and Ecaterina Sava-Huss,
\emph{Divisible sandpiles via random walks in random scenery},
\href{https://arxiv.org/abs/2604.13968}{arXiv:2604.13968}.

\bibitem[Bramson and Lebowitz(1991)]{BL}
Maury Bramson and Joel L. Lebowitz,
\emph{Asymptotic behavior of densities for two-particle annihilating random walks},
J. Statist. Phys. \textbf{62} (1991), 297--372;
\href{https://doi.org/10.1007/BF01020872}{doi:10.1007/BF01020872}.

\bibitem[Cabezas, Rolla and Sidoravicius(2014)]{CRS14}
Manuel Cabezas, Leonardo T. Rolla, and Vladas Sidoravicius,
\emph{Non-equilibrium phase transitions: activated random walks at criticality},
J. Stat. Phys. \textbf{155} (2014), no.~6, 1112--1125;
\href{https://doi.org/10.1007/s10955-013-0909-3}{doi:10.1007/s10955-013-0909-3};
\href{https://arxiv.org/abs/1307.4450}{arXiv:1307.4450}.

\bibitem[Cabezas, Rolla and Sidoravicius(2018)]{CRS}
Manuel Cabezas, Leonardo T. Rolla, and Vladas Sidoravicius,
\emph{Recurrence and density decay for diffusion-limited annihilating systems},
Probab. Theory Related Fields \textbf{170} (2018), 587--615;
\href{https://doi.org/10.1007/s00440-017-0763-3}{doi:10.1007/s00440-017-0763-3};
\href{https://arxiv.org/abs/1309.4387}{arXiv:1309.4387}.

\bibitem[Damron, Gravner, Junge, Lyu and Sivakoff(2019)]{DGJLS}
Michael Damron, Janko Gravner, Matthew Junge, Hanbaek Lyu, and David Sivakoff,
\emph{Parking on transitive unimodular graphs},
Ann. Appl. Probab. \textbf{29} (2019), 2089--2113;
\href{https://doi.org/10.1214/18-AAP1443}{doi:10.1214/18-AAP1443};
\href{https://arxiv.org/abs/1710.10529}{arXiv:1710.10529}.

\bibitem[Damron, Lyu and Sivakoff(2021)]{DLS}
Michael Damron, Hanbaek Lyu, and David Sivakoff,
\emph{Stretched exponential decay for subcritical parking times on $\Z^d$},
Random Structures \& Algorithms \textbf{59} (2021), no.~2, 143--154;
\href{https://doi.org/10.1002/rsa.21001}{doi:10.1002/rsa.21001};
\href{https://arxiv.org/abs/2008.05072}{arXiv:2008.05072}.

\bibitem[Diaconis and Fulton(1991)]{DiaconisFulton}
Persi Diaconis and William Fulton,
\emph{A growth model, a game, an algebra, Lagrange inversion, and characteristic classes},
Rend. Sem. Mat. Univ. Politec. Torino \textbf{49} (1991), no.~1, 95--119.

\bibitem[Donsker and Varadhan(1979)]{DonskerVaradhan}
Monroe D. Donsker and S. R. Srinivasa Varadhan,
\emph{On the number of distinct sites visited by a random walk},
Comm. Pure Appl. Math. \textbf{32} (1979), 721--747;
\href{https://doi.org/10.1002/cpa.3160320602}{doi:10.1002/cpa.3160320602}.

\bibitem[Johnson and Richey(2025)]{JR}
Tobias Johnson and Jacob Richey,
\emph{The odometer in subcritical activated random walk},
\href{https://arxiv.org/abs/2510.05514}{arXiv:2510.05514}.

\bibitem[Johnson, Junge, Lyu and Sivakoff(2023)]{JJLS}
Tobias Johnson, Matthew Junge, Hanbaek Lyu, and David Sivakoff,
\emph{Particle density in diffusion-limited annihilating systems},
Ann. Probab. \textbf{51} (2023), no.~6, 2301--2344;
\href{https://doi.org/10.1214/23-AOP1653}{doi:10.1214/23-AOP1653};
\href{https://arxiv.org/abs/2005.06018}{arXiv:2005.06018}.

\bibitem[Lawler and Limic(2010)]{LawlerLimic}
Gregory F. Lawler and Vlada Limic,
\emph{Random walk: a modern introduction},
Cambridge University Press, Cambridge, 2010;
\href{https://doi.org/10.1017/CBO9780511750854}{doi:10.1017/CBO9780511750854}.

\bibitem[Levine, Murugan, Peres and Ugurcan(2016)]{LMPU}
Lionel Levine, Mathav Murugan, Yuval Peres, and Baris Evren Ugurcan,
\emph{The divisible sandpile at critical density},
Ann. Henri Poincar\'e \textbf{17} (2016), 1677--1711;
\href{https://doi.org/10.1007/s00023-015-0433-x}{doi:10.1007/s00023-015-0433-x};
\href{https://arxiv.org/abs/1501.07258}{arXiv:1501.07258}.


\bibitem[Cipriani, Hazra and Ruszel(2018a)]{CHR}
Alessandra Cipriani, Rajat Subhra Hazra, and Wioletta M. Ruszel,
\emph{Scaling limit of the odometer in divisible sandpiles},
Probab. Theory Related Fields \textbf{172} (2018), 829--868;
\href{https://doi.org/10.1007/s00440-017-0821-x}{doi:10.1007/s00440-017-0821-x};
\href{https://arxiv.org/abs/1604.03754}{arXiv:1604.03754}.

\bibitem[Cipriani, Hazra and Ruszel(2018b)]{CHR2}
Alessandra Cipriani, Rajat Subhra Hazra, and Wioletta M. Ruszel,
\emph{The divisible sandpile with heavy-tailed variables},
Stochastic Process. Appl. \textbf{128} (2018), 3054--3081;
\href{https://doi.org/10.1016/j.spa.2017.10.013}{doi:10.1016/j.spa.2017.10.013};
\href{https://arxiv.org/abs/1610.09863}{arXiv:1610.09863}.

\bibitem[Contat(2022)]{Contat}
Alice Contat,
\emph{Sharpness of the phase transition for parking on random trees},
Random Structures Algorithms \textbf{61} (2022), 84--100;
\href{https://doi.org/10.1002/rsa.21061}{doi:10.1002/rsa.21061};
\href{https://arxiv.org/abs/2012.00607}{arXiv:2012.00607}.

\bibitem[Curien and H\'enard(2022)]{CH}
Nicolas Curien and Olivier H\'enard,
\emph{The phase transition for parking on Galton-Watson trees},
Discrete Anal. (2022), paper no.~1;
\href{https://doi.org/10.19086/da.33167}{doi:10.19086/da.33167};
\href{https://arxiv.org/abs/1912.06012}{arXiv:1912.06012}.

\bibitem[Goldschmidt and Przykucki(2019)]{GP}
Christina Goldschmidt and Micha{\l} Przykucki,
\emph{Parking on a random tree},
Combin. Probab. Comput. \textbf{28} (2019), 23--45;
\href{https://doi.org/10.1017/S0963548318000457}{doi:10.1017/S0963548318000457}.

\bibitem[Kaufman and Meisel(2025)]{KM}
Harley Kaufman and Josh Meisel,
\emph{Asymptotic behavior of the critical density of activated random walk},
\href{https://arxiv.org/abs/2512.00720}{arXiv:2512.00720}.

\bibitem[Levine and Peres(2010)]{LP10}
Lionel Levine and Yuval Peres,
\emph{Scaling limits for internal aggregation models with multiple sources},
J. Anal. Math. \textbf{111} (2010), 151--219;
\href{https://doi.org/10.1007/s11854-010-0015-2}{doi:10.1007/s11854-010-0015-2};
\href{https://arxiv.org/abs/0712.3378}{arXiv:0712.3378}.

\bibitem[Levine and Silvestri(2024)]{LS}
Lionel Levine and Vittoria Silvestri,
\emph{Universality conjectures for activated random walk},
Probab. Surv. \textbf{21} (2024), 1--27;
\href{https://doi.org/10.1214/24-PS25}{doi:10.1214/24-PS25};
\href{https://arxiv.org/abs/2306.01698}{arXiv:2306.01698}.

\bibitem[Arratia(1983)]{Arratia}
Richard Arratia,
\emph{Site recurrence for annihilating random walks on $\Z_d$},
Ann. Probab. \textbf{11} (1983), 706--713;
\href{https://doi.org/10.1214/aop/1176993515}{doi:10.1214/aop/1176993515}.

\bibitem[Balagurov and Vaks(1974)]{BV}
B. Ya. Balagurov and V. G. Vaks,
\emph{Random walks of a particle on lattices with traps},
Soviet Phys. JETP \textbf{38} (1974), 968;
Russian original Zh. Eksp. Teor. Fiz. \textbf{65} (1973), 1939.

\bibitem[Bramson and Lebowitz(2001)]{BL01}
Maury Bramson and Joel L. Lebowitz,
\emph{Spatial structure in low dimensions for diffusion limited two-particle reactions},
Ann. Appl. Probab. \textbf{11} (2001), 121--181;
\href{https://doi.org/10.1214/aoap/998926989}{doi:10.1214/aoap/998926989}.

\bibitem[Kang and Redner(1984)]{KR}
Kyungsik Kang and Sidney Redner,
\emph{Scaling approach for the kinetics of recombination processes},
Phys. Rev. Lett. \textbf{52} (1984), 955--958;
\href{https://doi.org/10.1103/PhysRevLett.52.955}{doi:10.1103/PhysRevLett.52.955}.

\bibitem[Ovchinnikov and Zeldovich(1978)]{OZ}
Alexander A. Ovchinnikov and Yakov B. Zeldovich,
\emph{Role of density fluctuations in bimolecular reaction kinetics},
Chem. Phys. \textbf{28} (1978), 215--218;
\href{https://doi.org/10.1016/0301-0104(78)85052-6}{doi:10.1016/0301-0104(78)85052-6}.

\bibitem[Toussaint and Wilczek(1983)]{TW}
Doug Toussaint and Frank Wilczek,
\emph{Particle-antiparticle annihilation in diffusive motion},
J. Chem. Phys. \textbf{78} (1983), 2642--2647;
\href{https://doi.org/10.1063/1.445022}{doi:10.1063/1.445022}.

\bibitem[Levine and Peres(2009)]{LevinePeres09}
Lionel Levine and Yuval Peres,
\emph{Strong spherical asymptotics for rotor-router aggregation and the divisible sandpile},
Potential Anal. \textbf{30} (2009), no.~1, 1--27;
\href{https://doi.org/10.1007/s11118-008-9104-6}{doi:10.1007/s11118-008-9104-6};
\href{https://arxiv.org/abs/0704.0688}{arXiv:0704.0688}.

\bibitem[Peskir and Shiryaev(2006)]{PeskirShiryaev}
Goran Peskir and Albert N. Shiryaev,
\emph{Optimal stopping and free-boundary problems},
Lectures in Mathematics ETH Z\"urich, Birkh\"auser, Basel, 2006;
\href{https://doi.org/10.1007/978-3-7643-7390-0}{doi:10.1007/978-3-7643-7390-0}.

\bibitem[Pinelis(1994)]{Pinelis}
Iosif Pinelis,
\emph{Optimum bounds for the distributions of martingales in Banach spaces},
Ann. Probab. \textbf{22} (1994), no.~4, 1679--1706;
\href{https://doi.org/10.1214/aop/1176988477}{doi:10.1214/aop/1176988477}.

\bibitem[Przykucki, Roberts and Scott(2023)]{PRS}
Micha\l{} Przykucki, Alexander Roberts, and Alex Scott,
\emph{Parking on the integers},
Ann. Appl. Probab. \textbf{33} (2023), no.~2, 1076--1101;
\href{https://doi.org/10.1214/22-AAP1836}{doi:10.1214/22-AAP1836};
\href{https://arxiv.org/abs/1907.09437}{arXiv:1907.09437}.

\bibitem[Rolla(2020)]{Rolla}
Leonardo T. Rolla,
\emph{Activated random walks on $\Z^d$},
Probab. Surv. \textbf{17} (2020), 478--544;
\href{https://doi.org/10.1214/19-PS339}{doi:10.1214/19-PS339};
\href{https://arxiv.org/abs/1906.05037}{arXiv:1906.05037}.

\end{thebibliography}

\end{document}